\documentclass[12pt,a4paper]{amsart}

\usepackage[latin1]{inputenc}
\usepackage[english]{babel}
\usepackage{geometry}
\usepackage{array}
\usepackage{enumerate}
\usepackage{amsmath, amsfonts, amssymb, amsthm}
\usepackage{xypic}
\usepackage{mathrsfs}
\usepackage{hyperref}
\usepackage{aliascnt}
\usepackage[all]{xy}

\usepackage[usenames,dvipsnames]{color}

\newcommand{\M}{\mathcal{M}}
\newcommand{\oM}{\overline{\mathcal{M}}}

\newcommand{\oOmega}{\overline{\Omega}}

\newcommand{\A}{\mathcal{A}}
\renewcommand{\O}{\mathcal{O}}

\renewcommand{\H}{\mathcal{H}}

\newcommand{\cC}{\mathcal{C}}
\renewcommand{\O}{\mathcal{O}}
\newcommand{\R}{\mathcal{R}}
\newcommand{\T}{\mathcal{T}}
\newcommand{\RR}{\mathcal{R}}

\renewcommand{\P}{\mathbb{P}}
\newcommand{\C}{\mathbb{C}}
\newcommand{\Z}{\mathbb{Z}}
\newcommand{\N}{\mathbb{N}}

\newcommand{\cZ}{\mathcal{Z}}

\DeclareMathOperator{\Res}{Res}

\DeclareMathOperator{\Aut}{Aut}
\DeclareMathOperator{\End}{End}

\DeclareMathOperator{\Tr}{Tr}

\DeclareMathOperator{\Id}{Id}

\theoremstyle{plain}
\newtheorem{theorem}{Theorem}[section]

\newtheorem{corollary}[theorem]{Corollary}
\newtheorem{conjecture}[theorem]{Conjecture}
\newtheorem{proposition}[theorem]{Proposition}
\newtheorem{lemma}[theorem]{Lemma}
\newtheorem{assumption}[theorem]{Assumption}

\theoremstyle{definition}
\newtheorem{definition}[theorem]{Definition}
\newtheorem{question}[theorem]{Question}
\newtheorem{example}[theorem]{Example}
\newtheorem{remark}[theorem]{Remark}

\newtheorem{problem}[theorem]{Problem}
\newtheorem{questions}[theorem]{Questions}



\begin{document}
\author{Sebastiano Carpi}
	\address{Dipartimento di Matematica, Universit\`a di Roma ``Tor Vergata'', Via della Ricerca Scientifica, 1, 00133 Roma, Italy\\
		E-mail: {\tt carpi@mat.uniroma2.it}
	}
\author{Giulio Codogni}
	\address{Dipartimento di Matematica, Universit\`a di Roma ``Tor Vergata'', Via della Ricerca Scientifica, 1, 00133 Roma, Italy\\
		E-mail: {\tt giulio.codogni@uniroma2.it}
	}

\title
[VOAs, partition functions and Teichm\"{u}ller modular forms]
{Vertex operator algebras, partition functions and Teichm\"{u}ller modular forms}

\begin{abstract}
In the spirit of the geometric approach to two-dimensional conformal field theory, we explicitly associate to every holomorphic vertex operator algebra a section of a power of Hodge line bundle on the moduli space of curves of arbitrary genus - or equivalently a Teichm\"{u}ller modular form in any genus. 
As a first application, we connect the geometry of the moduli space of curves to the classification program for holomorphic vertex algebras.
We then discuss how to use the theory of holomorphic vertex algebras to reach new results about the moduli space of curves.
In the last part of the paper we study how the above mentioned forms can be used to reconstruct the vertex algebra.

\end{abstract}
\maketitle
\tableofcontents

 \section{Introduction}

In this paper, we study some aspects of the connection between the moduli space of Riemann surfaces and two-dimensional conformal field theories (CFTs), trying to reach new results on both sides.

Our starting point is an indication given almost forty years ago in a remarkable paper by Friedan and Shenker \cite{FS}. Here, the authors, inspired by string theory, consider as a fundamental object the partition function $\cZ(\Sigma)$ of a modular invariant two-dimensional CFT  on any closed Riemann surface $\Sigma$.  The ``function" $\cZ(\Sigma)$ is  formally defined through 
a Euclidean path integral 
 \begin{equation*}
\mathcal{Z}(\Sigma) =  \int_{\mathcal{F}_\Sigma} e^{- S_\Sigma (\phi) } \mathcal{D}\phi
\end{equation*}
where $\mathcal{F}_\Sigma$ is a suitable space of (generalized) functions $\phi$ on $\Sigma$ and $S_\Sigma (\phi) \colon \mathcal{F}_\Sigma \to \mathbb{R}$   
is the Euclidean action of the CFT on $\Sigma$. Due to conformal covariance, one may argue that $\cZ(\Sigma)$ depends only on the genus of $\Sigma$, the equivalence class of its complex structure and some appropriate additional structure. This suggests to consider the genus $g$ partition function $\cZ_g$, that is the restriction of the map $\Sigma \mapsto  \cZ(\Sigma) \in \mathbb{C}$  to the genus $g$ Riemann surfaces this should correspond to  a section of a suitable vector bundle on the compactified moduli space of Riemann surfaces $\oM_g$. 

The Friedan-Shenker picture is closely related to the celebrated Segal axiomatization of two-dimensional CFT \cite{Segal}. Although the latter is heuristically based on path integrals, the formulation of the corresponding axioms does not make use of these integrals. The Friedan-Shenker picture should emerge by restricting the Segal functor to closed Riemann surfaces. The difficulty of defining either path integrals or the Segal functor in a mathematically rigorous way for concrete CFT models is proverbial, see e.g. \cite{GKRV21,GKRV24,Huang03,Ten17,Ten19a,Ten19b}. The theory of vertex operator algebras (VOAs) \cite{Bor86,VA,Axiomatic,FLM88,DMS97,Kac01,Gan06} gives a mathematical axiomatization of chiral two-dimensional conformal field theories (which are special two-dimensional CFTs generated by fields depending on one light-ray coordinate only), and manageable language to extract from CFTs various mathematical results and stimulating indications for future research. 

 Heuristic arguments on the degeneration limit of the partition function  $\cZ(\Sigma)$ when $\Sigma$ approaches a singular complex curve with nodes give formal series involving the correlation functions of the conformal quantum fields, see \cite{FS,GV09}. For a vertex operator algebra $V$, of CFT type and self-dual  - whose definition we recall in Section \ref{S:vert} - one can take these formal series as a mathematically unambiguous definition of the genus $g$ partition functions $\cZ_{V,g}$. This approach is taken, e.g., in  \cite{GV09,GKV10,MT10,MT14,TWgeq26}. In Section \ref{S:part_func} we define $\cZ_{V,g}$ as formal power series. 
 
 For a chiral CFTs, the modular invariance requirement, necessary in the work of Friedan and Shenker, is satisfied by the requirement that the corresponding VOA is holomorphic, see  Section \ref{S:vert}.  In this case the Friedan-Shenker picture suggests that the genus $g$ partition function $\cZ_g$ should give a section of a power of the Hodge line bundle on $\oM_g$.  Equivalently, $\cZ_g$ gives a Teichm\"{u}ller modular forms (= holomorphic functions on the Teichm\"{u}ller space $T_g$ which are modular invariant with respect to the action of the mapping class group; full details are given in Section \ref{S:scho}).

 Our first set of results, namely Theorems \ref{T:main1} and \ref{thm:altre_espansioni}, and Proposition \ref{P:ampl}, confirms the Friedan-Shenker picture and, as we are going to explain, helps to prove new results both about the moduli space of curves and about vertex operator algebras. We summarize these results with the following statement.

 \begin{theorem}\label{thm:intro1} 
 
 The formal series $\cZ_{V,g}\cZ_{M(1),g}^{-c}$, where $V$ is any holomorphic vertex operator algebra of central charge $c$ and $M(1)$ is the Heisenberg vertex algebra, converges on a suitable open set to a Teichm\"{u}ller modular form of degree $g$ and weight $c/2$.

There exist:
\begin{itemize}
    \item a trivialization $\phi_0$ of the Hodge line bundle $\lambda_g$ near the boundary divisor $\delta_0$ of $\oM_g$;
    \item for every holomorphic vertex operator algebra $V$ of central charge $c$, a global section $u(\overline{\Omega}_{V,g})$ of $\lambda_g^{\otimes c/2}$ on $\oM_g$;
\end{itemize} 
such that $\phi_0(u(\oOmega_{V,g}))=\cZ_{V,g}$.

 \end{theorem}

The proof of Theorem \ref{thm:intro1} relies on conformal blocks (Section \ref{SecConformalBlocks}) and Virasoro uniformization (Section \ref{S:Ati}). We also have a variant of the second part of Theorem \ref{thm:intro1} for the other boundary divisors $\delta_i$.

We refrain from properly explaining here all the terms appearing in Theorem \ref{thm:intro1}, because we will do that in the rest of this paper. We rather use this introduction to fully explain the genus one case. We assume some familiarity with the language of modular forms as in \cite{Serre}.

When $g=1$, the Teichm\"{u}ller space is the Siegel upper half plane $\H$; the mapping class group is $SL(2,\Z)$; the quotient $\H/SL(2,\Z)$ is the moduli space of elliptic curves $\M_{1,1}$ (in genus one, we have to add marked point, so we look at $\M_{1,1}$ rather than $\M_1$). An intermediate quotient is the punctured disk $\Delta^*$, with its corresponding variable $\mu = e^{2\pi i \tau}$, $\tau \in \H$. It is also the so called Schottky space $S_1$, whose generalization $S_g$ to the higher genus case is particularly relevant for our paper, and will be properly discussed in Section \ref{SS:schottky}. Teichm\"{u}ller modular forms are just the classical modular forms, and they give sections of the Hodge bundle on $\oM_{1,1}$.

Vertex operator algebras of CFT type  are graded vector spaces $$V=\oplus_{n\in \N}V_n\,,$$ with $\dim V_0=1$ and $\dim V_n < \infty$, called spaces of states, endowed with a ``quantized" operation
$$
Y\colon V \to \End(V)[[z,z^{-1}]]
$$
called state-to-field correspondence. There are a number of axioms discussed in Section \ref{S:vert}.

The partition function in genus one ignores the operator $Y$, and takes into account only the space of states. Following Remark \ref{RemarkGenus1Partition}, it is the graded dimension
$$\cZ_{V,1}(\tau)=\sum_{n \in \N} \dim V_ne^{2\pi i n\tau}$$

A remarkable result by Zhu in \cite{Zhu} states that if  $V$ is a holomorphic VOA with central charge $c$ (necessarily a non-negative integer multiple of 8), the series $\cZ_{V,1}(\tau)$ converges on upper half plane  $\H$, and $\cZ_{V,1}(\tau) \cZ_{M(1),1}(\tau)^{-c}$ is a modular form of weight $c/2$, where  $\cZ_{M(1),1}(\tau)$ is the genus one partition function of the (rank one) Heisenberg vertex algebras $M(1)$. Our approach also gives a different proof of this result.

As the ring of classical modular forms (or, equivalently, the moduli space $\oM_{1,1}$) is very well understood, this result gives an important tool to classify vertex operator algebras.  

A paradigmatic example is the moonshine VOA $V^\natural$, let us discuss it in detail. It is a holomorphic VOA of central charge $24$ constructed by Frenkel, Lepowsky and Meurman \cite{FLM88}. Its automorphism group is the Monster $\mathbb{M}$, the largest sporadic finite simple group. One has
$$
\cZ_{V^\natural,1}(\tau)=e^{2\pi i \tau}(j(\tau) -744)=1 + 1968884 \,e^{4\pi i \tau}+ \cdots\,,
$$
where $j(\tau)$ is the Klein $j$-function.

This gives an interpretation of the ``McKay equation" $1968884 = 1968883 +1$, which relates the coefficient $1968884$ in the expansion of $j(\tau)$ with the   smallest dimension $1968883$ of the non-trivial representations of the Monster $\mathbb{M}$. The relation between the expansion of the Klein $j$-function and the dimension of the representation of the Monster group leads to the so called monstrous moonshine conjecture of Conway and Norton. The existence of the VOA $V^{\natural}$ turned out to be the reason for this moonshine.

Frenkel, Lepowsky and Meurman conjectured that the moonshine VOA $V^\natural$ is the unique holomorphic VOA of central charge 24 and $\dim V_1=0$, \cite[page xxxiii]{FLM88}. This is a very important conjecture in the theory of vertex operator algebras which is still open, see Section \ref{sec: c=24}.

Observe that if $V$ is another such VOA, then $\cZ_{M(1),1}(\tau)^{-24}(\cZ_{V^{\natural},1}(\tau)-\cZ_{V,1}(\tau))$ is a modular form of weight 12 whose $e^{2\pi i\tau}$ expansion (Fourier series) starts with a term of order two in $e^{2\pi i \tau}$. It is well known that such a form is identically zero \cite[Theorem 4, Section 3.2]{Serre}, so the graded dimension of $V$ is equal to the graded dimension of $V^{\natural}$. 

For $g>1$, the partition function $\cZ_{V,g}$ captures information about the state-field correspondence $Y$, so it is natural to try to generalize the above approach using our Theorem \ref{thm:intro1}.

In a nutshell, \emph{the geometric invariant of $\oM_g$ that governs the classification (of partition functions) of holomorphic VOA is the slope $s_g$ of the effective cone of divisors of $\oM_g$}. It plays the role of ``a modular form of weight 12 whose expansion starts with a term of order two is identically zero" in the genus one case. We will review its definition in Section \ref{SecSlope}.

The sequence of numbers $\{s_g\}_{g\in \N}$ has been widely studied over the last decades, for reasons that have nothing to do with VOAs. There are many results and conjectures about it, again reviewed in Section \ref{SecSlope}. To make a long story short, we have the following corollary of Theorem \ref{thm:intro1}.
\begin{corollary}[= Corollary \ref{cor:classification2}]\label{cor_class_intro}
If $s_g>6$ for every $g$, then any holomorphic VOA $V$ of central charge 24 and $\dim V_1=0$ has the same partition function of the monster VOA $V^{\natural}$ for every $g$.
\end{corollary}
The validity of the assumption on the slope $s_g$ is presently open, and we call it the ``weak Harris-Morrison slope conjecture" \ref{conj:slope}, which is again discussed in Section \ref{SecSlope}. This is the input we need from the geometry of $\oM_g$ to study (partition functions of)  holomorphic VOAs.

\medskip

Other applications of our Theorem \ref{thm:intro1} are about the space of Teichm\"{u}ller modular forms (Section \ref{SecLinearIndependence}) and the Schottky problem (Section \ref{S:schottky_problem}).

\medskip

From a mathematical perspective, especially in view of Corollary \ref{cor_class_intro}, it is natural to ask if the partition function is a complete invariant of the vertex algebra, i.e. if one can reconstruct the vertex algebra out of the knowledge of $\cZ_{V,g}$ for every $g$. From the physics perspective we are asking if one can recover the Segal functor and hence the CFT from the knowledge of the partition functions as argued in \cite{FS}. In this regard Segal wrote  ``Friedan has conjectured that a field theory $U$ is completely determined by  its restriction to closed surfaces.\dots  This seems plausible but I don't know the proof", see \cite[page 457]{Segal}. Let us spell out this problem as a conjecture.
\begin{conjecture}[Reconstruction Conjecture]
\label{Conj:reconstruction}
Let $U$, $V$ be two self-dual VOAs of CFT type. If $\cZ_{U,g} = \cZ_{V,g}$ for every $g$, then $U$ and $V$ are isomorphic. 
\end{conjecture}
We put self-dual and CFT type as assumptions because it is the most general set-up where we can define unambiguously the partition function as a formal series. However, it is possible that Conjecture \ref{Conj:reconstruction} holds under some additional assumption on at least one of the two VOAs, such as holomorphicity or unitarity. Let us stress that there are VOAs whose partition functions are different only when $g$ is big enough, e.g. the partition functions of the VOAs associated to the lattices $D_{16}^+$ and $E_8^{\oplus 2}$ are different if and only if $g\geq 5$, see Remark \ref{rem:part_latt}.

Sections \ref{SecPartitionFunctionSubalgebra}, \ref{SectionExamples} and \ref{sec: c=24} are devoted to the study of Conjecture \ref{Conj:reconstruction}. Our results are mainly bounded to the special but important unitary case, as explained in Section \ref{sec: c=24}.

We define a vertex subalgebra $PV \subset V$, which we call \emph{the partition function subalgebra}, which can be extracted from the partition function $\cZ_V$. This algebra is contained in $V^{\operatorname{Aut}(V)}$.  If $V$ is unitary, $PV$ is a unitary subalgebra containing the conformal vector of $V$, see Lemma \ref{L:unitary}. 

\begin{theorem}[= Theorem \ref{T:main2_corpo} ] If $U$ and $V$ are two simple unitary vertex operator algebras such that $\cZ_{U,g}=\cZ_{V,g}$ for every $g$, then $PU$ is isomorphic to $PV$, and $U$ is isomorphic to $V$ as $PU$ (equivalently $PV$)- module.
\end{theorem}
 As a consequence, if the $PV$-module $V$ has, up to isomorphisms, a unique unitary VOA structure, then $V$ is determined by the partition function as a unitary VOA, see Corollary \ref{cor: unique structure}.

 We discuss various explicit examples of partition function subalgebras of unitary VOAs and, in various cases, we are able to prove that $PV = V^{\operatorname{Aut}(V)}$ and/or that the unitary VOA structure of $V$ is determined by $\cZ_V$. We conjecture that the equality $PV = V^{\operatorname{Aut}(V)}$ always implies that the $PV$-module $V$ has a unique unitary VOA structure so that $V$ is determined by $\cZ_V$.  We can prove, thanks to the results by Codogni and Shephered-Barron in \cite{CS-B14},  that the unitary VOA structure of a holomorphic lattice VOA $V_L$ is completely determined by its partition function $\cZ_{V_L}$, see Theorem \ref{thmLatticePartition}. Moreover, using previous results by Gaberdiel and Volpato \cite{GV09}, we also prove that the same result holds for all holomorphic VOAs with central charge $c=24$ and weight-one subspace $V_1 \neq \{0\}$ (the Schellekens VOAs), see Theorem \ref{th: Schellekens}. 
 
 We conclude our paper with a conjecture on the partition function subalgebra $PV^\natural$ of the moonshine VOA, see Conjecture \ref{conj: monster orbifold}. The conjecture says that $PV^\natural$ is equal to the monster orbifold 
${V^\natural}^{\mathbb{M}}$. We prove that this conjecture, together with our weak Harris-Morrison slope conjecture \ref{conj:slope}, and the conjectured strong rationality of the monster orbifold would prove the Frenkel-Lepowsky-Meurmann conjecture on the uniqueness of $V^\natural$.

\subsection*{Organization of the paper}

We have tried to be reasonably self-contained in order to make the content of this article understandable to all readers interested in the relations between the geometry of moduli spaces of complex curves and conformal field theory. In particular we hope that the paper is
accessible both for the community working on moduli space of curves, and for the community working on vertex operator algebras. An extended introduction to the topics of this paper can also be found in the survey \cite{CC geq 26}.

In Section \ref{S:vert} we discuss various preliminaries on VOAs, and their modules. This section consists of previously known material with the possible exception of  Proposition \ref{PropTensorLattice}. 

In Section \ref{S:part_func} we define and briefly discuss the partition function of a CFT type  self-dual VOA as a formal series and also introduce some related power series that we will need later.

Section \ref{S:scho} contains preliminary materials on moduli spaces and modular forms. In particular, we give an overview of some basic facts on families on nodal curves, moduli stacks of complex curves, Teichm\"{u}ller and Siegel modular forms, and the Schottky space.  

In Section \ref{SecConformalBlocks} we discuss various preliminary materials on the bundles of conformal blocks and covaqua of holomorphic VOAs. In Section \ref{S:Ati} we discuss and prove some facts related to the  Atiyah algebras and the Virasoro uniformization which play an essential role in our results.  

Section \ref{S:proof1} is mostly dedicated to our  Theorem \ref{T:main1} on the connection between of holomorphic VOAs and Teichm\"{u}ller modular forms. Section  \ref{SecBoundary} is devoted to the proof of Theorem \ref{thm:altre_espansioni}.

In Section \ref{SecSlope} we introduce the slope of the effective cone of $\oM_g$, and we prove Corollary \ref{cor_class_intro}. In Section \ref{SecLinearIndependence} we prove the linear independence of the partition functions and use this result to compare Teichm\"{u}ller and Siegel modular forms. The section also contains some constructions, problems, and conjectures of independent interest. In Section \ref{S:schottky_problem} we discuss some relations of our results with the Schottky problem. 

In Section  \ref{SecPartitionFunctionSubalgebra} we introduce the partition subalgebra of a VOA and discuss some of its properties. In particular, 
we prove Theorem \ref{T:main2_corpo} which completely characterizes the pairs of unitary VOAs having the same partition function. In Section \ref{SectionExamples} we discuss various examples of VOAs for which we can determine interesting properties of their partition subalgebra. In particular, we prove the previously mentioned Theorem \ref{thmLatticePartition} on the partition function of holomorphic lattice VOAs. Section \ref{sec: c=24} is dedicated to holomorphic VOAs with $c=24$. We discuss the Shellekens VOAs and the uniqueness of the moonshine VOA in connection to the monster orbifold.

\subsection*{Acknowledgments}
We thank Enrico Arbarello, Jethro van
Ekeren, Gabi Farkas, Tiziano 
Gaudio, Angela Gibney, Sam Grushevsky, Bin Gui, Reimundo Heluani, Marco Matone, Riccardo Salvati Manni,
Nick Shepherd-Barron, Yoh Tanimoto, Luca Tasin, Michael Tuite, Filippo Viviani and Roberto Volpato for useful dicussions and comments about the topics of this paper.

\subsection*{Funding statement} 

Both authors acknowledge the MUR Excellence Department Project MatMod@TOV awarded to the Department of Mathematics, University of Rome Tor Vergata, CUP E83C23000330006. S.C. is partially supported by the GNAMPA group of INdAM.G.C. is partially supported by the GNSAGA group of INdAM.

\section{Vertex operator algebras and their modules}\label{S:vert}

In this section, we discuss some preliminaries about vertex operator algebras (VOAs), unitary vertex operator algebras, and their representation theories that we will use in the rest of the paper. For more details, we refer the readers to  \cite{VA,DK06,Axiomatic,Kac01,FLM88,LL04}, see also \cite{Gan06}. We mainly follow the notation and terminology in \cite{CKLW18}.

\subsection{Vertex operator algebras}

In this paper, every vertex algebra  will be over the field $\mathbb{C}$ of complex numbers. Let $V$ be a conformal vertex algebra in the sense of \cite[Section 4.10]{Kac01},  with vacuum vector $\Omega^V \in V$ and conformal vector $\nu^V \in V$. Accordingly,  $V$ is a vertex algebra as defined \cite[Section 1.3]{Kac01} with even (bosonic) part $V_{\bar{0}}$ and odd (fermionic) part $V_{\bar{1}}$. 

When the underlying conformal vertex algebra 
is clear from the context we will simply write $\Omega$ and $\nu$ instead of $\Omega^V$ and $\nu^V$ respectively. We will always assume that $\Omega \neq 0$.

For every $a \in V$ the vertex operator
\begin{equation}
Y(a,z) = \sum_{n \in \mathbb{Z}}a_{(n)}z^{-n-1}
\end{equation} 
is a {\bf (quantum) field} on $V$, i.e., \ a formal Laurent series with coefficients in the algebra $\operatorname{End}(V)$ of vector space endomorphisms of $V$such that, for every $b\in V$,  $a_{(n)}b=0$ eventually as $n\to +\infty$. For every $a \in V$ we have $a_{(n)} \Omega = 0$ for all $n \geq 0 $  and    $a_{(-1)}\Omega = a$. Moreover, the map $a \to Y(a,z)$ (the {\bf state field correspondence}) is linear and $Y(\Omega,z) = 1_V$ where $1_V:V \to V$ denotes the identity linear endomorphism of the vector space $V$. 

The vertex operator $Y(\nu,z)$ corresponding to the conformal vector $\nu$ is called the {\bf energy-momentum field} 
of $V$ and it is often simply denoted by $T(z)$.  If $L_n : =\nu_{(n +1)}$ then 
\begin{equation} 
Y(\nu,z)= \sum_{n \in \mathbb{Z}} L_n z^{-n-2} 
\end{equation}
and the endomorphisms $L_n$, $n\in \mathbb{Z}$, give a representation of the Virasoro algebra on $V$ with  
central charge $c\in \mathbb{C}$ (the central charge of $V$), that is,

\begin{equation}
\label{EqVirasoro}
{[} L_n,L_m{]} = (n-m) L_{n+m} +\frac{c}{12}(n^3-n)\delta_{n,-m} 1_V
\end{equation}
for all $n,m\in\mathbb{Z}$. Moreover, $[L_{-1},a_{(n)}] = -n a_{(n-1)}$ for all $a \in V$ and all $n\in \mathbb{Z}$. 
Equivalently, 
\begin{equation}
[L_{-1},Y(a,z)] = \frac{d}{dz} Y(a,z)
\end{equation}
for all $a \in V$.  

The operator $L_0$ (the {\bf conformal Hamiltonian}) is diagonalizable on $V$, that is  
\begin{equation}
V = \bigoplus_{\alpha \in \mathbb{C}} V_{\alpha}
\end{equation}
where $V_\alpha := \mathrm{ker}(L_0 -\alpha 1_V)$ (the {\bf weight $\alpha$ subspace}). 
Sometimes we will write $L^V_n$ for $\nu^V_{(n +1)}$ if we want to specify $V$. 
\medskip

A {\bf vertex operator algebra} (VOA) $V$ conformal vertex algebra with $V_{\bar{1}} = \{0\}$ (i.e. $V$ is purely even) such that $V_\alpha = \{0\}$ if 
$\alpha \notin \mathbb{Z}$. Moreover, $V_n = \{0\}$ eventually as $n \to -\infty$ and $V_n$ is finite dimensional for all $n \in \mathbb{Z}$.
Accordingly, we have 
\begin{equation}
V = \bigoplus_{n \in \mathbb{Z}} V_n 
\end{equation}  
where so that $V$ has a natural $\mathbb{Z}$-grading given by the conformal Hamiltonian $L_0$. Moreover, $\Omega \in V_0$ and  $\nu \in V_2$.
A vertex operator algebra $V$ is said to be of {\bf CFT type} if  $V_n=\{0\}$ for $n < 0$ and $V_0 = \mathbb{C}\Omega$. 

For a vertex operator algebra $V$ we consider the  restricted dual
\begin{equation}
V^\vee := 
\bigoplus V_n^\vee 
\end{equation} 
where, for any integer $n$, $V_n^\vee $ the dual of the finite dimensional vector space $V_n$.  

Clearly $V^\vee$ is a subspace of the algebraic dual $\prod_{n \in \mathbb{Z}}V_n^\vee$ and this gives a
non-degenerate bilinear pairing 
$\langle \cdot , \cdot \rangle : V^\vee \times V \to \mathbb{C}$ defined by $\langle \varphi , a \rangle = \varphi(a)$ for all pairs $(\varphi,a)$ in 
 $V^\vee \times V$. 

\medskip

In various occasions, we will make use of the {\bf Borcherds identity}  \cite{Bor86},
that is, the equality
\begin{eqnarray}
\label{B-id}
\nonumber
\sum_{j=0}^{+\infty}
\binom{m}{j}
\left(a_{(n+j)}b\right)_{(m+k-j)}c =
\sum_{j=0}^{+\infty}(-1)^j
\binom{n}{j}
a_{(m+n-j)}b_{(k+j)}c \\ -
\sum_{j=0}^{+\infty}(-1)^{j+n}
\binom{n}{j}
b_{(n+k-j)}a_{(m+j)}c \, ,\;\;
\;a,b,c\in V, \,m,n,k\in \mathbb{Z} \, .
\end{eqnarray}
It can be taken as one of the axioms of vertex operator algebras or as a consequence of other axioms including the locality axiom, see \cite[Sect.\! 4.8]{Kac01}. 

We also mention {\bf skewsymmetry}: 
\begin{equation}
\label{skewsymmetry_equation}
a_{(n)}b = - \sum_{j=0}^{+\infty} \frac{(-1)^{j+n}}{j!}(L_{-1})^j b_{(n+j)}a  
\end{equation}
for all $a,b \in V$ and all $n \in \mathbb{Z}$, see \cite[Section 4.2]{Kac01}.

An important consequence of the Borcherds identity and of skewsymmetry is the fact that the wight one subspace $V_1$ of a vertex operator algebra $V$ is a Lie algebra with brackets
$[a,b]:=a_{(0)}b,$ $a,b \in V_1$.

The endomorphisms $a_n \in \operatorname{End}(V)$, $a\in V$, $n\in \mathbb{Z}$, are defined by 
\begin{equation}
Y(z^{L_0}a , z) = \sum_{n \in \mathbb{Z}} a_{n}z^{-n}.
\end{equation} 

A vector $a \in V_d$ is said to be homogeneous of conformal dimension or conformal weight $d$.  If $a$ is homogeneous of conformal dimension $d$ 
then $a_n = a_{(n + d -1)}$ and $a_{(n)} = a_{n+1-d}$ for all $n\in \mathbb{Z}$. In particular, $a_{-d} = a_{(-1)}$ so that $a = a_{-d}\Omega$. Moreover, 

\begin{equation}
Y(a,z) = \sum_{n \in \mathbb{Z}}a_{n}z^{-n-d_a} \,.
\end{equation}

We have 
\begin{equation}
a_m V_n \subset V_{n-m}
\end{equation}
for all $a \in V$ and all $m,n \in \mathbb{Z}$. In particular $a_{-n}  \in V_n$ for all $n \in \mathbb{Z}$.

The homogeneous vectors $a \in V$ and the corresponding field $Y(a,z)$ 
are called {\bf quasi-primary} if $L_1a=0$ and {\bf primary} 
if $L_na=0$ for every integer $n>0$. The vacuum vector $\Omega$ is a primary vector in $V_0$ and the conformal vector $\nu$ is a quasi-primary vector in $V_2$.

We have the following commutation relations:
\begin{equation}
\label{EqQuasi-Primary/PrimaryCommutation}
[L_{m}, a_n] = \left( (d -1)m- n \right) a_{m+n},  
\end{equation}
for all primary (resp. quasi-primary) $a \in V_d$, for all $n\in \mathbb{Z}$ and all $m \in \mathbb{Z}$ (resp. $m\in \{-1,0,1\}$), see, e.g., \cite[Cor.4.10]{Kac01}. 
\medskip

A {\bf VOA isomorphism} $\Phi: V \to U$ between two VOAs $V$ and $U$ is a complex vector space isomorphism from $V$ onto $U$ such that 
$\Phi \nu^V = \nu^U$ and $\Phi a_{(n)}b = (\Phi a)_{(n)}\Phi b$ for all $a, b \in V$ and all $n \in \mathbb{Z}$. If $\Phi$ 
satisfies all the above conditions with exception of $\Phi \nu^V = \nu^U$ then $\Phi$ is called a {\bf vertex algebra isomorphism}. 
Note that if  $\Phi: V \to U$  is a vertex algebra isomorphism we have 
\begin{eqnarray*}
\Phi \Omega^V &=& (\Phi \Omega^V)_{(-1)}\Omega^U  = \Phi \Omega^V_{(-1)} \Phi^{-1}\Omega^U \\
&=& \Phi \Phi^{-1}\Omega^U = \Omega^U   \,. 
\end{eqnarray*}

Similarly, an {\bf anti-linear VOA isomorphism} $\Phi: V \to U$ between two VOAs $V$ and $U$ is an anti-linear vector space isomorphism from $V$ onto $U$ such that $\Phi\Omega^V = \Omega^U$,  $\Phi \nu^V = \nu^U$ and $\Phi a_{(n)}b = (\Phi a)_{(n)}\Phi b$ for all $a, b \in V$ and all $n \in \mathbb{Z}$. A VOA-isomorphism or anti-linear VOA-isomorphism $\Phi: V \to U$ preserves the $\mathbb{Z}$-grading that is $\Phi V_n = U_n$ for all $n \in \mathbb{Z}$. 

A {\bf VOA automorphism} (or simply an automorphism) of a vertex operator algebra $V$ is a VOA-isomorphism $g:V\to V$. 
We denote by  $\operatorname{Aut}(V)$ the group of 
VOA automorphisms of $V$. Note that typically  $\operatorname{Aut}(V)$ is a proper subgroup of the group of vector spaces automorphisms of $V$.   Similarly, a {\bf VOA anti-linear automorphism} (or simply an anti-linear automorphism) of a vertex operator algebra $V$ is an anti-linear VOA-isomorphism $g:V\to V$. Note that if $g$ is a VOA-automorphism or anti-linear automorphism of $V$ then $g$ commutes with every$L_n$, 
$n \in \mathbb{Z}$.  

The map 

\begin{equation*}
\operatorname{Aut}(V) \ni g \mapsto (g \restriction_{V_n})_{n \in \mathbb{Z}}
\end{equation*}
gives a group isomorphism of $\operatorname{Aut(V)}$ onto a closed subgroup of the direct product $\prod_{n \in \mathbb{Z}}\operatorname{GL}(V_n)$
which makes $\operatorname{Aut}(V)$ into a Hausdorff topological group. 

\medskip

 By {\bf vertex subalgebra} $U$ of a vertex operator algebra $V$ is a subspace $U \subset V$ such that $\Omega \in U$ and $a_{(n)}b \in U$ for all $a, b \in U$ and all  $n \in \mathbb{Z}$. A vertex subalgebra $U$ of $V$ is always a vertex algebra, but in general not a vertex operator algebra nor a conformal vertex algebra. On the other hand, if the conformal vector $\nu$ of $V$ belongs to a conformal subalgebra $U \subset V$, namely $U$ is a {\bf conformal subalgebra} of $V$, then $U$ is a vertex operator algebra and $U_n = V_n \cap U$ for all $n\in \mathbb{Z}$ and one says that $V$ is a {\bf conformal VOA extension} of $U$.
 
 In this case, it is easy to see that in $V$ is of CFT type then also $U$ is CFT type. 
 
If $G$ is a subgroup of $\operatorname{Aut}(V)$ let $V^G$ be subset of $G$-invariant vectors of $V$, that is $a \in  V^G$ iff $ga=a$ for all $g \in G$. Then $V^G$ is a conformal subalgebra of $V$ called {\bf orbifold subalgebra}. Clearly, if $\overline{G}$ is the closure of $G$ in $\operatorname{Aut}(V)$ then $V^{\overline{G}} = V^G$.   The smallest orbifold subalgebra of $V$ is $V^{\operatorname{Aut{V}}}$.

For a subset $\mathscr{F} \subset V$,  there is a smallest vertex subalgebra $V_\mathscr{F}$ of $V$ containing $\mathscr{F}$ and we say that $V_\mathscr{F}$ is the vertex subalgebra generated by $\mathscr{F}$.  If $\mathscr{F} = \{\nu \}$ then $V_{\mathscr{F}}$ is the {\bf the Virasoro subalgebra}  of $V$. Note that the smallest orbifold subalgebra $V^{\operatorname{Aut{V}}}$ always contains the Virasoro subalgebra. 
$V$ is said to be finitely generated if there is a finite subset of $\mathscr{F} \subset V$ with $V_{\mathscr F} = V$.  
   
\medskip

For a vertex operator algebra $V$ we define the {\bf duality operator} $S:V \to V$ by $S(a) := e^{L_1} (-1)^{L_0}a$. 
If $a \in V_d$ is a homogeneous vector of conformal weight $d$ then 
\begin{equation*}
S(a)= (-1)^d \sum_{k=0}^{d+ n_V} L_1^k a
\end{equation*} 
where $n_V \in \mathbb{Z}_{\geq0}$ is the smallest non-negative integer such that $V_n = \{0\}$ for all integers $n$ with $n< - v_V$. Note that $n_V =0$ if $V$ is of CFT type. 

The endomorphism $S$ is always an involution, that is, $S^2 = 1_V$. Moreover, if $a$ is quasi-primary of dimension $d$ then, by Equation (\ref{eqSqp}),
\begin{equation} 
\label{eqSqp}
S(a) = (-1)^d a \,. 
\end{equation}

In particular $S(\Omega)=\Omega$ and $S\nu=\nu$. 
 
An {\bf invariant bilinear form} on a vertex operator algebra $V$ is a bilinear form $(\cdot,\cdot) : V \times V \to \mathbb{C}$ such that 
\begin{equation}
(b,a_nc ) = ((Sa)_{-n}b, c)
\end{equation}
for all $a,b,c \in V$ and all $n \in \mathbb{Z}$.  
In particular, if $a \in V_d$ is quasi-primary, then 

\begin{equation}
\label{eqInvBilqp}
(b,a_nc ) = (-1)^d(a_{-n}b, c)
\end{equation}
for all $b,c \in V$ and all $n \in \mathbb{Z}$.

Taking $a=\nu$ and $n=0$ we find $(L_0b,c)= (b,L_0c)$ for all $b,c \in V$ and this implies that $V_n$ and $V_m$ are orthogonal subspaces with respect to an invariant bilinear form  $(\cdot,\cdot)$ whenever $n \neq m$. An invariant bilinear form on a vertex operator algebra $V$ is necessarily symmetric. 

We say that $(\cdot,\cdot)$ is {\bf normalized} if  $(\Omega,\Omega)=1$.
If $V$ is of CFT type then,a non-zero invariant bilinear form exists iff $L_1 V_1 = \{0\}$ and, if this is the case, it is unique up to a non-zero multiplicative constant. 
Consequently, if $V$ is of CFT type and $L_1V_1=0$ there exists a unique normalized invariant bilinear form, which is determined by 
\begin{equation}
(a,b) \Omega = S(a)_n b  
\end{equation}
for $a,b \in V_n$, $n\in \mathbb{Z}$. As a consequence, if $V$ is of CFT type and $(\cdot, \cdot)$ is an invariant bilinear form on $V$ then 
$(ga,gb)= (a,b)$ for all $g \in \operatorname{Aut}(V)$ and all $a,b \in V$, that is, $(\cdot, \cdot)$ is $\operatorname{Aut}(V)$-invariant.

\subsection{Vertex algebra modules and representation theory}\label{sec:rep_theory_VOA}

Let $V$ be a VOA with central charge $c$. A {\bf vertex algebra module} (or weak module) for $V$ is a complex vector space $M$ together with  a linear map 

\begin{equation}
V \ni a \mapsto  Y^M(a,z) = \sum_{n \in \mathbb{Z}}a^M_{(n)}z^{-n-1}
\end{equation} 
where, for any $a \in V$, $Y^M(a,z)$ is a (quantum) field on $M$, i.e.\ a formal Laurent series with coefficients in the algebra $\operatorname{End}(M)$ of vector spaces endomorphisms of $M$ such that, for every $c \in M$,  $a^M_{(n)}b=0$ eventually as $n\to +\infty$. Moreover, the map 
$V \ni a \mapsto  Y^M(a,z)$  is such that $Y^M(\Omega,z) = 1_M$ and the Borcherds identity hols on $M$ that is, 

\begin{eqnarray}
\label{B-id_Mod}
\nonumber
\sum_{j=0}^{+\infty}
\binom{m}{j}
\left(a^M_{(n+j)}b^M\right)_{(m+k-j)}c =
\sum_{j=0}^{+\infty}(-1)^j
\binom{n}{j}
a^M_{(m+n-j)}b^M_{(k+j)}c \\ -
\sum_{j=0}^{+\infty}(-1)^{j+n}
\binom{n}{j}
b^M_{(n+k-j)}a^M_{(m+j)}c \, ,\;\;
\;a,b\in V,\; c \in M,\; \,m,n,k\in \mathbb{Z} \, .
\end{eqnarray}
As in the vertex operator algebra case, we define the endomorphisms $a^M_n$, $n \in \mathbb{Z}$ by 
\begin{equation}
Y^M(z^{L_0}a , z) = \sum_{n \in \mathbb{Z}} a^M_{n}z^{-n}.
\end{equation} 

The Borcherds identity in Equation (\ref{B-id_Mod}) implies that the coefficients $L^M_n := \nu^M_n$, $n \in \mathbb{Z}$ of the formal series 
\begin{equation}
Y^M(\nu,z) = \sum_{n \in \mathbb{Z}}L^M_n z^{-n-2}  
\end{equation}
define a representation of the Virasoro algebra with central charge $c$ on $M$. On the other hand, the conformal Hamiltonian $L^M_0$ need not to be diagonalizable on $M$ and, in fact, it could have no eigenvectors at all.

A {\bf VOA module} (or strong module) for $V$ is a vertex algebra module $M$ of $V$ that satisfies the following properties: 
\begin{itemize}
\item[$(i)$]  $M_\alpha := \operatorname{ker}(L^M_0 -\alpha 1_M)$ is finite dimensional for all $\alpha \in \mathbb{C}$

\item[$(ii)$] $M = \bigoplus_{\alpha \in \mathbb{C}} M_\alpha$ 

\item[$(iii)$] For any $\alpha \in \mathbb{C}$,  $M_{\alpha + n} = \{ 0 \}$ eventually as $n \to -\infty$, $n \in \mathbb{Z}$. 

\end{itemize} 

Sometimes we will call a VOA-module for a vertex operator algebra $V$ simply a {\bf $V$-module}. If $U$ is a conformal VOA extension of $V$ then $U$ is also a $V$-module. In particular $V$ is itself a $V$-module called the {\bf adjoint module}.

The {\bf character} $\chi_M(\tau)$ of $V$-module $M$ is the formal series (in the variable $e^{i2\pi i \tau}$)  
$$\Tr_M e^{ 2\pi i \tau (L^M_0 -\frac{c}{24}1_M) } \, ,$$ 
where $c$ is the central charge of $V$. In general, this should be considered as a formal series in the parameter  
$e^{2 \pi i \tau}$.  In various relevant cases, the character $\chi_M(\tau)$ converges to a holomorphic function 
(denoted again $\chi_M(\tau)$) in the upper half-plane $\Im\tau >0$. We will call the character $\chi_V(\tau)$ of the adjoint module the {\bf vacuum character.}

If $M$, $N$ are $V$-modules, then a {\bf $V$-module homomorphism} (or module map) $\Phi: M \to N$  is a linear map such that $\Phi a^M_{(n)} = a^N_{(n)}\Phi$ for all $a \in V$ and all $n \in \mathbb{Z}$. A $V$-module homomorphism $\Phi:M\to N$  which is injective and surjective is called a 
{\bf $V$-module isomorphism}. If there exists a $V$-module isomorphism $\Phi: M \to N$ we say that $M$ and $N$ are isomorphic or equivalent $V$-modules. 
A submodule $N \subset M$ is a subspace of $M$ that is invariant for the action of all endomorphisms $a^M_{(n)}$. $a \in V$, $n \in \mathbb{Z}$. Clearly, a submodule $N \subset M$ inherits from $M$ the structure of a $V$-module. A $V$-module $M$ is said to be simple or irreducible if 
$M\neq \{0\}$ and if its submodules are only  $\{0\}$ and $M$.   A vertex operator algebra is said to be {\bf simple} if its adjoint module is irreducible.
\medskip

The restricted dual $V^\vee$ also carries the structure of a $V$-module called the {\bf contragredient module}. The corresponding fields are
\begin{equation}
Y^{V^\vee}(a,z) = \sum_{n \in \mathbb{Z}} a^{V^\vee}_{(n)}z^{-n},
\end{equation} 
$a \in V$, are defined by
\begin{equation} 
\langle a^{V^\vee}_n \varphi, b \rangle := \langle \varphi, (Sa)_{-n} b \rangle \;\; \varphi \in V^\vee,\; b \in V,\; n \in \mathbb{Z}, 
\end{equation}
where, as before, $S= e^{L_1}(-1)^{L_0}$.

If $\Phi: V \to V^\vee$ is a $V$-module homomorphism then $(a,b):= \langle\Phi a, b\rangle$, $a,b \in V$ is an invariant bilinear form on 
$V$ and every invariant bilinear form arises in this way. Moreover $(\cdot,\cdot)$ is non-degenerate iff $\Phi$ is a $V$-module isomorphism. 
Accordingly, a non-degenerate invariant bilinear form on $V$ exists iff $V$ is a {\bf self-dual VOA}, that is  $V$ and  $V^\vee $ are isomorphic $V$-modules.  If $V$ is a simple VOA then every non-zero invariant bilinear form on $V$ is non-degenerate. Conversely, if $V$ is of CFT type and has a non-degenerate invariant bilinear form, then $V$ is simple, see \cite[Proposition 4.6]{CKLW18}. Moreover, if $V$ is simple, then $V$ has at most one normalized invariant bilinear form; see again  \cite[Proposition 4.6]{CKLW18}.  As a consequence, if $V$ is simple and $(\cdot,\cdot)$ is a normalized invariant bilinear form on $V$ then $(\cdot,\cdot)$ is $\operatorname{Aut(V)}$-invariant, that is $(ga,gb)=(a,b)$ for all 
$a,b \in V$ and all $g \in \operatorname{Aut}(V)$. 
\medskip

A vertex operator algebra $V$ is said to be {\bf regular} if every vertex algebra module for $V$ is a direct sum of simple $V$-modules. If $V$ is a regular VOA then it thas only finitely many equivalence classes of simple $V$-modules, see \cite{DLM98}. Examples of regular VOAs include the lattice VOAs,
$V_L$ with $L$ a positive-definite even lattice, the simple affine VOAs $V_k(\mathfrak{g})$ with $\mathfrak{g}$ a simple complex Lie algebra and positive integer level $k$, and the moonshine vertex operator algebra $V^\natural$, see \cite{DLM97}. 

Let $V$ be a VOA and let $C_2(V)$ be the linear span of the vectors of the form $a_{(-2)}b$, $a,b \in V$. $V$ is said to be {\bf $C_2$-cofinite} if the vector space quotient $V / C_2(V)$ is finite dimensional. By \cite{Li99} every regular VOA is $C_2$-cofinite. 

A vertex operator algebra $V$ is said to be {\bf strongly rational} if it is CFT type, self-dual and regular. In this case it is also simple because CFT type self-dual VOAs are simple. The above examples of regular VOAs are 
all strongly rational. A very important result in the theory of vertex operator algebra is that if $V$ is a strongly rational VOA then the category 
$\operatorname{Rep}(V)$ of $V$-modules has a natural tensor structure, making it into a modular fusion category \cite{Hua08}, see also \cite{Gan06}. 

In this article a central role is played by {\bf holomorphic VOAs}. We say that vertex operator algebra $V$ is holomorphic if it is strongly rational and 
every simple $V$-module $M$ is equivalent to the adjoint-module $V$. Equivalently, $V$ is holomorphic if it is strongly rational and 
$\operatorname{Rep}(V)$ is equivalent, as a braided tensor category, to the  fusion category of finite dimensional complex vector spaces. Note that sometimes a weaker definition of holomorphic VOA is used in the literature, see e.g. \cite{DM04b}. The central charge of a holomorphic VOA must be a non-negative multiple of $8$, see \cite{DM04a}. The case $c=0$ corresponds to the trivial holomorphic vertex operator algebra $V = \mathbb{C}$. Non-trivial examples of holomorphic VOAs are given by the lattice VOAs $V_L$ with $L$ an even, positive definite unimodular lattice. Another important example is the moonshine VOA $V^\natural$ whose central charge is $24$.  Further examples with $c=24$ are given by the Schellekens VOAs discussed in Section \ref{SectionExamples}.  
\medskip

We end this subsection with the following proposition that we will need in Section \ref{SecLinearIndependence}.

\begin{proposition}
\label{PropTensorLattice}
Let $U$ and $V$ be two non-trivial VOAs. Then $U\otimes V$ is a lattice VOA if and only if $U$ and $V$ are both lattice VOAs. 
\end{proposition}

\begin{proof} For the if part, note that if $L_1$ and $L_2$ are two even positive definite lattices and $U$, $V$ are isomorphic to $V_{L_1}$ and $V_{L_2}$, respectively, then $L_1 \times L_2$ is again an even positive definite lattice and $U\otimes V$ is isomorphic to $V_{L_1 \times L_2}$. 

Let us now come to the only if part and assume that $U\otimes V$ is a lattice VOA. We first prove that $V$ is regular. To this end, let us recall that a lattice VOA is regular. Let $M$ be a vertex algebra module for $V$. Then $U \otimes M$ is a vertex algebra module for $U \otimes V$. Since the latter VOA is a lattice VOA and hence regular we have a direct sum decomposition 
$$U \otimes M = \bigoplus_j U \otimes M^j $$ where each $M^j$ is a simple $V$-module. It follows that 
$$M = \bigoplus_j  M^j . $$
Hence, since $M$ was arbitrarily arbitrary, $V$ is regular. Similarly, $U$ is also regular. Note that since $U \otimes V$ is simple and CFT type, both $U$ and $V$ are simple and CFT type. Moreover, it follows from the fact that $U \otimes V$ self-dual that both $U$ and $V$ are self-dual, because the normalized invariant bilinear form of $U \otimes V$ restricts to normalized invariant bilinear forms on $U$ and on $V$. Therefore, $U$ and $V$ are strongly rational. It follows from \cite[Theorem 1.1 ]{DM04b} that the weight one subspaces $U_1$ , $V_1$ and 
$(U \otimes V)_1 = U_1\otimes \Omega^V \oplus \Omega^U \otimes V_1$ are reductive Lie algebras. Since $U \otimes V$ is a lattice VOA, the central charge $c^{U\otimes V}$ of $U \otimes V$ is equal to the rank $r((U \otimes V)_1)$ of the Lie algebra $(U \otimes V)_1$, cf. \cite[Theorem 1.3]{DM04b} which is equal to the sum 
$r(U_1) + r(V_1)$ of the ranks of $U_1$ and $V_1$. Moreover, we have $c^{U\otimes V} = c^U + c^V$ where $c^U$ and $c^V$ are the central charges of $U$ and $V$, respectively. 

Let us now prove that $U$ and $V$ are Lattice VOA. Since $U\otimes V$ is a lattice VOA then, in every simple module $W$ 
we have $W_\alpha := \operatorname{ker}(L^W_0 -\alpha 1_W)$ so that $W_\alpha =\{0\}$ if $\alpha$ is not a non-negative real number. Now if 
$M$ is a simple $V$ module then $W:= U \otimes M$ is a simple $U\otimes V$-module and it follows that    
$$M = \bigoplus_{\alpha \in \mathbb{R}_{\geq 0}} M_\alpha $$ 
where $M_\alpha := \operatorname{ker}(L^M_0 -\alpha 1_M)$. Hence, the effective central charge $\tilde{c}^V$ of $V$ is equal to $c^V$, see \cite[Section 1]{DM04b} for the definition. Similarly, the effective central charge  $\tilde{c}^U$ of $U$ is equal to $c^U$. By \cite[Theorem 1.2]{DM04b} $c^U \geq r(U_1)$ and $c^V \geq r(V_1)$ . 
Accordingly,   
$$c^{U\otimes V} = r(U_1) + r(V_1) \leq c(U_1) + c(V_1) = c^{U\otimes V}$$
so that  $c^U = r(U_1)$ and $c^V = r(V_1)$ and the conclusion follows from \cite[Theorem 3.1]{DM04b}.

\end{proof}

Some preliminaries on unitary VOAs are given in \ref{Subsec:unitaryVOAs}.

\section{Definition of partition functions}\label{S:part_func}

In this section we fix a self-dual vertex operator algebra $V$ of CFT type. In particular $V$ is simple. This means that it has a unique normalized invariant bilinear form $( \, , \, )$. 

For any non-negative integer $g$, we introduce its genus $g$ partition functions $\cZ_{V,i}$ for $i\in \{0,1,\cdots, \lfloor \frac{g}{2}\rfloor\}$. Here, we introduce them only as formal power series, we will later intreprete them as functions, and study their properties. $\cZ_{V,g,0}$ will depend on $3g$ formal variables; for $i\neq 0$,  $\cZ_{V,g,i}$ will depend on $3g+1$. 

First, we introduce the Casimir bilocal fields. To this end we note that, for any non-negative integer $k$, the restriction of  $( \, , \, )$ to $V_k$ is non-degenerate. As a consequence, we can find a basis $v^{(i)}$, $i=1,\dots \dim V_k$ for $V_k$ that is orthonormal to
 $( \, , \, )$, that is, $(v^{(i)},v^{(j)}) = \delta_{i,j}$. 
  
\begin{definition}[Casimir bilocal field]\label{def:cas_field}
Given a non-negative natural number $k$, and two formal variable $w$ and $z$, the {\bf $k$-th Casimir bilocal field} is
$$
\gamma_k(w,z)=\sum_{i=1}^{\dim V_a}Y(v^{(i)},w)Y(v^{(i)},z)  $$
where $v^{(i)}$ is a orthonormal basis of $V_k$ with respect to the normalized invariant bilinear $( \, , \, )$. Note thatIf $V_k$ is trivial, we set the $k$-th Casimir bilocal field to be zero.
\end{definition}

The definition of Casimir bilocal fields does not depend on the choice of the basis. Now comes the main definitions.

\begin{definition}[Parititon functions]\label{def:parition_functions} With the above notations, for $g\geq 1$ we define  the {\bf genus $g$ partition function} of $V$ as

 \begin{align*}
 \cZ_{V,g}(w_1,\dots , w_g,z_1,\dots , z_g,q_1,\dots ,q_g):= \\ \sum_{(n_1, \dots , n_g) \in \mathbb{N}^g}  
 \left(\Omega^V,\gamma_{n_1}(w_1,z_1)\cdots \gamma_{n_g}(w_g,z_g) \Omega^V\right) q_1^{n_1}\cdots q_g^{n_g}
 \end{align*}
We also let $\cZ_{V,0}:=1$.

The {\bf partition function}  $\cZ_{V}$ of $V$ is the sequence $\{ \cZ_{V,g} \}_{g \in \mathbb{Z}_{\geq 0}}$.

\medskip

For $g\geq 2$ and $1\leq i\leq \lfloor \frac{g}{2}\rfloor $, we define the {\bf $i$-the genus $g$ partition function} $\cZ_{V,g,i}$  of $V$ as

 \begin{align*}
\cZ_{V,g,i}(w_1,\dots , w_g,w,z_1,\dots , z_g,z,q_1,\dots ,q_g,q):=\\
\sum_{ k \in \N}\sum_{v^{(j)} \in B_k}\left(\sum_{(n_1, \dots , n_i) \in \mathbb{N}^i}  \left(\Omega^V, \gamma_{n_1}(w_1,z_1)\cdots \gamma_{n_i}(w_i,z_i) Y(v^{(j)},w)\Omega^V \right) q_1^{n_1}\cdots q_i^{n_i}\right) \cdot \\ \left(\sum_{(n_{i+1}, \dots , n_{g}) \in \mathbb{N}^{g-i}}  \left(\Omega^V,Y(v^{(j)},z)\gamma_{n_{i+1}}(w_{i+1},z_{i+1})\cdots \gamma_{n_{g}}(w_{g},z_{g}) \Omega^V\right) q_{i+1}^{n_{i+1}}\cdots q_{g}^{n_{g}}\right)q^k
\end{align*}
where $B_k$ is a orthonormal basis of $V_k$; if $V_k$ is trivial, the corresponding addendum is zero.  

The {\bf $i$-th partition function}  $\cZ_{V,i}$ of $V$ is the sequence $( \cZ_{V,g,i} )_{g \in \mathbb{Z}_{\geq 2}}$.

To have uniform notations, we some time call the partition function the $0$-th partition function, and thus let $\cZ_{V,g,0}=\cZ_{V,g}$.
\end{definition}

 \begin{remark}[Genus one partition function]
\label{RemarkGenus1Partition}
 The genus 1 partition function $\mathcal{Z}_{V,1}$ has been computed in \cite[Thm. 4.2]{MT14}. The result is 
 $ \mathcal{Z}_{V,1}(w,z,q) = \mathrm{Tr}_{V} \mu^{L^V_0}$, for 
 $\left|\frac{q}{(w-z)^2} < \frac14 \right|$,
 where 
 $$\mu = \frac{(w-z)^2}{2q} \left(\sqrt{1 + \frac{4q}{(w-z)^2-1}}-1\right)-1$$ 
 is a solution of the equation  $\frac{\mu}{(1+\mu)^2} = -\frac{q}{(w-z)^2} \,. $
 As a consequence, if $U$ and $V$ are VOAs with $\mathcal{Z}_{U,1} = \mathcal{Z}_{V,1}$ then we have  
 $\Tr_U e^{2\pi  i \tau L^U_0} =  \Tr_V e^{2\pi  i \tau L^V_0} \,, $ in the sense of formal series in the variable 
 $e^{2\pi i \tau}$. In particular, if $U$ and $V$ also have the same central charge, then their vacuum characters $\chi_U(\tau)$ and $\chi_V(\tau)$ defined in Section \ref{sec:rep_theory_VOA} coincide. 
  \end{remark}

Let us also recall the following classical definition.
\begin{definition}[Correlation functions]\label{d:correlation_function}
Given an integer $n \geq 1$, and $n$ vectors $a_1,\dots , a_n$ in $V$, their ($n$-points) {\bf correlation function (on the sphere)}  is the formal power series 
$$
(\Omega^V,Y(a_1,z_1)\cdots Y(a_n,z_n)\Omega^V)
$$
\end{definition}

The coefficients of the $q_i$'s in the definition of the partition functions from Definition \ref{def:parition_functions} are sum of correlation functions. Locality implies the following classical result about correlation functions, see \cite[page 265]{Zhu} and \cite[Proposition 3.5.1]{Axiomatic}.

\begin{theorem}\label{T:commutation}
The correlation functions converge in the domain $|z_1|>|z_2|>\cdots > |z_n|$, $z_1,\dots,z_k \in \mathbb{C}$; they can be extended  to rational functions of $z_1,\dots , z_n$ with poles just along the diagonals and the coordinate axes. 
For any permutation $\sigma$ of $\{1,\dots , n\}$, we have the equality of rational functions
$$(\Omega^V,Y(a_1,z_1)\cdots Y(a_n,z_n)\Omega^V)=(\Omega^V,Y(a_{\sigma(1)},z_{\sigma(1)})\cdots Y(a_{\sigma(n)},z_{\sigma(n)})\Omega^V) .$$
\end{theorem}

\begin{remark}
\label{remarkCommutativityRationality} It is important to note that equality 
$$(\Omega^V,Y(a_1,z_1)\cdots Y(a_n,z_n)\Omega^V)=(\Omega^V,Y(a_{\sigma(1)},z_{\sigma(1)})\cdots Y(a_{\sigma(n)},z_{\sigma(n)})\Omega^V) $$ is an equality of rational functions and not an equality of formal series. The two corresponding Laurent series are obtained by expanding the same rational function in two different regions. In the paper, we will use the same symbol to denote the correlation function as a formal power series or as a rational function and add, if necessary, the missing information to avoid possible confusions. 
\end{remark}

\begin{remark} 
\label{Lurent vs Power series}
By Theorem \ref{T:commutation} the coefficient of the  formal Laurent series 
$$\left(\Omega^V,\gamma_{n_1}(w_1,z_1)\cdots \gamma_{n_g}(w_g,z_g) \Omega^V\right)$$ 
which appear in the definition of the genus $g$-partition functions converge to rational functions, and hence we can also look at the partition function  
$$\cZ_{V,g}(w_1,\dots , w_g,z_1,\dots , z_g,q_1,\dots ,q_g)$$ as to a formal power series in the variables $(q_1,\dots ,q_g)$ 
whose coefficient are rational functions in the variables  $(w_1,\dots , w_g,z_1,\dots , z_g)$.
\end{remark}

\begin{remark}[Product of partition functions]
\label{remarkProductFormal} 
Let us recall that the product of formal power series in the same variables always make sense, but this is not the case for the product of formal Laurent series, see e.g. \cite[Section 2.1]{Kac01}. However, the product of two  $n$-points correlation functions in the same variables $z_1,\dots,z_n$, possibly corresponding to two different VOAs, always makes sense as a formal Laurent series because of the convergence part of Theorem \ref{T:commutation}. Moreover, their product converges in the domain $|z_1|>|z_2|>\cdots > |z_n|$, $z_1,\dots,z_k \in \mathbb{C}$ to the product of the two corresponding rational functions.

As a consequence, if $U$ and $V$ are two VOAs of CFT type having a non degenerate invariant bilinear form then, for any $g\in \Z_{\geq 0}$ the 
product $\cZ_{U,g} \cZ_{V,g}$ gives a well defined formal series and moreover,  $\cZ_{U,g} \cZ_{V,g}= \cZ_{U\otimes V,g} $.

\end{remark}

\section{Moduli spaces and modular forms}\label{S:scho}

\subsection{Family of nodal curves}\label{SS:family}
A Riemann surface, also known as smooth curve, is a compact complex manifold of complex dimension one. In this paper we consider or also nodal curves, i.e. one dimensional analytic spaces (= singular manifolds) which are either smooth or with nodal singularities. We use the euclidean topology.

A family of nodal curves is a surjective morphism $\pi\colon \cC \to B$, where $\cC$ is a reduced analytic space,  $B$ is a smooth complex manifold and the fibers of $\pi$ are nodal curves. In other words, for every point $b$ of $B$, the fiber $\cC_b=\pi^{-1}(b)$ is a nodal curve that varies holomorphically with $b$. The coordinates on $B$ are sometimes called moduli or the parameters of the family.

We often consider families with $n$ sections $s_i$, where $n$ is some positive integer and the $s_i$'s are maps $s_i\colon B \to \cC$ such that $\pi \circ s_i$ is the identity on $B$, the images of the $s_i$ are all disjoint and do not intersect the singular points of the fibers. Informally, the family is a curve with $n$ marked points which change in a holomorphic way. If $B$ is just a point, then $\cC$ is a nodal curve $C$ with $n$ distinct smooth marked points $p_1,\dots, p_n$.

Let $\pi\colon \cC\to B$ be a family of nodal curves over a smooth base with $n$ sections $s_1,\dots , s_n$. We say that $\pi$ has the \textbf{property (Q)} if for every $b$ in $B$, after having removed the nodes, every connected component of the fiber $\cC_b$ intersects at least one section. An equivalent formulation that we will often use is that the complement of the (image of the) sections is affine.

\subsection{Moduli space of Riemann surfaces, Hodge line bundle, boundary divisor, gluing and forgetful maps}\label{S:moduli}

A general reference for this section is the book \cite{ACG}. The moduli space of smooth genus $g$ Riemann surfaces with $n$ marked points and $2g+n-2>0$ is a Deligne-Mumford stack denoted by $\M_{g,n}$. Its compactification $\oM_{g,n}$ is the moduli space of genus $g$ nodal Riemann surfaces with $n$ marked points. (If $n=0$, we omit it. By abuse of notation, we let $\M_1:=\M_{1,1}$ and $\M_0:=\M_{0,3}$; we adopt the analog convention for all the moduli space we treat in this paper.)

If the reader is not familiar with stacks, can think at the moduli space as a manifold; general references about stacks with informal introductions are \cite{ACG,Alper}. By definition of moduli stack, it is the calassifying space for families of curves, i.e. there is one to one correspondence between families of marked nodal curves $\pi\colon \cC \to B$ and maps $m(\pi)\colon B \to \oM_{g,n}$. The map $m(\pi)$ is called the moduli map associated to $\pi$, and maps a point $b$ to the isomorphism class of $\pi^{-1}(b)$. (An analogue principle holds for every fine moduli  moduli space that we will encounter later, such as  the spaces $S_g$, $C_g$ and $T_g$. )

To every Deligne-Mumford stack one can associate a coarse moduli space, which is a (possibly singular) analytic space. In particular, the coarse space $\overline{M}_{g,n}$ of $\oM_{g,n}$ is a normal projective variety, the coarse space $M_{g,n}$ of $\M_{g,n}$ is an open dense subset of $\overline{M}_{g,n}$. We can still associate to every family a moduli map from $B$ to the coarse space, but it is no longer true that every map $m\colon B \to \overline{M}_{g,n}$ comes from a family of curves over $B$. In particular, the points of $\overline{M}_{g,n}$ are in bijective correspondence with isomorphism classes of marked nodal curves.

The complement of $\M_g$ in $\oM_g$ is a reducible divisor $\delta$, called the boundary divisor. The divisor $\delta$ is the union of irreducible divisors $\delta_i$ with $i \in \{0,\dots \lfloor \frac{g}{2}\rfloor \}$. The divisor $\delta_0$ parameterizes curves with at least one non-separating node (a node is non-separating if, once it has been removed, the curve remains connected). The divisor $\delta_i$, for $i \in \{1,\dots \lfloor \frac{g}{2}\rfloor \}$, parameterizes curves with a least one node that divides the curve in a component of genus $i$ and a component of genus $g-i$. 

On $\oM_{g,n}$, we can construct the line bundle $\lambda_g$, usually called Hodge line bundle or lambda class. Its fiber over a curve $(C,p_1,\dots, p_n)$ is the determinant of the ($g$ dimensional) vector space of holomorphic differential forms on $C$. In particular, the fiber does not depend neither on the number nor on the choice of the marked points $p_1, \dots , p_n$.

The forgetful morphism $F\colon \oM_{g,n+1}\to \oM_{g,n}$ is the map that forgets a marked point. There are several gluing morphisms, all of them glue two marked points to form a new curve with one new node and two less marked points: $G_0\colon \oM_{g-1,n+2}\to \oM_{g,n}$ and $G_{h,k}\colon \oM_{h,m+1}\times \oM_{k,n+1}\to \oM_{h+k,m+n}$ (by abuse of notations, we do not keep track of the number of marked points).

The Hodge line bundle behaves well with respect to this map, i.e. $F^*\lambda_g=\lambda_g$, $G_0^*\lambda_g=\lambda_g$, and $G^*_{h,k}\lambda_{h+k}=\lambda_h\boxtimes \lambda_k$.

\subsection{Schottky space}\label{SS:schottky}

We now recall, following \cite[Chapter 5]{Schottky} and \cite[Section 2.2]{TWgeq26}, the definition of the Schottky space $S_g$. When $g=1$, the space $S_1$ is isomorphic to the punctured open disk of radius one, see  \cite[Proposition 5.5]{Schottky}. We focus on the case $g\geq 2$.

The group of Mobi\"{u}s transformations $PSL(2,\C)$ acts on $\P^1$. Given a discrete subgroup $\Gamma$ of $PSL(2,\C)$, its region of discontinuity $\Omega(\Gamma)$ is the biggest open subset where it acts properly discontinuously. The group $\Gamma$ is called a genus $g$ Schottky group if $\Omega(\Gamma)/\Gamma$ is a Riemann surface $C_{\Gamma}$ of genus $g$. In particular, a genus $g$ Schottky group is a free group on $g$ generators such that $\Omega(\Gamma)$ is not empty. The description of $\Omega(\Gamma)$ is quite subtle when $g\geq 2$, and we do not need it in this paper 
A fundamental domain for the action of $\Gamma$ on $\Omega(\Gamma)$ can be obtained by removing  from $\P^1$ $2g$ simply connected closed domains that are bounded by Jordan curves. 

 A marked Schottky group is a Schottky group with the choice of $g$ ordered generators $(\gamma_1,\dots , \gamma_g)$. The marked Schottky space $C_g$ is the space parameterizing marked Schottky groups. Each $\gamma_i$ is hyperbolic, so it is unique determined by its attractive and repelling fixed points $W_i$ and $Z_i$, and its multiplier $\mu_i$ which is a complex number with $0<|\mu_i|<1$. The marked Schottky space can thus be realized as an open subset of $(\P^1)^{\times 2g}\times (\Delta^*_1)^{\times g}$, where $\Delta^*_1$ is the open pointed disk of radius one in the complex plane.

A family of curves with (marked) Schottky structure over a base $B$ is an open subset $U$ of $\P^1\times B$, and an action of the free group $G$ on $g$ (marked) elements such that the quotient $U/G$ is a family of genus $g$ Riemann surfaces  over $B$.  Observe that $C_g$ is a fine moduli space for curves with marked Schottky structure, so there is a universal curve with marked Schottky structure over $C_g$. Every family of curves with marked Schottky structure over a base $B$ can be obtained pulling back the universal curve via a map $B\to C_g$.

Let $C_g^{\circ}$ be the dense open subset of $C_g$ where none of the fixed points is $\infty$. On $C_g^{\circ}$, the fixed points and the multiplier give coordinates, and $C_g^{\circ}$ can be realized as an open subset of $\C^{3g}$. To the best of our knowledge, there is no explicit description of $C_g^{\circ}$ in $\C^{3g}$.

 There is another system of coordinates on $C_g^{\circ}$ that is less common in the literature, but more useful for our purposes. Recall that simple Mobi\"{u}s transformations are the translations $z\mapsto z+\alpha$, the inversion $z\mapsto z^{-1}$, and the homotheties $z\mapsto q z$. Every Mobi\"{u}s transformation $g$ in a Schottky group parameterized by $C_g^{\circ}$ can be uniquely written as the composition of a translation, the inversion a homothety, and another translation, i.e.,
\begin{equation}\label{dec_in_simples}
\gamma_i(z)=w_i+\frac{q_i}{z-z_i}\,,
\end{equation}
for convenient choices of $w_i$, $z_i$ and $q_i$; in other terms, $w_i$, $z_i$ and $q_i$ provide another natural set of coordinates on $C_g^{\circ}$. The formulae to go from one set to the other can be found, for example, in \cite[Section 2]{TWgeq26}:
\begin{equation}\label{e:coo_tw}
w_i=\frac{W_i-\mu_iZ_i}{1-\mu_i} \; , \;  z_i=\frac{Z_i-\mu_iW_i}{1-\mu_i} 
\; , \; q_i=-\frac{\mu_i(W_a-Z_a)^2}{(1-\mu_i)^2} 
\end{equation}
(in loc. cit., their $q$'s are our $\mu$'s, and our $q$'s are their $\rho$'s). They can be obtained observing that $w_i=\gamma_i(\infty)$, $z_i=\gamma^{-1}(\infty)$, and for $q_i$ imposing in Equation \ref{dec_in_simples} that $\gamma_i(w_i)=w_i$. 

As spoiler, let us say that using the coordinates $q_i$, $w_i$ and $z_i$  introduced in Equation \eqref{e:coo_tw}, the partition function $\cZ_{V,g,0}=\cZ_{V,g}$ from Definition \ref{def:parition_functions}, can be used to define a holomorphic function on a convenient open subset of $C_g^{\circ}$.

\medskip

The presentation of a Riemann surface $C$ as quotient $\Omega(\Gamma)/\Gamma$ is called a Schottky structure on $C$. If one also gives generators of $\Gamma$, this is called marked Schottky structures. Using the decomposition from Equation \ref{dec_in_simples}, the gluing relation $x\sim y \Leftrightarrow x= \gamma_i(y)$ takes the explicit form
\begin{equation}\label{plumbing}
x\sim y \, \Leftrightarrow \, (x-w_i)(y-z_i)=q_i \,. 
\end{equation}

Using the coordinates from Equation \eqref{e:coo_tw}, we can present a curve $C$ with a marked Schottky structure in the following explicit way. We consider a fundamental domain $U$ for $\Gamma$ obtained removing from $\P^1$ the interiors $D_i$ and $E_i$ of $2g$ simply connected closed domains, bounded by Jordan curves, containing $w_i$ and $z_i$, and such that $\gamma_i$ identifies the boundary of $D_i$ with the boundary of $E_i$. Then we glue these boundaries using  $\gamma_i$, i.e. according to the relation \eqref{plumbing}. The marked Schottky group is called \emph{classical} if one can take as $E_i$ and $D_i$ disks.  Let us stress that not all marked Schottky groups are classical, see \cite{non-class} and references therein. Let $C_{g,class}$ the space of classical marked Schottky groups. For $g\geq 2$, $C_{g,class}$ is strictly smaller than $C_g$, and we do not know if it is open or closed. In the following paragraphs, we introduce an open subset of $\C^{3g}$ is contained in $C_{g,class}^{\circ}:=C_{g,class}\cap C_g^{\circ}$.

\begin{definition}\label{def:open_in_Schottky}  
For every fixed $r\in \mathbb{R}_{>0}$, let $U_{g,r}$ be the open subset of $\C^{3g}$ defined as
$0<|q_i|<r^2$ for every $i$,  and  $|x-y|>2r $ for every $\{x,y\}\subset \{w_1,\dots, w_g,z_1,\dots , z_g\}$. 

Let $U_{g,r}^+$ the open subset of $U_{g,r}$ defined by the further condition $|w_1|>|z_1|>|w_2| > \cdots > |w_g|>|z_g|$.

\end{definition}
In other words, $U_{g,r}$ the open subset of $\C^{3g}$ such that: the disks in $\C$ of centers $w_i$ and $z_i$ and radii $r$ are all disjoint subsets of $\C$, and $|q_i|<r^2$.

\begin{proposition}\label{prop:open_in_Schottky}
 For every fixed $r\in \mathbb{R}_{>0}$, $U_{g,r}$ is contained in $C_{g,class}^{\circ}:=C_{g,class}\cap C_g^{\circ}$.
\end{proposition}
\begin{proof}
For every point of $U_{g,r}$ we have to construct the corresponding curve $C$ with the classical marked Schottky structure. Let $\gamma_i$ be the element of $PSL(2,\C)$ associated with $(w_i,z_i,q_i)$ by Equation \eqref{dec_in_simples}, and let $\Gamma$ be the group generated by the $\gamma_i$'s. Let $E_i$ (resp. $D_i$) be the disk in $\P^1$ of radius $\sqrt{|q_i|}$. These disks are all disjoint because of the definition of $U_{g,r}$, so $\Gamma$ is a Schottky group according to \cite[Definition 5.1]{Schottky}, and, as explained in loc. cit., a fundamental domain $U$ can be obtained by removing from $\P^1$ the disks $D_i$ and $E_i$'s. Since we are removing disks, the marked Schottky groups are classical.

\end{proof}

The Schottky space $S_g$ is the quotient of $C_g$ by the action of $PSL(2,\C)$ by conjugation. The open set $C_g^{\circ}$ surjects onto $S_g$. 

The quotient $C_g\to S_g$ has a section obtained by choosing a marking for each Schottky group such that the fixed points of $\gamma_1$ are $0$ and $\infty$, and one fixed point of $\gamma_2$ is one; hence $S_g$ is an open subset of $\C^{3g-3}$ ( we will seldom use this presentation as the $\gamma_i$ do not play a symmetric role).

\medskip

As explained in \cite[Propositions 5.2 and 5.6]{Schottky}, the association $\Gamma\mapsto C_{\Gamma}$ defines a surjective morphism $\pi\colon S_g\to \M_g$ with discrete infinite fibers.

\subsection{Theicm\"{u}ller modular forms}\label{SS:Tecih}

We follow mainly \cite{Schottky} and \cite{Ich}. We fix $g\geq 2$, see Remark \ref{rem:g=1} for the case $g=1$. Let $T_g$ be the Teichm\"{u}ller space of genus $g$ curves. This space parameterizes isomorphism classes of pairs $(C,f)$, where $C$ is a genus $g$ Riemann surfaces and $f$ is a diffeomerphism from $C$ to a fixed Riemann surface $X$; diffeomorphisms are considered up to homotopy; see e.g. \cite[Introduction]{Schottky}. 

The space $T_g$ has a natural projection $\rho \colon T_g\to \M_g$, which turns out to be the universal cover. The space $\M_g$ can be realized as stack quotient $T_g/\Gamma_g$, where $\Gamma_g$ is a discrete group called the mapping class group. It is possible to define an action of $\Gamma_g$ on the homology of $X$, and this gives surjective homorphism
$$\psi \colon \Gamma_g \to Sp(H_1(X,\Z))\cong Sp(2g,\Z)  \,.$$
(Recall that $H_1(X,\Z)$ is a free rank $2g$ module, and the intersection pairing gives a non-degenerate symplectic form, which is why $Sp(H_1(X,\Z))\cong Sp(2g,\Z)$.)

From now on, we fix a symplectic basis of the homology of $X$.  Via $f$ one gets a basis of the homology of $C$, and can then look at the period matrix\footnote{For the reader convenience let us recall the definition of the period matrix of $C$. Let $\{A_i,B_i\}$ be a symplectic basis of the homology. Let $\omega_i$ be a basis of the space of holomorphic one forms such that $\int_{A_i}\omega_j=\delta_{ij}$. Then the $(i,j)$ entry of the period matrix is $\int_{B_i}\omega_j$. This matrix is a $g$ by $g$, symmetric, and its imaginary part is positive definite. Torelli's theorem states that it possible to reconstruct the curve $C$ and the basis of the homology $\{A_i,B_i\}$ out of the period matrix $j(C,f)$.} $j(C,f)$ of $C$.

If $(C,f)$ and $(C',f')$ are equivalent for the action of an element $\gamma$ of $\Gamma_g$, then
$$
j(C,f)=(aj(C',f')+b)(cj(C',f')+d)^{-1} \quad \textrm{ where }\quad \begin{pmatrix}
a & b \\
c & d
\end{pmatrix}=\psi(\gamma)$$

The cocycle $\det(cj(C,f)+d)$ defines a linearization of the action of $\Gamma_g$ on the trivial line bundle on $T_g$. We denote by $L_g$ the trivial line bundle endowed with this linearization, and we call it the line bundle of weight $1$  Teichm\"{u}ller modular forms. In other words, $L_g$ the trivial line bundle on $T_g$ endowed with a non-trivial action of $\Gamma_g$ such that, for every $k\in \Z$, the invariant section of $L_g^{\otimes k}$ are the holomorphic functions $F$ on $T_g$ which satisfy the functional equation
\begin{equation}\label{func_teich}
F\left(\gamma((C,f)) \right)=\det(cj(C,d)+d)^kF(C,f)\,,
\end{equation}
for every $(C,f)$ in $T_g$ and every $\gamma$ in $\Gamma_g$. (For $k<0$, the only section is the zero function.)

The group of isomorphism classes of line bundles on $T_g$ with an action $\Gamma_g$ is isomorphic to $\Z$, and generated by the line bundle $L_g$.

\medskip

Let us now consider the pull-back $\rho^*\lambda_g$, where $\rho \colon T_g \to \M_g=T_g/\Gamma_g$ is the quotient map  and $\lambda_g$ is the Hodge line bundle from Section \ref{S:moduli}. There is canonically defined action of $\Gamma_g$ on $\rho^*\lambda_g$ such that the pull-backs of the sections of $\lambda_g$ are exactly the $\Gamma_g$-invariant sections of $\rho^*\lambda_g$. With this action, $\rho^*\lambda_g$ is isomorphic to the line bundle $L_g$ of weight one Teichm\"{u}ller modular forms as $\Gamma_g$-line bundle. In particular, there is an isomorphism  between $H^0(\M_g,\lambda^{\otimes k})$ and the vector space of holomorphic functions $F$ on $T_g$ satisfying Equation \eqref{func_teich} (see e.g. \cite[Proposition 1.4]{Ich}).

\medskip

The inclusion of $\M_g$ in $\oM_g$ gives an injective restriction map 
$$
H^0(\oM_g,\lambda^{\otimes k}_g)\to H^0(\M_g,\lambda^{\otimes k}_g) \,;
$$
when $g\geq 3$ this map is an isomorphism, otherwise it is not surjective\footnote{The surjectivity for $g\geq 3$ is well-known, it follows from the fact that $\lambda_g$ is semi-ample on the coarse space $\overline{M}_g$, and its Itaka map embeds $M_g$ as big open subset of a normal projective variety - the Satake compactification of $M_g$, see \cite[pages 435-437]{ACG}. For $g\leq 2$, the above is no longer true and the map is known to be non-surjective.}.

\begin{definition}[Teichm\"{u}ller modular forms]
A weight $k$ and degree $g$ Teichm\"{u}ller modular form is a holomorphic functions $F$ on $T_g$ such that
$$
F\left(\gamma((C,f)) \right)=\det(cj(C,d)+d)^kF(C,f)
$$
for every $(C,f)$ in $T_g$ and every $\gamma$ in $\Gamma_g$. Furthermore, if $g\leq 2$, we require that the corresponding section of $\lambda_g^{\otimes k}$ on $\M_g$ extends to a section of $\lambda_g^{\otimes k}$ over the compactification $\oM_g$.

\end{definition}

Let us stress that the above definition depends on the choice of the basis of the homology of $X$; if we change the basis, the space of Teichm\"{u}ller modular forms change accordingly.

\medskip

We now construct a surjective map from the Teichm\"{u}ller space to the Schottky space, using the chosen basis of the homology of $X$. Recall that a cut system on $X$ is the choice of  $g$ simple non intersecting loops linearly independent in homology. A symplectic basis $\{A_i,B_i\}$ of the homology gives a cut system choosing representative of $A_i$'s. Given an element $(C,f)$ of $T_g$, we can use $f$ to obtain a cut system for $C$. A cut system on $C$ is equivalent to a marked Schottky structure on $C$, see e. g.  \cite[proof of Prop. 5.2]{Schottky}. This gives a surjective holomorphic map from $T_g$ to $S_g$.  A function on the Schottky space is called a Teichm\"{u}ller modular form of weight $k$ if and only if its pull-back via this map is a Teichm\"{u}ller modular form of weight $k$.

A Teichm\"{u}ller modular form on $C_g$ is the pull-back of a Teichm\"{u}ller modular form from $S_g$; in particular, it is $PSL(2,\C)$ invariant. Since Teichm\"{u}ller modular forms are analytic, they are completely determined by their restriction to $C_g^{\circ}$ or $U_{g,r}$.

Let us summarize the relation among our moduli spaces with a diagram. 
$$
\begin{array}{ccccc}
	& &C_g & &\\
	& &\downarrow && \\
	T_g&\rightarrow & S_g&\rightarrow & \M_g 
\end{array}
$$
Recall that the map from $T_g$ to $\M_g$ is canonical, whereas all the others depend on the choice of a symplectic basis of the homology of $X$.  Using that diagram, we have also identified the space $H^0(\oM_g,\lambda_g^{\otimes k})$, with the space of weight $k$ Teichm\"{u}ller modular forms over $T_g$, and with convenient spaces of functions on $S_g$ and $C_g$.

\begin{remark}[Genus one and zero]\label{rem:g=1}
All the above theory can be worked out also when $g=1$, the only twist is that one has to consider curves with one marked point, and so the moduli spaces are $T_{1,1}$, $S_1$, $C_1$, and $\overline{\mathcal{M}}_{1,1}$. Let us describe the final results. $T_{1,1}$ is the Siegel upper half space, the mapping class group is $SL(2,\Z)$, $S_1$ is the punctured disc, $\overline{M}_{1,1}$ is the affine line (with some stacky structure), $\overline{\mathcal{M}}_{1,1}$ is the projective line (with some stacky structure), the restriction map of sections of powers of the Hodge line bundle from $\overline {\mathcal{M}}_{1,1}$ to $\mathcal{M}_{1,1}$ is not surjective because one has to impose convergence at infinity. Teichm\"{u}ller modular forms are the classical modular forms.

When $g=0$, we consider the spaces $T_{0,3}$ and $\mathcal{M}_{0,3}$, which are just two points, and the mapping class group is trivial. In this case, modular forms are constant functions.
\end{remark}

\subsection{Schottky coordinates and vanishing order along boundary divisors}\label{S:coord}

We can use Section \ref{SS:schottky} to construct coordinates around some points of the boundary divisors of $\oM_g$ (a similar construction works for $\oM_{g,n}$ but we do not need it), and use them to compute the vanishing order of Teichm\"{u}ller modular forms along boundary divisors. This kind of computation will be very important in the proofs of Corollary \ref{cor:classification1} and Corollary \ref{cor:classification2}.

\subsubsection{Boundary divisor $\delta_0$ and the family over $B_0$}\label{S:delta0} 

Pick in $\delta_0$ the isomorphism class of a curve $[C_0]$ whose normalization is $\P^1$, i.e., $C_0$ is constructed by gluing $g$ pairs of points $w_i$ and $z_i$ in $\P^1$ - we can and do assume that none of them is $\infty$. We construct a family which will give a chart around $[C_0]$.
 
 Take $B$ a small open neighborhood of $(z_1, \dots , z_g,w_1,\dots , w_g)$ in $\P^{\times 2g}$. Take $\Delta$ the product of $g$ small disks in the complex plane with parameters $q_i$'s; with $\Delta^*$ we indicate the product of punctured disks.  Up to shrinking $B$ and $\Delta$, we can assume that $B\times \Delta^*$ is inside the set $U_{g,r}\subset C_g^{\circ}$ from Definition \ref{def:open_in_Schottky} for some conveniently small $r$. We can thus consider the family of curves with marked Schottky structure over it (see e.g. the proof of Proposition \ref{prop:open_in_Schottky}).
 
We complete it over $q_1\cdots q_g=0$ adding nodal curves (roughly speaking, if $q_i=0$, then we glue $w_i$ and $z_i$ to form a node); the fiber over the point 
$$(z_1, \dots , z_g,w_1,\dots , w_g,0,\dots ,0)$$ 
is $C_0$ . This extension is done via the so called plumbing or sewing construction, described for instance in \cite[Section 1.6]{Gui}, \cite[pages 184-185]{ACG} and \cite[Section 2]{CS-B14}. It is an extension of the gluing relation from Equation \eqref{plumbing}. Let $\pi \colon \mathcal{C}\to B\times \Delta$ be the family we have obtained.

 \textbf{Let us call $B_0$ the base $B\times \Delta$} we have just constructed, and let the index $0$ be to recall its relation to the boundary $\delta_0$.
 
 Observe that up to the proof of convergence, the partition function $\cZ_{V,g}$ from Definition \ref{def:parition_functions} is a power series which gives a holomorphic function on an open subset of $B_0$,  and will extend to all $B_0$.

\medskip

When, later on, we will construct conformal blocks and covacua on $B_0$, we will put on $\P^1$ a standard coordinate $z$ such that none of the $w_i$'s and $z_i$'s is $0$ or $\infty$. We will also need a section of $\pi$. The section $x$ of $\P^1\times B \to B$ which is constantly equal to $0$, up to shrinking $B$, $\Delta$ and $r$, will induce a section of $\pi$ which by abuse of notations we will still call $x$. Observe that the complement of the image of $x$ is affine, i.e. this family has the property (Q). We will use $x$ together with the standard coordinate to construct conformal blocks and covacua. 
 
\medskip

 The moduli map $m(\pi)\colon B\times \Delta\to \oM_g$ is not injective, two points have the same image if and only if the corresponding points in $\P^{\times 2g}$ differ by the action of $SL(2,\C)$. In $\P^{\times 2g}$, we can take an open set $B \cong W\times O$, where $O$ is a neighborhood of the identity in $SL(2,\C)$ and $W\subset \P^{\times 2g}$, such that $m(\pi)$ restricted to $W\times \Delta$ gives a local chart around $[C_0]$. 
 
 \medskip
 
Let us explain how to use this family to compute the vanishing order of a Teichm\"{u}ller modular form along $\delta_0$. Since $\delta_0$ is an irreducible divisor in $\oM_g$, we can carry our computation in an open subset of $\oM_g$ that does not intersect trivially $\delta_0$. In $B_0=B\times \Delta$, the divisor $D=\{q_1\cdots q_g=0\}$ is the equation of the preimage of $\delta_0$ by $m(\pi)$. Let $F$ be a section of a power of the Hodge line bundle on $\oM_g$, we want to compute its vanishing order along $\delta_0$. Observe that $m(\pi)(D)$ is open in $\delta_0$, and $\delta_0$ is irreducible, so it is enough to compute it along $m(\pi)(D)$. We pull-back $F$ to $B\times \Delta$, and then the vanishing order along $D$ is the divisibility of $F$ by $q_1\cdots q_g$ (i.e. we trivialize the Hodge line bundle so that $F$ becomes a function, we write $F=(q_1\cdots q_g)^kG$, where $G$ is not divisible by $q_1\cdots q_g$, and then $k$ is the vanishing order). More generally we can pull-back $F$ to the Schottky space $C_g$ to obtain a Tecichm\"{u}ller modular form, then restrict to $B\times \Delta^*$, this pull-back extends along $D$ because $F$ was defined on all $\oM_g$, not just $\M_g$, and then check the divisibility by $q_1\cdots q_g$.

\subsubsection{Boundary divisor $\delta_i$ and the family over $B_i$}\label{S:deltai} 

Let us now treat the boundary divisor $\delta_i$ with $i\neq 0$. We pick a curve $C_0$ obtained taking two smooth curves $X$ and $Y$ of genus $i$ and 
$g-i$ respectively, points $w$ in $X$ and $z$ in $Y$, and then glues $w$ and $z$. We also assume that $X$ is a smooth member of a family of curves over basis $B_X\times \Delta_X$ which is as $B_0$ from Subsection \ref{S:delta0} with the addition of an extra section $w$.  More explicitly, $B_X$ parameterizes $2i+1$ marked points in $\P^1$, $w_1,\dots w_i, z_1,\dots z_i,w$, and $\Delta_X$ is the product of $i$ disck. Observe that this is open in $C_i\times \Delta$, where $C_i$ is the space genus $i$ curves with marked Schottky structure, and $\Delta$ is a further small disk parametrizing $w$. Thus, we have a family of genus $i$ curves $\pi_X\colon \mathcal{C}_X \to B_X\times \Delta_X$ with a section $w$. We make the analogue construction for $Y$.

We now construct a family over $\pi \colon \mathcal{C} \to B_i:=B_X\times \Delta_X \times B_Y\times \Delta_Y\times \Delta$, where $\Delta$ is a small disk with variable $q$. \textbf{ We call this base $B_i$}; the index $i$ is used to recall its relation with the boundary divisor $\delta_i$. Before doing it, let us immediately observe that, up to proving convergence, the partition function $\cZ_{V,g,i}$ from Definition \ref{def:parition_functions} gives a function on some open subset of $B_i$.

\medskip

We glue the families $\mathcal{C}_X$ and $\mathcal{C}_Y$ with the usual plumbing/sewing construction(\cite[Section 1.6]{Gui}, \cite[pages 184-185]{ACG} and \cite[Section 2]{CS-B14}), using the parameter $q$, and the coordinates around the marked points $w$ and $z$ given by the Schottky uniformization. We quickly review the construction for the reader convenience. We take open neighborhoods $W$ and $Z$ in $\mathcal{C}_X$ and $\mathcal{C}_Y$ of the images of $w$ and $z$. On these neighborhoods, we take local coordinates $x$ and $y$ induced by the standard coordinate from $\P^1$. We remove smaller neighborhoods $W'$ and $Z'$ of the images of $w$ and $z$, and then we glue $W\setminus W'$ with $Z\setminus Z'$ according to the relation $xy=q$.

\medskip

As in Subsection \ref{S:delta0}, we add sections to this family. We put standard coordinates on the two copies of $\P^1$ so that none of the marked points is $0$ or $\infty$. We consider constant sections $x$ and $y$ of $\P^1\times B_X\to B_X$ and of $\P^1\times B_Y\to B_Y$ constantly equal to zero. These sections induce sections of $\pi$ such that the family has property $(Q)$. We will use them together with the standard coordinates to construct conformal blocks and covacua.

\medskip

The moduli map $m(\pi)\colon U\to \oM_g$ is not injective, two points have the same image if and only if the corresponding points in $(\P^1)^{\times 2i+1}\times (\P^1)^{\times 2g-2i+1}$ differ by the action of $SL(2,\C)\times SL(2,\C)$. Arguing similarly to Section \ref{S:delta0}, we can take a neighborhood $O$ of the identity in $SL(2,\C)\times SL(2,\C)$, so that $B_X\times B_Y\cong W\times O$, and $m(\pi)$ restricted to $W\times\Delta_X\times \Delta_Y\times \Delta$ gives a local chart.

\medskip

Let us now explain how to use this family to compute the vanishing order of Teichm\"{u}ller modular forms on $\delta_i$. On $B_i$, let $D$ be the divisor $\{q=0\}$. At one point of $D$, there is the curve $C_0$ with which we started with. The preimages of $m(\pi)^{-1}(\delta_i)$ in $B_i$ is $D$.  Let $F$ be a section of a power of the Hodge line bundle on $\oM_g$, we want to compute its vanishing order along $\delta_i$. Observe that $m(\pi)(D)$ is open in $\delta_i$, and $\delta_i$ is irreducible, so it is enough to compute it along $m(\pi)(D)$. We pull-back $F$ to $B_i$, and then the vanishing order along $D$ is its divisibility by $q$.

\subsection{Siegel modular forms}\label{S:Sieg} 

 Siegel modular forms are a variant of Teichm\"{u}ller modular forms, let us recall here the definition. For $g=1$, Siegel and Teichm\"{u}ller modular forms are the same. Fix $g\geq 2$. Let $\H_g$ be the $g$ dimensional Siegel upper half space, i.e., $g$ by $g$ matrices $\tau$ that are symmetric and with positive definite imaginary part. It carries a natural action of the symplectic group $Sp(2g,\Z)$, namely 
$$\begin{pmatrix}
 a & b \\
 c & d\\
\end{pmatrix} \tau=(a\tau+b)(c\tau+d)^{-1}
$$

A degree $g$ and weight $k$ Siegel modular form is a holomorphic function $F\colon \H_g\to \mathbb{C}$ such that
$$
F\left(\begin{pmatrix}
 a & b \\
 c & d\\
\end{pmatrix} \tau\right)=\det(c\tau+d)^{k}F(\tau)
$$
for every $\tau$ and all matrices of $Sp(2g,\Z)$. Classical examples of weight $k$ and degree $g$ Siegel modular forms are the theta series associated to a positive definite, even unimodular lattice $(L, \langle \cdot, \cdot \rangle )$ of rank $2k$. The definition is 
$$
\Theta_{L,g}(\tau):=\sum_{v_1,\dots, v_g \in \Z^g}\exp  \left(2 \pi i \sum_{i,j}\langle v_i,v_j \rangle \tau_{ij}\right)
$$
A classical result of Freitag \cite{Fre77}, (see also \cite[Satz 5.3, Section IV.5]{Fre83} and \cite[Section 2.2 and references therein]{PhD_Cod}) says that, for $k$ divisible by $4$ and $g$ big enough with respect to $k$, the space of weight $k$ and degree $g$ Siegel modular forms has a basis given by all theta series associated to positive definite even unimodular lattices. 

The rank of such lattices is always divisible by $8$, and their classification is known only up to rank $24$. In rank $8$ there is only the lattice $E_8$; in rank $16$ there are the lattices $E_8^{\oplus 2}$ and $D_{16}^+$; in rank $24$ there are $24$ such lattices, they are usually called {\bf Niemeier lattices}.  This implies that the dimension of the space of weight $4$ (respectively $8$, $12$) Siegel modular forms is 1 (resp. $2$ , $24$) for $g$ big enough.

For rank at least $32$, the number of unimodular positive definite even lattices increases very quickly. Already for rank $32$ the exact number is not known and estimates show that it is larger than $8 \times 10^7$.

Later, we will relate Theta series to lattice vertex algebras.

\medskip

Denote by $\A_g$ the stack quotient $\A_g:=\H_g/ Sp(2g,\Z)$ (this is the moduli space of $g$ dimensional principally polarized abelian varieties, but this is not relevant for this paper). It has a Hodge line bundle $N_g$ such that, for $g\geq 2$, weight $k$ Siegel modular forms are exactly pull-backs of sections of $N_g^{\otimes k}$ (for $g=1$, one needs to add the classical convergence at infinity). In other words, the space of weight $k$ Siegel modular forms is isomorphic to $H^0(\A_g,N_g^{\otimes k})$. A general reference is \cite{Fre83}, and references therein. 

The Hodge line bundle $N_g$ is ample on $\A_g$. This in particular means that there exists a positive integer $A$, which depends on $g$, such that for all $k$ divisible by $A$, given a point $\tau$ of $\H_g$ there exists a positive definite even unimodular lattice $L$ of rank $2k$ such that $\Theta_{L,g}(\tau)\neq 0$.

\medskip

The period matrix of a Riemann surface gives a map  $J\colon T_g \to \H_g$, usually called Jacobi map. The image $J(T_g)$ is usually denoted by $J_g$, it is called the Jacobi locus, and it is the locus of matrices which are period matrices of Riemann surfaces. This map descends to a map $j\colon \M_g \to \A_g$. A classical theorem of Torelli says that $j$ is injective.

The pull-back of the Hodge line bundle $N_g$ is $\lambda_g$, and via $J$, weight $k$ and degree $g$ Siegel modular forms pull-back to weight $k$ and degree $g$ Teichm\"{u}ller modular forms. In \ref{prop:comp}, we will see that not all Teicum\"{u}ller modular forms are pull-back of Siegel modular forms.

Observe that the dimension of $J_g$ is $3g-3$ and the dimension of $\H_g$ is $\frac{g(g+1)}{2}$, so $J_g$ is strictly smaller than $\H_g$ for $g\geq 4$. This gives rise to the Schottky problem, which is discussed in more detail in Section \ref{S:schottky_problem}.

\section{Conformal blocks and covacua}
\label{SecConformalBlocks}

\subsection{Motivations and references}

In this section, we recall the theory of conformal blocks of vertex algebras needed for the proof of Theorems \ref{T:main1} and \ref{thm:altre_espansioni}. The properties of conformal blocks are predicted by conformal field theory, and they were widely expected to hold \cite{FS,Segal}.

For the Wess-Zumino-Witten-model, these results were established a long time ago, mainly in \cite{TUY}, see also \cite{Ueno}. They used Kac-Moody algebras rather than vertex algebras.

In the case of vertex algebras, the first main attempt was carried out in the book \cite{VA}; the main issue is that loc. cit. deals just with smooth curves, whereas we need nodal curves. It is also worth mentioning the papers \cite{NT} for the genus zero case, and \cite{HvE16} for (super)-elliptic curves. In \cite{Cod_partititon}, the second author stated and tried to prove all the results listed below for nodal curves, but the pre-print still contains some inaccuracies. Recently, there have been a series of three papers on this topic by A. Gibney, C. Damiolini and N. Tarasca \cite{DGT1,DGT2,DGT3,T}, where the reder can find algebraic proofs.

From an analytic point of view, we quote and use recent results by Bin Gui from \cite{Gui24,Gui24c}. Se also \cite{Gui,GZa,GZb,GZc}.

\subsection{The definition and some examples}\label{S:def}
Let $V$ be a holomorphic vertex algebra\footnote{Most of the facts that we are going to recall work for conformal, rational and $C_2$ cofinite vertex algebras, but we are going to stick to holomorphic vertex algebras as this simplifies some definitions and some statements, and it is the only case of interest for this paper. In particular, to construct these sheaves one has to choose a representation of the vertex algebra for each section; in the holomorphic, we have only one possible choice for the module, and we thus omit it from the notations.}, and $\pi \colon \cC \to B$ a family of nodal curve with $n$ sections.  The sheaves of covacua $\T$ and of conformal blocks $\T^*$ are dual sheaves on $B$, whose construction depends on the family and on the vertex algebra. When we want to include the family in the notation, we write $\T(\pi)$ and $\T^*(\pi)$; when we want to include the vertex algebra in the notation, we write $\T(V)$ and $\T^*(V)$. These sheaves turn out to be line bundles, i.e. locally free sheaves (= vector bundles) of rank one. When $B$ is just a point, these sheaves are vector spaces, and one speaks of spaces of covacua and conformal blocks. The construction is as in \cite[Sections 2,3 and 4]{DGT2}, \cite[Section 3 and 4]{Gui24c} and \cite[Sections 2 and 3]{Gui24}. We also find useful the exposition from \cite[Chapter 3]{Gui}.

We denote by $\O_B$ the sheaf of holomorphic functions on $B$, i.e. the trivial line bundle on $B$. We denote by $V^{\otimes n}\otimes \O_B$ the trivial vector bundle on $B$ with fibers $V^{\otimes n}$, and by  $\left(V^{\vee}\right)^{\otimes n}\otimes \O_B$ the trivial vector bundle on $B$ with fibers $\left(V^{\vee}\right)^{\otimes n}$. 

The sheaf $\T$ can locally be defined as a quotient of  $V^{\otimes n}\otimes \O_B$, whereas $\T^*$ can locally be defined as the sub-sheaf of $\left(V^{\vee}\right)^{\otimes n}\otimes \O_B$.

First, we have to construct the locally free sheaf (= vector bundle) $\mathcal{V}_{\cC}$ on $\cC$, usually called the sheaf of vertex algebras. Let us recall the construction away from the nodes and refer to the above mentioned papers for the construction in a neighborhood of the nodes. First, we define the sheaf $\mathcal{V}^{\leq d}_{\cC}$. Let $V^{\leq d}:=\bigoplus_{k\leq d} V_k$. Given an open set $\mathcal{U} \subset \cC$ with a vertical coordinate\footnote{By vertical coordinate we mean a holomorphic function $z_\mathcal{U} \colon \mathcal{U} \to \C$ which, for every $b$ in $B$, when restricted to $\mathcal{U}_b :=\pi^{-1}(b) \cap \mathcal{U}$ realizes an isomorphism between $\mathcal{U}_b$ and an open subset of $\C$} $z_\mathcal{U}$, the sheaf is a copy of $V^{\leq d} \otimes \O_\mathcal{U}$, i.e. the trivial vector bundle with fiber $V^{\leq d}$. Given another open set $\mathcal{W}\subset \cC$ with local coordinate $z_\mathcal{W}$, we have to define the transition function on $\mathcal{U}\cap \mathcal{W}$. At a point $p$ in $\mathcal{U}\cap \mathcal{W}$ we have two local coordinates, $z_\mathcal{U}-z_\mathcal{U}(p)$ and $z_\mathcal{W}-z_\mathcal{W}(p)$, and let $\rho_p$ be the change of coordinates, i.e. $\rho_p(t)$ is a holomorphic function on $p$ and $t$ such that $\rho_p(z_\mathcal{U}-z_\mathcal{U}(p))
 =z_\mathcal{W}-z_\mathcal{W}(p)$. The group of change of coordinates acts on $V^{\leq n}$ via the Virasoro algebra, see \cite[Section 6.3]{VA} or \cite[Section 2.3]{Gui}, and this action is used to define transition functions. 

The description of this transition function simplifies if we look at the quotient  $\mathcal{V}^{\leq d}_{\cC}/ \mathcal{V}^{\leq d-1}_{\cC}$. On an open subset $\mathcal{U}$ with local coordinates $z_\mathcal{U}$, the quotient  $\mathcal{V}^{\leq d}_{\cC}/ \mathcal{V}^{\leq d-1}_{\cC}$ is the trivial vector bundle with fiber $V_d$; given another open subset $\mathcal{W}$ with local coordinate $z_\mathcal{W}$, the transition function on $\mathcal{U}\cap \mathcal{W}$ is the holomorphic function given by the $d$-th power of the derivative of one coordinate with respect to the other, i.e. $\left(\frac{\partial z_\mathcal{U}}{\partial z_\mathcal{W}}\right)^d $, see  \cite[Equations 2.5.7 and 2.5.8, Proposition 2.5.4]{Gui} and \cite[Sections 6.4 and 6.5]{VA}. A similar formula holds for quasi-primary vectors: if $v$ is quasi-primary, we can see it as a constant local section of $\mathcal{V}^{\leq d}$ on $\mathcal{U}$, and on $\mathcal{U}\cap \mathcal{W}$ it is glued to the (non-constant) section $\left(\frac{\partial z_\mathcal{U}}{\partial z_\mathcal{W}}\right)^d v $.

The sheaf $\mathcal{V}_{\cC}$  is then defined as the inductive limit of the sheaves $\mathcal{V}_{\cC}^{\leq d}$.

\begin{example}\label{ex:P1_vb_and_set_up}
We (partially) describe the sheaf of vertex algebras $\mathcal{V}_{\cC}$ when $B$ is a point, and $\cC$ is $\mathbb{P}^1$. 

On $\P^1$, let $z$ be the standard coordinate,  $\mathcal{U}=\P^1\setminus \infty$ with coordinate $z_\mathcal{U}=z$, and $\mathcal{W}=\P^1\setminus 0$ with coordinate $z_\mathcal{W}=z^{-1}$. 

The transition function for $\mathcal{V}_{\P^1}^{\leq d}/\mathcal{V}_{\P^1}^{\leq d-1}$ on $\mathcal{U}\cap \mathcal{W}$ is the multiplication by $(-1)^d z^{2d}$, hence the quotient $\mathcal{V}_{\P^1}^{\leq d}/\mathcal{V}_{\P^1}^{\leq d-1}$ is isomorphic to $\mathcal{O}_{\P^1}(2d)^{\oplus \dim V_d}$.

Let us construct some special global sections of $\mathcal{V}_{\P^1}$ associated to quasi-primary vectors. Given a quasi-primary vector $v$ of degree $d$, we define a section $\phi_v$ as the constant section $v$ on $\mathcal{U}$, and $(-1)^d z^{-2d}v$ on $\mathcal{W}$. If $v$ is a (non-necessarily primary) vector of degree $d$, the constant section $v$ on $\mathcal{U}$ is glued on $\mathcal{U}\cap \mathcal{W}$ with $(-1)^d\sum_{i=0}^d(i!)^{-1} z^{-2d+i}L_1^iv$, see \cite[Example 2.3.2]{Gui}. 
\end{example}

To keep on going with the construction, we make the following assumptions on the family
\begin{assumption}\label{assumption_construction}
\begin{enumerate}
\item  the family has the property (Q) introduced in Section \ref{SS:family}, i.e. the complement of the image of the sections is affine. In particular, the family has at least one section, i.e. $n\geq 1$. We denote by $S_i$ the image of $s_i$, and by $S$ the union of all the $S_i$'s.
\item We can and do fix coordinate around each $S_i$; more explicitly, for every $i$ we can and do fix an open neighborhood $\mathcal{U}_i$ of $S_i$ and a vertical coordinate $z_i$ on $\mathcal{U}_i$ such that $S_i=\{z_i=0\}$.
\end{enumerate}
\end{assumption}

We recall some standard notation about sheaves for the reader convenience. Given a sheaf $F$ on a variety $X$, the space of global sections is denoted by $H^0(X,F)$. Given a family $\pi \colon \cC\to B$ and a sheaf $F$ on $\cC$, the push-forward $\pi_*F$ is a sheaf on $B$ whose section over an open subset $A$ of $B$ is $H^0(\pi^{-1}(A),F)$. 

Let $\omega_{\cC/B}(*S)$ be the sheaves of relative meromorphic differentials on $\cC$ with algebraic poles only at $S$. Fix $i$, and let $\phi_i$ be a section of the sheaf $\mathcal{V}_{\cC}\otimes \omega_{\cC/B}(*S)$ over a neighborhood $\mathcal{U}_i$ of $S_i$. On $\mathcal{U}_i$, we can write $\phi_i$ as a sum of pure tensors $\sum_{j=1}^{k_i}v_i^{(j)}\otimes f_i^{(j)}dz_i$, where $v_i^{(j)}$ is a vector of $V$, and $f_i^{(j)}$ is a meromorphic function on $U_i$, so $f_i^{(j)}$ depends both on the vertical coordinate $z_i$ and the coordinates from $B$. We associate to $\phi_i$ an endomorphism of $V\otimes \O_B$, which we still denote by $\phi_i$; for $v$ in $V$, we let
$$
\phi_i(v):=\sum_{j=1}^{k_i} \Res_{z_i=0}\left(Y\left(v_i^{(j)},z_i\right)f_i^{(j)}(z_i)dz_i\right)v \in V \otimes \O_B 
$$
We extended $\phi_i$ from $V$ to $V\otimes \O_B$ by $\O_B$-linearity.

Given a section $\phi$ of $\pi_*\left(\mathcal{V}_{\cC}\otimes \omega_{\cC/B}(*S)\right)$, we call $\phi_i$ its restriction to a neighborhood $\mathcal{U}_i$ of $S_i$. We define an action of $\pi_*\left(\mathcal{V}_{\cC}\otimes \omega_{\cC/B}(*S)\right)$ on $V^{\otimes n}\otimes \O_B$ in the following way: for $v_1\otimes \cdots \otimes v_n$ in $V^{\otimes n}$, we let
$$
\phi\cdot (v_1\otimes \cdots \otimes v_n)=\sum_{i=1}^nv_1\otimes \cdots \otimes \phi_i(v_i)\otimes \cdots \otimes v_n
$$
we extended this action by $\O_B$-linearity from  $V^{\otimes n}$ to $V^{\otimes n}\otimes \O_B$. This action is as in \cite[Equations 3.2.4 and 3.2.5]{Gui}. (Let us make contact with the notations of loc. cit., first the symbol $v$ is used for a section of $\mathcal{V}_{\cC}\otimes \omega_{\cC/B}(*S)$ restricted to a neighborhood of $S_i$, and the operator $\mathcal{V}_{\rho}(\eta_i)$ is a trivialization of the bundle, see Equation 3.2.3; then $v$ is used for global sections of $\mathcal{V}_{\cC}\otimes \omega_{\cC/B}(*S)$).

\medskip

The sheaf of \textbf{covacua} is the sheaf of coinvariant for this action: it is the biggest quotient of $V^{\otimes n}\otimes \O_B$ where $\pi_*\left(\mathcal{V}_{\cC}\otimes \omega_{\cC/B}(*S)\right)$ acts trivially. In other words, the sheaves of covacua $\T$ is the quotient of $V^{\otimes n}\otimes \O_B$ by $\pi_*\left(\mathcal{V}_{\cC}\otimes \omega_{\cC/B}(*S)\right)\cdot V^{\otimes n}\otimes \O_B$, see \cite[Equation 3.2.7]{Gui}.

The sheaf of \textbf{conformal blocks} is the dual of the sheaf of covacua. It consists of all sections of $\left(V^{\vee}\right)^{\otimes n}\otimes \O_B$ that are annihilated by the action of $\pi_*\left(\mathcal{V}_{\cC} \otimes \omega_{\cC/B}(*S)\right)$.

\begin{example}\label{ex:covacP1}
We continue Example \ref{ex:P1_vb_and_set_up}. We take $S=\{0\}$, so just one marked point. A tensor product $ v\otimes z^n dz$, where $v$ is a vector of degree $d$ in $V$, and $n\in \mathbb{Z}$, is a local section of $\mathcal{V}_{\P^1}\otimes \omega_{\P^1}(*0)$ around $0$. Assume now that $v$ is a quasi-primary vector. We want to check if $ v\otimes z^n dz$ extends to a global section $\phi_{v,n}$ of $\mathcal{V}_{\P^1}\otimes \omega_{\P^1}(*0)$ - i.e. to an element of $\pi_*\left(\mathcal{V}_{\P^1}\otimes \omega_{\cC/B}(*S)\right)=H^0(\P^1,\mathcal{V}_{\P^1}\otimes \omega_{\P^1}(*0))$. Following Example \ref{ex:P1_vb_and_set_up}, if we restrict it to $\mathcal{U}\cap \mathcal{W}$, and we change coordinates, this section becomes $w^{2d}v\otimes w^{-n-2}dw$, where $w=z^{-1}$, hence it extends to all $\P^1$ if and only if $n\leq 2d-2$.  The span of the vectors of the form $(v\otimes z^ndz)\cdot \Omega^V$, where $v$ is any non-trivial quasi-primary vector and $n$ is any non-positive integer such that $n\leq 2\deg{v}-2$ is $\bigoplus_{k> 0}V_k$, hence the covacua are a quotient of $V_0$. 

Let us show that they are non zero, so (canonically) isomorphic to $V_0$. Let $\phi$ be a global section of $\mathcal{V}_{\P^1}\otimes \omega_{\P^1}(*0)$. Since the pole at $\infty$ of $\phi$ is algebraic, locally around $0$ can be written out as a finite sum $\sum_j v_j\otimes z^{n_j}dz$, where the degree of $v_j$ is $d_j$. On $\mathcal{U}\cap \mathcal{W}$, it is glued with $(-1)^{d_j}\sum_j \left(\sum_{i=0}^{d_j} w^{2d_j-i}L_1^iv_j\right)\otimes w^{-n_j-2}dw$. If this section is regular at $\infty$, then $2d_j-n_j-2\geq 0$. This implies that the operator obtained out of global sections have strictly positive degree, so the space of covacua is exactly $V_0$.

\medskip

 The space of conformal blocs is dual to the space of covacua, so it is spanned by the linear functional $a \mapsto (\Omega^V,a)$.
\end{example}

\begin{example}[Correlation functions]\label{ex:covacP1_tantipunti}
We generalize \ref{ex:covacP1}. Again we take $B$ a point and $\cC=\mathbb{P}^1$ the projective line, with standard coordinate $z$. Now we take $n$ marked points $z_i$, different from $\infty$, and local coordinates $w_i:=z-z_i$.  We define a linear functional $\Phi$ on $V^{\otimes n}$ as 
$$a_1\otimes \cdots \otimes a_n \mapsto (\Omega^V,Y(a_1,z_1)\cdots Y(a_n,z_n)\Omega^V)$$ 
This function is usually called a correlation function; see \ref{d:correlation_function} and \ref{T:commutation}. Let us explain why the above formula gives a well defined complex number. First, one has to see the $z_i$'s as formal variables. Then, by \cite[Section 3.5]{Axiomatic}, the above formula gives a rational function in the $z_i$'s with poles only where two $z_i$'s coincide. The variables can then be evaluated at the desired values.

We have to show that $\Phi$ is a conformal block, so that every element of $\phi$ of $H^0(\P^1,\mathcal{V}\otimes \omega_{\P^1}(*S))$ annihilates $\Phi$, where now $S=\{z_1,\cdots , z_n\}$. This is a classical consequence of associativity. 

As a consequence of Propagation of Vacua (Section \ref{S:prop}) and Example \ref{ex:covacP1}, $\Phi$ is a basis of the space of conformal blocks, and it is mapped to the conformal block $a \mapsto (\Omega^V,a)$ from Example \ref{ex:covacP1} by the isomorphism given in the Theorem of Propagation of Vacua.
\end{example}
The definition, so far, depends on the choice of local coordinates. One can show that it is covariant with respect to the choice of local coordinates, so the resulting sheaves are independent. This is done for instance in \cite[page 71-72, see also Lemma 3.11]{Gui}. More precisely, one introduces a sheaf on $B$ similar to $\mathcal{V}$: for every open set $\mathcal{U}$ in $B$ and every choice of vertical local coordinates $z_i$ around $S_i$, it is isomorphic to  $V^{\otimes n}\otimes \O_{\mathcal{U}}$, and then one introduces convenient transition functions. For every choice of local coordinates, we have an action of $\pi_*\left(\mathcal{V}\otimes \omega_{\cC/B}(*S)\right)$ as before, and one can show that the action is independent of the choice of local coordinates. (In this way, one can also see that the action is well-defined also if one can not choose vertical local coordinates on the full family: it is enough to define them locally on $B$, and this shows that the action glue together).

\medskip

If the family does not satisfy the two assumptions  \ref{assumption_construction} (maybe it does not even have sections at all), one first covers the base $B$ with a small enough open subset $\mathcal{U}_i$ such that if we restrict the families over the $\mathcal{U}_i$'s, it is possible to add sections so that property (Q) holds. One then construct the sheaves on each $\mathcal{U}_i$. Now, one has to show that the sheaves are independent of the choice of the sections (this is a consequence of Propagation of Vacua, see Section \ref{S:prop}), and they glue together.

\subsection{Conformal blocks and covacua are line bundles}\label{S:rank}
One shows that conformal blocks and covacua are dual vector bundles on $B$, not just sheaves. This is a consequence of many results, such as factorization formula, sewing construction, and the description of the algebra of differential operators acting on these bundles. References are \cite[VB corollary page 4]{DGT2} or \cite[Theorem 5.5]{Gui24}.

Since $V$ is holomorphic, the rank of these vector bundles is one, i.e. these sheaves are dual line bundles. This is well-known; let us give a quick explanation. By functoriality (Section \ref{S:funct}), it is enough to do it for the universal family over $\oM_g$. We compute the rank of the fiber of $V$ over a curve whose normalization is $\mathbb{P}^1$. There, using the factorization formula \cite[page 3 Equation (1)]{DGT2}, propagation of vacua (Section \ref{S:prop}), and then Example \ref{ex:covacP1}, we obtain that the dimension of that fiber is one. The reader can also look at the discussion after Example 5.3.2 of \cite{Gui}. 

(For the sake of completeness, let us say that when $V$ is not holomorphic but just rational and $C_2$ cofinite, these bundles are finite rank dual vector bundles. The rank is always strictly greater than one if $V$ is not holomorphic; it depends non-trivially on the genus of the curves and the representation theory of the algebra.)

\subsection{Functoriality}\label{S:funct}
Let $f\colon T\to B$ be a morphism and consider the pulled-back family $\pi_T\colon \cC_T=\cC\times_B T \to T$. We also pull-back sections and local coordinates, when present. Let $\T(\pi)$ and $\T(\pi_f)$ be, respectively, the bundles of covacua on $B$ and $T$. Then there is a natural map between $f^*\T^*(\pi)$ and $\T^*(\pi_f)$ which turns out to be an isomorphism. By duality, the analogous statement holds for the sheaf of covacua $\T$.

This is equivalent to say that $\T$ and $\T^*$ define line bundles on the moduli stack of nodal curves $\oM_{g,n}$. This is shown in \cite{DGT1}, see also \cite[Section 5.4]{Gui}.

\subsection{Tensor product}\label{S:tens}
Given two holomorphic vertex algebras $U$ and $V$, we have natural inclusion $i\colon \T^*(U)\otimes \T^*(V)\hookrightarrow \T^*(U\otimes V)$. Because of Section \ref{S:rank}, domain and codomain of $i$ are line bundles. Thanks to the functoriality of Section \ref{S:funct}, to show that this map is an isomorphism it is enough to show it on $\oM_{g,n}$. Since holomorphic functions on $\oM_{g,n}$ are constant, and $i$ is a non-zero map, if we show that abstractly $\T^*(U)\otimes \T^*(V)$  and $\T^*(U\otimes V)$ are isomorphic, then we obtain that $i$ is an isomorphism. Denoting by $c_U$ (resp. $c_V$) the central charge of $U$ (resp. $V)$, the central charge of $U\otimes V$ is $c_U+c_V$. In Section \ref{S:unif}, we will see that $\T^*(U)$ (resp. $ \T^*(V)$, $\T^*(U\otimes V)$) is isomorphic to $\lambda_g^{-c_U}$ (resp. $\lambda_g^{-c_V}$, $\lambda_g^{-c_U-c_V}$), hence, $i$ is an isomorphism. Dually, we get an isomorphism  $\T(U\otimes V)\cong \T(U)\otimes \T(V)$.

\subsection{Propagation of vacua}\label{S:prop}
Let $\varrho_{m,n}\colon V^{\otimes n}\to V^{\otimes (m+n)}$ be the map defined taking the tensor with $(\Omega^V)^{\otimes m}$, and let $\varrho_{m,n}^{\vee}\colon \left(V^{\vee}\right)^{\otimes (n+m)}\to \left(V^{\vee}\right)^{\otimes n}$ be its dual. 

Suppose one adds to our original family $\pi \colon \cC\to B$ other $m$ sections $t_1,\dots, t_m$. Let $\T_n$ be the bundle constructed using only the first $n$ sections, and $\T_{m+n}$ be the bundle constructed using all the sections.  We use the analog notation for $\T_n^*$ and $\T_{m+n}^*$.

The \textbf{Theorem of Propagation of Vacua} asserts that $\varrho_{m,n}$ induces an isomorphism between $\T_n$ and $\T_{m+n}$, and $\varrho_{m,n}^{\vee}$ induces an isomorphism between $\T_n^*$ and $\T_{m+n}^*$.

The proof of propagation of vacua from \cite[Theorem 10.3.1]{VA} goes through also in this set-up with minor modifications. Another reference is \cite[Theorem 7.1]{Gui24c}, see also \cite[Section 3.4]{Gui}.

\subsection{Factorization}\label{S:factorization}
Assume now that $\pi$ has a horizontal node, i.e. there is a section $n$ of $\pi$ such that $n(b)$ is a node of $\cC_b$ for every $b$ in $B$. Let $\nu\colon \hat{\cC}\to \cC$ the normalization of this horizontal node. This is either a new family $\pi_{\nu}$ on $B$, or the union of two families $\pi_1$ and $\pi_2$ on $B$. 

In the first case, the \textbf{Factorization Theorem} states that the map $\varrho_{n,2}$ defined in Section \ref{S:prop} induces an isomorphism between $\T(\pi)$ and $\T(\pi_{\nu})$. In the second case, it states that $\varrho_{n,2}$ induces an isomorphism between $\T(\pi)$ and $\T(\pi_1)\otimes \T(\pi_2)$. The analogous statement holds for the bundle of conformal blocks.

This is the main result of \cite{DGT2}, see also \cite[Sections 4.7]{Gui}.

\subsection{Sewing of conformal blocks}\label{S:sewing}

The sewing theorem is usually stated in a rather general form; it is an interplay between conformal blocks and plumbing construction (the plumbing construction is reviewed in Section \ref{S:deltai}). Roughly speaking, it goes as follows. It starts from a section of the bundle of conformal blocks $s$ for a family of genus $g$ curves with $k+2$ sections; then, on the family of genus $g+1$ curves with $k$ marked points obtained via plumbing construction, we get a new conformal block $s'$ constructed via the so called sewing procedure. Alternatively, one  starts from two families, one of genus $a$ with $k+1$ sections and one conformal block $s_a$, the other of genus $b$ with $h+1$ section and a conformal block $s_b$, construct a family of genus $a+b$ curves with $h+k$ sections gluing points with the plumbing construction and construct there a conformal block $s$ sewing $s_a$ and $s_b$.

Let us spell out the sewing theorem in the case that will be used in this paper. First, the version for the family $\pi\colon \cC\to B_0=B\times \Delta$ from Subsection \ref{S:delta0} that will be used in the proof of Theorem \ref{T:main1}. 

We start from the trivial family $p\colon \P^1\times B\to B$, with $2g+1$ sections, which we denote by $w_1,\dots, w_g,z_1,\dots , z_g,x$. Using all sections $x$, $w_i$ and $z_i$, we can construct the bundle of conformal blocks $\T^*(p)$ as the sub-bundle of the trivial bundle $\left( V^{\vee}\right)^{\otimes (1+ 2g)} \otimes \O_B$.

Following Subsection \ref{S:delta0}, we can construct a family $\pi\colon \cC\to B_0=B\times \Delta$ with a section $x$  and a local coordinate around $x$ out of the family $p$. Using the section $x$ with its local coordinates, we can construct the bundle of conformal blocks $\T^*(\pi)$ as a sub-bundle of the trivial $V^{\vee}\otimes \O_{B\times \Delta}$.

The sewing of conformal blocks produces a section $s'$ of $\T^*(\pi)$ out of a section $s$ of $\T^*(p)$. It goes as follows. Denote by $\Id_k$ the identity element in $V_k \otimes V_k^{\vee}$. Choosing an orthonormal basis $v^{(i)}$ with respect to the scalar product $( \, , \,)$ on $V_k$, we can identify $V_k$ and $V_k^{\vee}$, so we can identify $\Id_k$ with $\sum v^{(i)}\otimes v^{(i)}$ in $V_k\otimes V_k$. Now we let
\begin{equation}\label{E:sewing}
\overline{s}:=\sum_{(n_1,\dots , n_g)\in \N^g}\left( s\otimes \Id_{n_1}\otimes \cdots \otimes \Id_{n_g}\right) q_1^{n_1}\cdots q_g^{n_g}
\end{equation}
This is a formal power series in $\left(\left(V^{\vee}\right)^{1+2g} \otimes \O_B \right)\otimes V^{\otimes 2g}[[q_1,\dots ,q_k]]$. Observe that there is a natural contraction map from $\left(V^{\vee}\right)^{\otimes (1+2g)} \otimes V^{\otimes 2g}$ to $V^{\vee}$. We let $s'$ be the image of $\overline{s}$ via this contraction.

The \textbf{Sewing Theorem} asserts that $s'$ converges on $B\times \Delta$ to a section of $\T^*(\pi)$. (Observe that $s'$ is a Taylor series in the $q_i$'s with coefficients holomorphic functions on $B$, so it makes to say that it converges on all $B\times \Delta$ ; it does not have all the issues about domain of convergences of Laurant series.)

The convergence of sewing is the main result of \cite{Gui24}, see also \cite[Sections 3.3 and 4.3]{Gui}; in \cite[Introduction]{Gui24} the reader can find the history of the proofs of convergence and a comprehensive bibliography.

\medskip

The following version of the sewing theorem will be needed in the proof of \ref{thm:altre_espansioni}, it is for the families $\pi\colon \cC\to B_i$ from Section \ref{S:deltai}. Let $s_X$ be a conformal block for $\pi_X$, and $s_Y$ a conformal block for $\pi_Y$ (these are families of curves with one marked point, so $s_X$ and $s_Y$ are linear functionals on $V$). On $V$, consider the power series
$$
s':= \sum_{k\in \mathbb{N}}  \sum_{i=1}^{\dim V_k} s_X\left(v^{(i)}_k\right)s_Y\left(v^{(i)}_k \right) q^k
$$
where $\{v^{(i)}_k\}$ is an orthonormal base of $V_k$. The Sewing Theorem asserts that $s'$ converges on $B_i$ to a section of the conformal blocks of $\pi$ on all $B_i$.

Separately, $s_X$ and $s_Y$ can be thought of as conformal blocks obtained via the sewing procedure described in the first part, so $s'$ can be further expanded in variables $w_i$'s, $z_i$' and $q_i$'s.

\subsection{The vacuum section}\label{SS:vacuumSec}

The line bundle of covacua $\T_g(V)$ always has a canonical section, called the vacuum section and denoted by $\oOmega_{V,g}$. Let us explain the construction.

First, we cover $B$ with small enough open subsets $B_i$ to allow us to add sections to the family to guaranty the property (Q), and be able to choose local coordinates. 

On each $B_i$, we choose local coordinates, so that $\T$ can be realized as a quotient of $V^{\otimes n}\otimes \O_B$, for some $n$. We take as vacuum section the image of $\left(\Omega^V\right)^{\otimes n}$ in $\T$, where $\Omega^V$ is the vacuum vector of $V$.

The group of change of local coordinates acts on $V$ via its Lie algebra, which is the subalgebra of the Virasoro algebra spanned by the $L_n$'s with $n\geq 0$, and it annihilates the vacuum vector. This means that the vacuum section is independent on the choice of local coordinates. The vacuum sections on each open subset of $B$ can then be shown to glue to a global section of $\T$ on $B$.

As we shall see in Section \ref{S:proof1}, the vacuum section is strongly related to the genus $g$ partition function of $V$.

The vacuum section was first considered in \cite{Cod_partititon}; because of its strong connection with the power series appearing in Definition \ref{def:parition_functions}, it was also called partition function. This section was then considered in \cite[Definition 3.1]{Gui24}, where it was called vacuum section. We have opted for this second name to stress the difference with the power series. We have decided to use the name partition function only for the power series.

Due to the Theorem of Propagation of Vacua (Section \ref{S:prop}) and the Factorization Theorem (Section \ref{S:factorization}), the vacuum section is preserved by the morphisms from Section \ref{S:moduli}, namely the forgetful morphism $\oM_{g,n+1}\to \oM_{g,n}$, and the gluing morphisms $\oM_{g-1,n+2}\to \oM_{g,n}$ and $\oM_{h,m+1}\times \oM_{k,n+1}\to \oM_{h+k,m+n}$, for all values of h,k,g,m and n. The vacuum section is also preserved by the tensor product construction described in Section \ref{S:tens}.

\section{Atiyah algebras and Virasoro uniformization}\label{S:Ati}
\subsection{Generalities on Atiyah algebras of line bundles} 

We gather some general facts about Atiyah algebras mainly following \cite[Section 5.2]{MOP}, see also  \cite[Section 2]{Tsu} , \cite{BS}, and  \cite[Sections 2 and 6]{DGT1}. 

Given a line bundle $L$ on a smooth variety $B$ or a smooth Deligne-Mumford stack, its Atiyah algebra $\A_L$ is the sheaf of derivations acting on sections of $L$, with its natural structure of sheaf of Lie algebras. An important variant is the log Atiyah algebra $\A_L(-\log(D))$ associated to $L$ and a normal crossing divisor $D$ in $B$: it is the sheaf of differential operators $A$ such that for every section $s$ of $L$ vanishing along $D$, also $A(s)$ vanishes along $D$. 

 The symbol map naturally presents $\A_L(-\log(D))$ as an extension of both sheaves of Lie algebras and $\O_B$-modules
\begin{equation}\label{eq:fund_sequence}
0 \to \O_B \to \A_L(-\log(D)) \to TB(-\log(D)) \to 0
\end{equation}
Where $TB(-\log(D))$ is the sheaf of vector fields on $B$ tangent to $D$. This extension is called the fundamental sequence of $\A_L(-\log(D))$. When $D=0$, i.e. $\A_L(-\log(D))=\A_L$, the class of this extension is equivalent to the first Chern class of $L$. A holomorphic connection on $L$ is a splitting of this exact sequence.  

We can define the  notion of (log) Atiyah algebra without using a line bundle: a (log) Atiyah algebra is a sheaf of Lie algebras and $\O_B$-modules which sits in an exact sequence as \eqref{eq:fund_sequence}.

Assuming now that a log Atiyah algebra $\A$ acts on two isomorphic line bundles $L$ and $M$, we claim that we can find an isomorphism between $L$ and $M$ which is $\A$-equivariant. Let us explain how to do it. Following e.g. \cite[Section 5.2]{MOP}, the trivial Atiyah algebra $\O_B\oplus TB(-\log(D))$ acts on $L\otimes M^{-1}$ (see e.g. \cite[Section 5.2]{MOP}). Since $L\otimes M^{-1}$ is trivial, the constant function is annihilated by the trivial Atiyah algebra, and it corresponds to the $\A$ equivariant isomorphism.

(For the sake of completeness, observe that if two line bundles have the same Atiyah algebra, then they have the same first Chern class, hence if the group of isomorphism classes of degree zero line bundles on $B$ is trivial, they are isomorphic.)

\subsection{Virasoro uniformization of the Hodge line bundle}\label{S:unif_and_kod} The Virasoro uniformization, sometime also called Krichever's construction, is described in many papers, here we follow mainly \cite[Section 2]{DGT1} and \cite[Section 3, in particular Prop. 3.19]{ADKP}, see also \cite[Section 17.3]{VA} and \cite[Sections 3 and 4]{BS}. 

Given a family of curves $\pi\colon \cC\to B$ with sections $s_1,\dots, s_n$ and at least one smooth fiber, we have the \emph{Kodaira-Spencer map}
\begin{equation}\label{eq:KodSpe}
R^1\pi_*T_{\pi}(-S)\to TB(-\log(D))
\end{equation}
where $T_{\pi}$ is the relative tangent bundle of $\pi$, $S$ be the union of the image of the sections, $R^1\pi_*T_{\pi}(-S)$ is the vector bundle on $B$ whose fiber over a point $b$ is $H^1(\cC_b,T(-S_b))$, $D$ is the divisor in $B$ over which the fibers are singular, and $TB(-\log(D))$ is the sheaf of vector fields on $B$ that preserve $D$.

Let $W$ be the Witt algebra, i.e. the algebra spanned by differential operators $L_n:=z^n\frac{d}{dz}$ with $n\in \mathbb{Z}$. On $B$, we look at the trivial vector bundle $\O_B\otimes W$ with fiber $W$.  If the family has the property (Q), i.e. the complement of the sections is affine, choosing vertical coordinates around the sections, we get a surjective map
\begin{equation}\label{E:vir_unif_p}
W^{\oplus n} \otimes \O_B\to TB(-\log(D))
\end{equation}
defined as follows. We define a map from  $W^{\oplus n} \otimes \O_B$
to $R^1\pi_*T_{\pi}(-S),$ and then compose with the Kodaira-Spencer map of Equation \eqref{eq:KodSpe}. To define the first map, we cover $\cC$ with affine open sets given by the neighborhoods of the sections and their complements; then we associate to elements of $W^{\oplus n} \otimes \O_B$ vector fields in the intersections of open sets of this cover using the local coordinates, and we conclude using the description of $R^1\pi_*T_{\pi}(-S)$  based on Chech co-homology.  ( In other words, given a point $b$ of $B$, we attach to every element of $W^{\oplus n}$ an infinitesimal deformation of $\pi^{-1}(b)$ preserving the singularities. )

For every positive integer $c>0$, we consider the central extension $Vir_{c,n}$ of $\O_B\otimes W^{\oplus n}$ by $\O_B$ given by the following commutator
$$
\left[\bigoplus_{i=1}^n f_iL_{n_i} \, , \, \bigoplus_{i=1}^n g_iL_{m_i}\right]=\bigoplus_{i=1}^nf_ig_iL_{n_i-m_i} \oplus \frac{c}{12} \sum_{i=1}^n \delta_{m_i,n_i}(n_i^3-n_i)f_ig_i \,.
$$

We can extend the map \eqref{E:vir_unif_p} to  $Vir_{c,n}$ mapping the central element to zero. If we quotient  $Vir_{c,n}$ by the kernel of the map \eqref{E:vir_unif_p}, we obtain an Atiyah algebra $\A_{c,g,n}$ on $B$, where $g$ is the genus of the fiber of $\pi$.

(Let us describe the Kernel of  the map \eqref{E:vir_unif_p}. It consists of two parts. First, it contains the Lie algebra of the group of change of local coordinates. Then, we can attach to every meromorphic vector field on $\cC$ an element of $W^{\oplus n}$ looking at its expansion around $S$ in terms of the local coordinate, and this is also in  the kernel of the morphism \eqref{E:vir_unif_p}. See also \cite[Section 17.3.5]{VA}.)

The Atiyah algebra $\A_{c,g,n}$ acts on the Hodge line bundle $\lambda_g^{\otimes c/2}$, the action is by Lie derivatives  (see e.g. \cite[Section 3]{BS}). This construction is usually known as the \emph{Virasoro uniformization} of the Hodge line bundle, it identifies the log Atiyah algebra of $\lambda_g^{c/2}$ on $B$ with $\A_{c,g,n}$.

This construction is functorial, gives a concrete description of the Atiyah algebra of $\lambda_g^{\otimes c/2}$ over $\M_{g,n}$, and of the log Atiyah algebra of $\A_{\lambda_g^{\otimes c/2}}(-\log(\delta))$ over $\oM_{g,n}$. As $\lambda_g$ is a pull-back from $\oM_g$, $\A_{c,g,n}$ is the pull-back of an Atiyah algebra $\A_{c,g}\cong \A_{\lambda_g^{\otimes c/2}}(-\log \delta)$ from $\oM_g$.

\subsection{Virasoro uniformization of covacua}\label{S:unif} 
The VOA $V$ comes, by definition, with an action of the Virasoro algebra given by the field associated with the conformal vector. Using this action, the algebra $Vir_{c,n}$ acts on $\O_B\otimes V^{\otimes n}$. One can show that the above action induces an action of $\A_{c,g}$ on the bundle of covacua $\mathcal{T}(V)$ over $\oM_g$, and this identifies the log Atiyah algebra $\A_{\mathcal{T}(V)}(-\log(\delta))$ with $\A_{c,g}$ over $\oM_g$. This is studied in detail in \cite{DGT1}. (Observe that this identification is carried out first for families having the property (Q); in particular, it is carried out with families with marked points, i.e. $n>0$. Then one sees that because of propagation of vacua, this is independent of the choice of marked points, and they can be removed.)

The group of line bundles on $\M_g$ is freely generated by $\lambda_g$ \cite{AC87}, therefore, $\mathcal{T}(V)$ must be isomorphic to $\lambda_g^{\otimes k}$ for some $k$ in $\mathbb{Z}$. The identification of the Atiyah algebra of $\mathcal{T}(V)$ with the Atiyah algebra of $\lambda_g^{\otimes c/2}$ implies that $k=c/2$. Using factorization and propagation of vacua, one can see that the isomorphism between $\mathcal{T}(V)$ and $\lambda_g^{c/2}$ extends over $\oM_g$, see e.g. \cite[Section 5.1]{DGT3}.

In particular, we can fix an isomorphism between $\mathcal{T}(V)$ and $\lambda_g^{c/2}$ on $\oM_g$ that is equivariant for the action of $\A_{c,g}$. Since holomorphic functions are constants in $\oM_g$, the isomorphism between $\mathcal{T}(V)$ and $\lambda_g^{c/2}$ is unique up to a scalar; hence, all isomorphisms are equivariant for the action of $\A_{c,g}$. On the bundle of conformal blocks one has the dual construction. This is also discussed in \cite[Section 6]{DGT1}.

\subsection{Differential equation near the boundary}\label{sec:diff_eq}
Loosely speaking, in this section we show that all conformal blocks obtained via the sewing construction of Section \ref{S:sewing} satisfy the same differential equation. As we will explain, the differential equation does not depend on the vertex algebra but only on its central charge. We generalize the argument of \cite[Section 5.3]{Ueno}, and an analogous result is  \cite[Theorem 4.3.6]{Gui}.

Let us first consider the family over $B_0$ from Subsection \ref{S:delta0}. We construct $g$ sections $\ell_k$, for $k\in \{1,\dots , g\}$, of the Atiyah algebra $\A_{c,g}$. Fix $k$, let $q=q_k$. Specialize the other $q_i$'s at generic values, and fix a point of $B$; we obtain a family over a disk $\Delta$ with parameter $q$ (and this family will be used in the proof of Theorem \ref{t:diff_eq}). The family has only one singular member, the fiber over $q=0$; it has only one node, let $X$ be its normalization.

Up to restricting the base, thanks to \cite[Lemma 5.15 and Corollary 5.17]{Ueno}, on $X$ we have a meromorphic vector field $\ell_k$ with the following  properties:
\begin{itemize}
	\item the poles of $\ell_k$ are only at the marked points and the preimage of the node;
\item if we take its expansion around the marked points and see this expansion as an element of $W^{\oplus n}$ as in Section \ref{S:unif_and_kod}, its image via the map \eqref{E:vir_unif_p} is $q\frac{d}{dq}$;
\item  its expansion in the coordinate $z_a$ (resp. $z_b$) we get $\frac{1}{2}w_a\frac{d}{dw_a}$ (resp.  $\frac{1}{2}z_b\frac{d}{dz_b}$). (The local coordinates $z_a$ and $z_b$ are the coordinates on $X$ around the preimages $a$ and $b$ of the node such that the family around the node is locally isomorphic to $z_az_b=q$)
\end{itemize}
Following \cite[Lemma 5.15 and Corollary 5.17]{Ueno}, we also see that $\ell_k$ vary holomorphically in the coordinates of $B_0$. 
We can see the vector field $\ell_k$ described above as an element of the Atiyah algebra $\A_{c,g}$ on $B_0$, and let it act on the sheaf of conformal blocks. We stress the following fact for further reference.
\begin{lemma}\label{uniqueness}
Let $Z$ be the subvariety of $B_0$ defined by $q_1=\cdots = q_g=0$. Let $f$ and $h$ be two holomorphic functions on $B_0$. If $f$ and $h$ are equal when restricted to $Z$ and $\ell_i f= \ell_i h$ for  all $i=1,\dots , g$, then $f=h$ as functions on $B_0$.

\end{lemma}
\begin{proof}
Let $F:=f-g$, we have to show that $F=0$. For what we have said above, see e.g. the reference from \cite{Ueno}, the symbol of $\ell_i$ is $q_i \frac{\partial}{\partial q_i}$, i.e. $\ell_i$ is the degree one operator $q_i\frac{\partial}{\partial q_i}+A$, where $A$ is a holomorphic function on $B_0$. 

We fix an index $i$, let $q:=q_i$, and show that $F=0$, when all variables different from $q_i$ are specialized at generic values. We have $F(0)=0$, and $\frac{dF}{dq}=A(q)F(q)$ for some holomorphic function $A$ of $q$. Then $F=0$.

	\end{proof}

Since we have chosen local coordinates around the marked points, we see the conformal blocks as a sub-sheaf of $V^{\vee}$ tensor and the sheaf of holomorphic functions.  The vector fields $\ell_k$ act on the functions as $q\frac{d}{dq}$; and on $V^{\vee}$ replacing  $z^{n+1}\frac{d}{dz}$ with $L_n$.

Let us now construct the conformal block that will be annihilated by $\ell_k$. If we restrict the family at $q=0$ and normalize the corresponding horizontal node, we obtain a family with three marked points. Given a conformal block $s$ for this family, we can consider the conformal block $s'$ for $\pi$ obtained by the sewing procedure described in Section \ref{S:sewing}. We will show that $\ell_ks'=0$ for every $k$.

The analogous construction holds for the family over $B_i$, $i\neq 0$, from Section \ref{S:deltai}. The only difference is that we have $g+1$ differential operators $\ell_k$ because there are $g+1$ nodes: one corresponding to the parameter $q$ and the other $g$ corresponding to the parameters $q_i$'s. In addition, the analog of Lemma \ref{uniqueness} is true.

\begin{theorem}[Differential equation satisfied by conformal blocks obtained by sewing]\label{t:diff_eq}

With the above notation, let $\ell$ be any of the $\ell_k$'s, then
$
\ell s'=0
$

\end{theorem}

\begin{proof}
We can first restrict ourselves to the family over a disk $\Delta$. We also carry out the proof only in the case $i=0$, and specify the minor twists needed for the case $i\neq 0$.

Let $\nu$ be the conformal vector of $V$. We consider the energy-momentum tensor $T(z)=Y(\nu,z)=\sum L_nz^{-n-2}$.  Let $g(z)$ be a change of coordinate; the Lie algebra of the group of change of coordinates acts via the Virasoro algebra on $V$ (\cite[Section 2.3]{Gui} or \cite[Section 6.3]{VA}); denoting by $R$ this action, we have the following well-known transformation rule obtained by direct computation; see e.g. \cite[Section 8.2.2]{VA}:
\begin{equation}\label{e:trasf_T}
T(z)=R(g)T(g(z))R(g)^{-1}(g'(z))^2+\frac{c}{12}S(g)\Id_V \, ,
\end{equation}
where $S(g)$ is the Schwarzian derivative of $g$. Recall that projective transformations are exactly the functions $g$ such that $S(g)=0$. Up to shrinking $B$, we can assume that we have a projective atlas for $\cC$, i.e. an atlas such that the change of coordinates is projective transformation (this is well-known, a reference is \cite[Section 4 and Theorem 4.1.6]{Gui}, or \cite[Section 8.2, in particular subsection 8.2.12]{VA}). Thus, we can always take $S(g)=0$. This means that, up to the proof of convergence, $T(z)$ is a section of the second power of the canonical bundle valued in $\End(V)$.

We can let $T$ act on conformal blocks, and consider
$$\Omega(z):=s'T(z)\, .$$
For every fixed vector $v$ in $V$, we can then consider $\Omega(z,v)$; this is a formal power series in $z$ and, due to the transformation rule of Equation \eqref{e:trasf_T} with $S(g)=0$, it is valued in the line bundle of holomorphic two forms. Let us show that it converges to a meromorphic two form on $\cC$ with poles just around $x$ (or just around $x_1$ and $x_2$, when $i\neq 0$). 

Since $s'$ is a conformal block, for all meromorphic one form $\omega$ on $\cC$ with poles just around $x$ we have
$$
\Res_{z=0}\left(f(z)\Omega(z,v)dz\right)=0\,.
$$
Where $f(z)dz$ is the local expression of $\omega$ around $x$. We can now deduce the convergence from the Strong Residue Theorem (\cite[Theorem 9.2.9]{VA}, \cite[Lemma 1.1.6]{Ueno} or \cite[Section 1.4]{Gui} ).

Let $\ell(z)$ be a meromorphic vertical vector field on $\cC$, or just the Laurent expansion of a vector field around a marked point $x$ (or around $x_1$ and $x_2$, when $i \neq 0$). Following e.g. \cite[page 79]{Ueno},we introduce the operator
$$D[\ell](z):=\Res_{z=0}\left(\ell(z)T(z)dz^2 \right)$$
This is a formal power series in $z$ whose coefficients are in the Virasoro algebra.  In practice, we are just replacing  $z^{n+1}\frac{d}{dz}$ with $L_n$. In other words, $D[\ell](z)$ is how the vector field $\ell$ acts on the space of conformal blocks.

Let $\gamma$ be a small circle around $x(0)$. We have
\begin{equation}\label{eq:res_at_x=0}
\int_{\gamma}\ell(z)s'T(z) dz^2=D[\ell]s'
\end{equation}

The one form $\ell(z)\Omega(z)$ can be regarded as a form on $U\subset X$. Its integral around $x$ has been computed in Equation \ref{eq:res_at_x=0}. Let us compute its residue at $a$, and hence let $\gamma_a$ be a small circle around $a$.  Using the special form of $\ell(z)$ near $a$, we get
$$\int_{\gamma_a}\ell(w_a)\Omega(w_a)(dw_a)^2= \frac{1}{2}\int_{\gamma_a} w_i\frac{d}{dw_a}s'T(w_a)dw_a^2=\frac{1}{2}L_0s'$$
where the last inequality is because $\frac{d}{dw_a}$ gets contracted with one $dw_a$, and the coefficient of $w_a^{-1}$ in the power series $w_aT(w_a)$ is $L_0$. To compute $L_0s'$ we have to go through the sewing construction. As we compute the residue at $a$, we consider the action on the copy of $V^{\vee}$ corresponding to the point $a$. For an integer $d$, the coefficients of $q^d$ in $s'$ are a basis of the degree $d$ part $V^{\vee}$ , so
$$L_0s'=q\frac{d}{d q}s'$$

The analog computation is around $b$.
\medskip

On $X$, by the classical Residue Theorem,  the sum of the residues of $\ell(z)\Omega(z)$ around $a$ and $b$ is the opposite of the residue of $\ell \Omega$ at $x$ (or at $x_1$ and $x_2$, when $i\neq 0$); we obtain the final results combining the explicit computations of these residues carried out above.

\end{proof}

\section{Relation between partition functions and Teichm\"{u}ller modular forms }\label{S:proof1}

\begin{definition}\label{def:vuoto_uniformizzato}
Let $V$ be a holomorphic vertex algebras of central charge $c$.  Let $u_{V,g}$ be  the unique up to a multiplicative scalar $\A_{g,c}$-equivariant isomorphism between $\T_g(V)$ and $\lambda_g^{c/2}$ over $\oM_g$ (compare to \ref{S:unif}). We will normalize $u_{V,g}$ as explained in Definition \ref{Normalization}.

We will consider the image of the vacuum section under this isomorphism; we should denote it by $u_{V,g}(\oOmega_{V,g})$, to shorten the notations we will denote it just by $u(\oOmega_{V,g})$.

The pull-back of $u_{V,g}$ to the Teichm\"{u}ller space gives a canonical up to a multiplicative scalar isomorphism between $\T_g$ and $L_g^{c/2}$ as $\Gamma_g$ and $\A_{g,c}$ line bundles; by abuse of notations, we still denote it by $u_{V,g}$. Again by abuse of notations, we denote by $u(\oOmega_{V,g})$ the image of the vacuum section, which is a Teichm\"{u}ller modular forms of weight $c/2$. Observe that, on $T_g$, $u_{V,g}$ is a section of the conformal blocks $\T_g^*$ because $L_g$ is trivial. 

Similarly, we also have $u_{V,g}$ and $u(\oOmega_{V,g})$ on the Schottky space $S_g$ and the marked Schottky space $C_g$, and $u(\oOmega_{V,g})$ is a Teichm\"{u}ller modular forms of weight $c/2$ over these spaces.

\end{definition}

\begin{definition}[Normalization of $u_{V,g}$]\label{Normalization}  We will adopt the following normalization. Let $Z\subset \oM_g$ be the locus of curves whose normalization is $\P^1$. The bundle of covacua blocks $\T_g(V)$ is canonically isomorphic to the trivial bundle with fiber $V_0$; under this isomorphism the vacuum section is the dual of the vacuum vector. On the Schottky space, approaching $Z$ is equivalent to letting all the $q_i's$ tend to zero. We can thus normalize $u$ in such a way that if we let $q_i=0$ for every $i$ in $u(\oOmega_{V,g})$ we obtain $1$.

\end{definition}

\begin{theorem}\label{T:main1}
With the notations of Definition \ref{def:parition_functions}, \ref{def:open_in_Schottky} and \ref{def:vuoto_uniformizzato}, for every integer $g\geq 1$ and every $r\in \mathbb{R}_{>0} $, we have
\begin{enumerate}
	\item\label{T:conv} the formal Laurant power series  $\cZ_{V,g}$ converges on $U_{g,r}^+$, and can be analytically continued to a holomorphic function on all $U_{g,r}$;
	\item\label{T:modularity}  there exists a unique non-zero meromorphic function $F_g$ on $U_{g,r}$ such that for every holomorphic vertex algebra $V$ of central charge $c$, the product 
	$F_g^{c/8}\cZ_{V,g}$ 
	is the weight $c/2$ Teichm\"{u}ller modular form $u(\oOmega_{V,g})$ from Definition \ref{def:vuoto_uniformizzato} restricted to $U_{g,r}$.
\end{enumerate}

\end{theorem}

\begin{proof}

(The following proofs work also when $g=1$; in this case, one has to replace $\M_1$, $\oM_1$, and $T_1$ with $\M_{1,1}$, $\oM_{1,1}$, and $T_{1,1}$, see Remark \ref{rem:g=1}.)

Let $(\, , \,)$ be the unique normalized invariant bilinear form on $V$. First, consider the trivial family $\P^1\times B$ with one constant marked point, i.e. zero. From Example \ref{ex:covacP1}, we can see that the sheaf of conformal block is one dimensional, it is the trivial line bundle on $B$, and it has a basis given by the linear function 
$$F_1(a)=(\Omega^V, a)$$

Now consider the trivial family $\P^1\times B$  with $2g+1$  distinct marked points (i.e. $2g+1$ sections $z_i\colon B \to \mathbb{P}^1$), where the last is constantly zero, and the others are non-constant and avoid $\infty$. Following Example \ref{ex:covacP1_tantipunti} for every point $b$ of $B$, the correlation function
\begin{equation}\label{E:cor2}
F_{2g+1}(a_1 \otimes\cdots \otimes a_{2g}\otimes a)=\left(\Omega^V,Y(a_1,z_1(b))\cdots Y(a_{2g},z_{2g}(b))a)\right) \,,
\end{equation}
is a basis of the fiber of the line bundle of conformal blocks at $b$. Basic properties of these functions are discussed in \cite[Section 3.5]{Axiomatic}, see also \ref{T:commutation} and \cite[Introduction]{Gui24}. In particular, these functions converge on the open subset $B^+$ of $B$ defined by conditions $|z_1(b)|>|z_2(b)|> \cdots > |z_{2g+1}(b)|$ (by a standard abuse of notation, $z_i(b)$ is both a point of $\P^1\setminus \infty$ and the complex number obtained evaluating the standard coordinate on $z_i(b)$), and extends to a holomorphic function on all $B$. 

For later use, it is convenient to relabel the sections as $w_1,\dots ,w_g$, $z_1,\dots , z_g$ and $x$. Composing with the standard coordinate on $\P^1$, we can think of these sections as functions on $B$, we will assume that $w_1,\dots ,w_g$, and $z_1,\dots , z_g$ are coordinates on $B$, even though this is not strictly necessary.

Consider now a family $\pi \colon \cC\to B_0=B\times \Delta$ from Subsection \ref{S:delta0}. Applying  the Sewing Theorem from Section \ref{S:sewing} to the conformal blocks $F_{2g+1}$ described in Equation \eqref{E:cor2}, we obtain a conformal block $\mu_{V,g}$ on $B\times \Delta$.

If we apply $\mu_{V,g}$ to the vacuum section $\oOmega_{V,g}$ we obtain a holomorphic function on $B\times \Delta$.

On the open subset $B^+$ of $B$ defined above, $F_{2g+1}$ converges to a function. Thus, on $B^+\times \Delta$, we can compute the expansion of $\mu_{V,g}\left(u(\oOmega_{V,g})\right)$ using the expansion of $F_{2g+1}$ and the expansion of $\mu_{V,g}$ given in Equation \eqref{E:sewing} in the variables  $w_i$'s, $z_i$'s and $q_i$'s. To this end, we have to start from $F_{2g+1}$, plug $a=\Omega^V$, and for the other $a_i$'s follow the receipt of Section \ref{S:sewing}. The final result is the formal power series $\cZ_{V,g}$ from Definition \ref{def:parition_functions}! This thus is convergent on $B^+\times \Delta$ and can be continued to the holomorphic function  $\mu_{V,g}\left(u(\oOmega_{V,g})\right)$ on $B\times \Delta$.

Let the space $U_{g,r}$ (Definition \ref{def:open_in_Schottky}) be covered by chart of the form $B_0=B\times \Delta^*$ as above. Since on each chart the formal power  $\cZ_{V,g}$ series defines a function, meaning that in each chart there is an open subset where it converges and can be extended to a holomorphic function on all $B\times \Delta^*$, we conclude that  $\cZ_{V,g}$ converges on an open subset of $U_{g,r}^+$ and extends to a holomorphic function on all $U_{g,r}$. Because of what we have said above, the open subset where it converges contains $U_{g,r}^+$ from Definition \ref{def:open_in_Schottky}. This concludes the proof of item \ref{T:conv}.

 The space of global sections of $\T_g^*$ is a free $\O_{U_{g,r}}$-module of rank one, where $\O_{U_{g,r}}$ is the sheaf of holomorphic functions on $U_{g,r}$, so there exists a unique non-zero meromorphic function $F_V$ on $U$ such that $F_V\mu_{V,g}$ is equal to the restriction of $u_{V,g}$ from Definition \ref{def:vuoto_uniformizzato} to $U$. We are going to show that this function $F_V$ does not depend on $V$ but just on its central charge $c$. 

Let $U$ be another holomorphic VOA with the same central charge of $V$; denote by $\mu_{U,g}$ and $u_{U,g}$ the sections of $\T^*_g(U)$ constructed with the above procedures.

Following \ref{S:unif}, the line bundles $\T_g(V)$ and $\T_g(U)$ have the same Atiyah algebra $\A_{g,c}$, and there is a unique up to a multiplicative scalar isomorphism $\phi$ on $\oM_g$ between  $\T^*_g(V)$ and $\T^*_g(U)$ which is compatible with its action. We normalize it asking that its restriction to the compact locus $Z$ from Definition \ref{Normalization}, where  $\T^*_g(V)$ (resp. $\T^*_g(U)$) is canonically isomorphic to $V_0^{\vee}$ (resp. $(U_0)^{\vee}$) maps the dual of the vacuum vector of $V$ to the dual of the vacuum vector of $U_0$.

Restricting $\phi$ to $\M_g$, and then pulling it back to $U$, we obtain an  isomorphism, which we keep on calling $\phi$, between $\T^*_g(V)$ and $\T^*_g(U)$ which commutes with the action of $\A_{g,c}$, and such that $\phi \circ u_{V,g}$ is equal up to a scalar to $u_{U,g}$. We will show that 
$\phi(\mu_{V,g})=\mu_{U,g}$. As they are analytic sections, it is enough to do it on a chart of the form $B_0=B\times \Delta^*$ from Subsection \ref{S:delta0}.

The sections $\mu_{U,g}$ and $\phi(\mu_{V,g})$ satisfy the same $g$ differential equations from Theorem \ref{t:diff_eq}. We want to apply Lemma \ref{uniqueness} so we are going to show that they are equal on the locus defined by $q_1=\cdots = q_g=0$. To this end, extend the family over $B\times \Delta$, the $2g$ dimensional locus will be $B\times \{(0,\dots , 0)\}$. There, we have $\mu_{U,g}=\phi(\mu_{V,g})$ because of the definition of $\mu$ and the normalization we have chosen for $\phi$.

Putting everything together, we get
$$F_{U}\mu_{U,g}=u_{U,g}=\phi(u_{V,g})= \phi(F_V\mu_{V,g})=F_V\phi(\mu_{V,g})=F_V\mu_{U,g} \,,$$
where all equalities are true up to a non-zero constant scalar factor. We conclude that $F_V\mu_{U,g}=F_{U}\mu_g'$, hence  $F_V=F_{U}$. 

\medskip

So far, we have shown that there exists a meromorphic function $F_{g,c}$ such that $F_{g,c}\cZ_{V,g}$ is a weight $c/2$ Teichm\"{u}ller modular form for every vertex algebra $V$ of central charge $c$. This would already be enough for our applications. Let us further show that we can take $F_{g,c}=F_{g,8}^{c/8}$, for a convenient $F_{g,8}$. First, observe the following useful relation: given two vertex algebras $U$ and $V$, we have
\begin{equation}\label{E:somma}
\cZ_{U\otimes V,g}=\cZ_{U,g}\cZ_{V,g} \, ,
\end{equation}
cf. Remark \ref{remarkProductFormal}.

Through the tensor product isomorphism of Section \ref{S:tens}, we have
\begin{equation}\label{E:somma2}
\mu_{U\otimes V,g} =\mu_{U,g}\otimes \mu_{V,g} \quad \textrm{and} \quad u_{U\otimes V,g} =u_{U,g}\otimes u_{V,g}\,.
\end{equation}

Let now a $V_{E_8}$ be the unique central holomorphic vertex algebra of central charge 8, and take $F_{g,8}$ to be a meromorphic function such that $F_{g,8}\cZ_{V_{E_8},g}=u_{V_{E_8},g}$. Let $V=V_{E_8}^{\otimes c/8}$. We then have that $F_{g,8}^{c/8}\mu_{V,g}=u_{V,g}$.  We have shown that if a meromorphic function realizes that equality for a specific vertex algebra, it realizes it for every vertex algebra of the same central charge, hence we can take $F_{g,c}=F_{g,8}^{c/8}$, and let $F_g=F_{g,8}$. 

For the uniqueness  of $F_g$, observe that $F_g=u_{V_{E_8},g}/\mathcal{Z}_{g,V_{E_8}}$

This concludes the proof of Theorem \ref{T:main1}.

\end{proof}

Let us recall from  \cite[Section 7.2]{TWgeq26} the following explicit description of $\cZ_{V_L,g}$, where $V_L$ is the lattice vertex algebra associated to the positive definite even lattice $L$.

\begin{theorem}[Partition function of lattice VOA]\label{thm:partition_of_lattice}
With the notations of Section \ref{S:Sieg}, for a lattice VOA $V_L$, $\cZ_{V_L,g}$, as a formal series, is the pull-back of the degree $g$ theta series $\Theta_{L,g}$ associated to $L$ via the Jacobi map, times the $c$-th power $( \cZ_{M(1),g} )^c$ of the genus $g$ partition function $\cZ_{M(1),g}$ of the rank-one Heisenberg vertex algebra $M(1)$.
\end{theorem}

For $g=1$, Theorem \ref{thm:partition_of_lattice} was well-known and, after pulling back to the upper half-plane, reads
$$
\cZ_{V_L,1}(\tau)=\left(\prod_{n=1}^{\infty}(1-e^{2n \pi i \tau})\right)^{-c}\Theta_{L,1}( \tau),
$$
$\Im \tau > 0$, cf. Remark \ref{RemarkGenus1Partition}.

Theorem \ref{thm:partition_of_lattice} gives an identity of formal power series, so the convergence of $\cZ_{M(1),g}$ on $U_{g,r}^+$ follows from the well-known convergence of theta series, and the convergence of partition functions proven in Theorem \ref{T:main1}; see also \cite{MT10,TZ11,Tui21,TWgeq26}.

\begin{remark}\label{rem:part_latt}
Theorem \ref{thm:partition_of_lattice} together with \cite[Corollary 1.5]{CS-B14} implies that the partition functions of holomorphic lattice VOA are different for $g$ big enough. In particular, the partition functions of the VOA associated to the lattices $D_{16}^+$ and $E_8^{\oplus 2}$ are different if and only if $g\geq 5$, see \cite[page 2]{CS-B14} and \cite{GruSa}.
\end{remark}

The following result combines Theorem \ref{T:main1} and Theorem \ref{thm:partition_of_lattice}.
\begin{proposition}\label{P:ampl}
In  Theorem \ref{T:main1}, we can take $F_g$ holomorphic. Let $H_g$ be another meromorphic function such that $H_g \cZ_{V,g}$ is a Teichm\"{u}ller modular form of weight $k$ for every holomorphic vertex algebra $V$ of central charge $c$. Then $H_g=\phi F_g$, where $\phi$ is a holomorphic Teichm\"{u}ller modular form of weight $k-c/2$. Moreover, $F_g = (\cZ_{M(1),g})^{-8}$; in particular, for $g=1$, we have $F_1(\tau)=\left(\prod_{n=1}^{\infty}
(1-e^{2n\pi i \tau })\right)^8$.
\end{proposition}
\begin{proof}
  
Given $g$, as explained in  Section \ref{S:Sieg}, we can fix a $c$ such that for every point $\tau$ of $\H_g$, there exists a positive definite even unimodular lattice $L$ of rank $c$ such that $\Theta_{L,g}(\tau)\neq 0$. Since both $F_g^{c/2}\cZ_{V,g}$ and $\cZ_{V,g}$ are holomorphic for every $V$, and for lattice VOA by Theorem \ref{thm:partition_of_lattice} $\cZ_{V,g}$ is related to $\Theta_{L,g}$, varying $V$ among all lattice vertex algebras we conclude that $F_g$ cannot have poles, and hence it is holomorphic.

Take $H_g$ as in the statement. Then $H_g\cZ_{V_L,g}=H_gF_g^{-1}\Theta_{L,g}$ is a weight $k$ holomorphic modular form for all positive definite even unimodular lattices $L$. This means that $H_gF_g^{-1}=:\phi$ is a modular form of weight $k-c/2$, possibly with poles. Using again the fact that theta series have no common zero for $c$ large enough, we conclude that $\phi$ is holomorphic.

Degree $g$ theta series are well known to be Teichm\"{u}ller modular forms of weight $c/2$, so, by Theorem \ref{thm:partition_of_lattice}, we can take as $F_g$ the $8$-th power of the genus $g$ partition function of the Heisenberg vertex algebra.

\end{proof}

\section{Expansion around the boundary divisors}
\label{SecBoundary}

 \begin{theorem}\label{thm:altre_espansioni}
 Using the notation of Subsection \ref{S:delta0}, Section \ref{S:deltai}, Definition \ref{def:parition_functions} and \ref{def:vuoto_uniformizzato}. Fix integers $c$ and $g$, then for every $i\in \{0, \dots , \lfloor\frac{g}{2} \rfloor\}$ 
\begin{enumerate}
    \item the formal power series $\cZ_{V,g,i}$ converges on the open subset $B_i^+$ of $B_i$ defined by the inequalities $|w_1|>|z_1|>\cdots >|w_i|>|z_i|>|w|$ and $|z|>|w_{i+1}|>|z_{i+1}|>\cdots >|w_g|>|z_g|$; moreover, it extends to a holomorphic function on all $B_i$;
\item there exists a trivialization $\phi_i$ of the pull-back of the Hodge line bundle $\lambda^{\otimes c/2}$ to $B_i$, such that for every holomorphic vertex algebra $V$ of central charge $c$, on $B_i^+$ we have
$$\phi_i(u(\oOmega_{V,g}))=\cZ_{V,g,i}$$\,.
 \end{enumerate}

 \end{theorem}
\begin{proof}

Fix $V$. We first construct a $\phi_i$ as in the statement that a priori depends on $V$ (so temporarily we call it $\phi_{i,V}$).  

For $i=0$, it is $u_{V,g}^{-1}$ composed with the conformal block $\mu_{V,g}$ from the proof of Theorem \ref{T:main1}. For $i\neq 0$, the construction is similar. First, we construct a conformal block $\mu_{V,g,i}$ such that $\mu_{V,g,i}(\oOmega_{V,g})=\cZ_{V,g,i}$. On the locus where all $q_i$'s and $q$ are zero, this conformal block is given in Example \ref{ex:covacP1_tantipunti}. We then extend it to $B_i$ using the Sewing of conformal blocks  from the last part of  Section \ref{S:sewing} . If we apply this conformal block to the vacuum section, we have to plug $A=\Omega^V$, and for the other $A_i$'s we follow the receipt of Section \ref{S:sewing}, and we obtain on an open subset the power series of the partition function in Definition \ref{def:parition_functions}. Then $\phi_{i,V}$ is the composition of $u_{V,g}^{-1}$ with $\mu_{V,g,i}$. The convergence properties of $\cZ_{V,g,i}$ are the same of the convergence properties of $\mu_{g,i}(\oOmega_{V,g})$, so we have the first claim.

Now, we have to show that $\phi_{i,V}$ does not depend on $V$. We argue in the same spirit of the proof of Theorem \ref{T:main1}, in particular we will use the differential equations from Theorem \ref{t:diff_eq}, and apply Lemma \ref{uniqueness}. Let $U$ be another VOA. Let $\psi$ be the composition of $u_V$ and the inverse of $u_{U}$, so, up to dualizing, this is an $\A_{g,c}$-equivariant isomorphism between $\T^*(V)$ and $\T^*(U)$. We have to show that $\psi(\mu_{V,g})=\mu_{U,g}$. To this end, observe that both are equal on the locus where $q$ and all the $q_i$'s are zero, and then they satisfy the same differential equations from Theorem \ref{t:diff_eq}.

\end{proof}

\begin{corollary}\label{confronto_espansioni}
	Fix $g$ and $i\in \{0,\dots , \lfloor\frac{g}{2}\rfloor\}$. Let $V$ and $U$ be two holomorphic VOA of the same central charge $c$. If
	$$
	\cZ_{V,g,i}=\cZ_{U,g,i}
	$$
	Then for every $j$ in  $\{0,\dots , \lfloor\frac{g}{2}\rfloor\}$ we have
	$$
	\cZ_{V,g,j}=\cZ_{U,g,j}
	$$
\end{corollary}
\begin{proof}
	Using notations and results from Definition \ref{def:vuoto_uniformizzato} and \ref{thm:altre_espansioni}, since $u(\oOmega_{V,g})$ and $u(\oOmega_{U,g})$ are holomorphic sections of the same line bundle, if $\cZ_{V,g,i}=\cZ_{U,g,i}$, then $u(\oOmega_{V,g})=u(\oOmega_{U,g})$, hence $\cZ_{V,g,j}=\cZ_{U,g,j}$.
\end{proof}

\section{The slope of the moduli space and the uniqueness of the moonshine VOA}
\label{SecSlope}

When $g$ is large, the geometry of the moduli space $\oM_g$ is rather complicated and has been extensively studied. An important quantitative invariant is the  ``slope of the effective cone".   It is often called just  ``the slope of $\oM_g$", and denoted by $s_g$, two surveys are \cite{Far,Far2}. Let us briefly describe it.

Given a non-zero Teichm\"{u}ller modular form $f$, we define its slope as the ratio
$$
s(f):=\frac{\textrm{weight of } f}{\textrm{vanishing order of } f \textrm{ along } \delta}
$$
By vanishing order of $f$ along $\delta$, we mean the vanishing order of the section of $\lambda_g^{(\otimes (\textrm{weight of } f))}$ on $\oM_g$ associated to $f$. The slope $s_g$ is the infimum of the slopes of all non-zero Teichm\"{u}ller modular forms.

\medskip

Let us show that the slope is related to the classification of vertex algebra, in particular with the problem of the uniqueness  of the moonshine vertex algebra $V^{\natural}$.

\begin{corollary}\label{cor:classification1}
If, for a given integer $\overline{g}$, $s_g>6$ for every $g \leq \overline{g}$, then $$\cZ_{V,g,i}=\cZ_{V^{\natural},g,i}$$ for every $g\leq \overline{g}$, every $i$, and every holomorphic vertex algebra $V$ of central charge $24$ and trivial $V_1$.
\end{corollary}
\begin{proof}
Thanks to Theorem \ref{T:main1}, the difference $D_g:=F_g^{3}\left(\cZ_{V,g,0}-\cZ_{V^{\natural},g,0}\right)$ is a Teichm\"{u}ller modular form of weight 12 (or, equivalently, we can look at $D_g=u(\oOmega_{V,g})-u(\oOmega_{V^{\natural},g})$, which is a section of $\lambda_g^{\otimes 12}$ as described in Definition \ref{def:vuoto_uniformizzato}). We show by induction on $g$ that $D_g$ is constantly equal to zero for $g\leq \overline{g}$. 

For $g=1$ the result is well-known: $D_1$ is a modular form of weight 12 with a zero at $\{q=0\}$ of order at least two (see \ref{RemarkGenus1Partition}), and hence by the classical theory of modular forms it is zero.

To prove that $D_g$ is zero, we will show that its slope $s(D_g)$ is less than or equal to 6. To this end, we compute the vanishing order of $D_g$ along $\delta_i$, for every $i$, and we show it is at least equal to $2$.

To carry out this computation, we pull-back $D_g$ to the basis $B_i$ described in Section \ref{S:delta0} and Section \ref{S:deltai}. The expansion of $D_g$ is given by  \ref{thm:altre_espansioni}. The zero order term is zero because it is equal to $D_{g-1}$ and we can apply induction. The degree one term is trivially zero because $V_1=V^{\natural}_1=\{0\}$.

\end{proof}

The exact value of $s_g$ is known only for low values of $g$. In particular, for 
$g\leq 12$, $g=15$ and $g=16$ \cite{Tseng}, we have $s_g>6$. The recent upper bounds on $s_{22}$ and $s_{23}$ are still greater than six \cite{FP}

For arbitrary values of $g$, the first upper bound was proven by Harris and Morrison \cite{HS90}, they proved that $s_g\leq 6 + 12(g+1)^{-1}$, and conjectured that this upper bound was sharp. This went under the name of slope conjecture and was disproved by Farkas and Popa in \cite{Farkas_Popa_slope}. Still, all known upper bounds known by now are of the form $6+O(1/g)$. On the other hand, known lower bounds are of the form $O(1/g)$.

In 2009, experts in the field made a by now famous dinner bet on whether the liminf of the sequence $\{s_g\}$ is zero or six \cite[page 2]{Far2}. The following conjecture is a variant of the hypotheses suggested by experts supporting the six, which is consistent with all recent results.

\begin{conjecture}[Weak Harris-Morrison slope conjecture]\label{conj:slope}
For every $g$, we have $s_g>6$.
\end{conjecture}

The above conjecture, together with Corollary \ref{cor:classification1}, gives the following uniqueness result.

\begin{corollary}\label{cor:classification2}
Assume that Conjecture \ref{conj:slope} is true, then any holomorphic vertex algebra $V$ of central charge $24$ and trivial $V_1$ has the same partition functions of the moonshine vertex algebra $V^{\natural}$. In symbols, $\cZ_{V,g,i}=\cZ_{V^{\natural},g,i}$ for every $g$ and every $i$.
\end{corollary}

In view of Corollary \ref{cor:classification1}, Corollary \ref{confronto_espansioni}, and of the difficulty in computing the value of $s_g$ for $g$ large, we ask for the following effective version of the reconstruction Conjecture \ref{Conj:reconstruction}.

\begin{question}
Fix a positive integer $c$. Does there exist an effectively computable number $N(c)$ such that for every pair of self-dual VOA's  of CFT type $V$ and $U$ of central charge $c$ if, for some $i$, $\cZ_{V,g,i}=\cZ_{U,g,i}$ for every $g\leq N(c)$ then $V$ is isomorphic to $U$?
\end{question}
The above question is already interesting if one restricts the attention of holomorphic vertex algebras, $c=24$ and one of the VOA's is the Moonshine.

\section{Linear independence of partition functions}
\label{SecLinearIndependence}

This section is loosely inspired by \cite{Segal}. Fix an even positive integer $c$. By  $\RR(\oM_g,\lambda_g^{\otimes c/2})$ we denote the graded algebra $\bigoplus_{n\geq 0}H^0(\oM_g,\lambda_g^{\otimes nc/2})$. Consider the gluing morphism $G_0\colon \oM_{g-1,2}\to \oM_g$, and the forgetful map $F\colon \oM_{g-1,2}\to \oM_{g-1}$ from Section \ref{S:moduli} (here we forget two points rather than one with $F$). Since $G_0^*\lambda_g=F^*\lambda_{g-1}$, the projection formula, which is an algebraic geometry variant of the integration along fibers, gives a morphism of graded algebras
\begin{equation}\label{eq:sigma}
\sigma_g \colon \RR(\oM_g,\lambda_g^{\otimes c/2})\to \RR(\oM_{g-1},\lambda_{g-1}^{\otimes c/2})\,.
\end{equation}

We can thus define the graded algebra of \textbf{weight $c/2$ stable Teichm\"{u}ller modular forms} as inverse limit
$$\RR(\oM_\infty,\lambda_\infty^{\otimes c/2}):=\lim_g  \RR(\oM_g,\lambda_g^{\otimes c/2}) $$
(The space $\oM_\infty$ has been seldom considered in the literature, see however \cite{Cod_hyp,Odaka}. In this paper we use  $\oM_\infty$ just as a suggestive symbol, the only mathematical object we care about here is the algebra of stable Teichm\"{u}ller modular forms)

Similarly, the gluing morphism $G_{g-h,h}\colon \oM_{g-h,1}\times \oM_{h,1}\to \oM_g$ gives rise to a map

$$
\sigma_{g,h}\colon \R(\oM_g,\lambda_g^{\otimes c/2})\to \RR(\oM_h,\lambda_h^{\otimes c/2})\otimes \RR(\oM_{g-h},\lambda_{g-h}^{\otimes c/2})
$$
we observe that $\sigma_g$ is compatible with $\sigma_{g,h}$ for every $g$ and $h$, and $\sigma_{g,h}=\sigma_{g,g-h}$, hence we obtain a co-commutative co-multiplication
$$
\Delta \colon \RR(\oM_\infty,\lambda_\infty^{\otimes c/2}) \to \RR(\oM_\infty,\lambda_\infty^{\otimes c/2}) \otimes \RR(\oM_\infty,\lambda_\infty^{\otimes c/2})
$$

 It is also possible to define a co-unit so that $\RR(\oM_\infty,\lambda_\infty^{\otimes c/2})$ acquires the structure of a co-algebra compatible with the structure of algebra, i.e. a bialgebra structure; we refrain from doing it as we do not need it.

\medskip

Given a holomorphic vertex algebra $V$, using the properties of covacua from Sections \ref{S:prop} and \ref{S:factorization}, we can carry out the analog construction for the ring 
$$\RR(\oM_\infty,\T(V)):=\lim_g \bigoplus_n H^0(\oM_g,\T_g(V)^{\otimes n})\,,$$ 
obtaining a co-commutative co-multiplication. 

We denote by $\oOmega_V$ the collection of vacuum sections $\{\oOmega_{V,g}\}_{g\geq 0}$ (cf.  \ref{SS:vacuumSec} for the definition). This collection defines an element $\oOmega_V$ of $\RR(\oM_\infty,\T(V))$, which is a group-like element, i.e. $$\Delta(\oOmega_V)=\oOmega_V\otimes \oOmega_V\,.$$ 
These facts follow from the properties of the vacuum section, see \ref{SS:vacuumSec}. 

The maps  from Definition \ref{def:vuoto_uniformizzato}, give an isomorphism of co-algebras
$$
u_V\colon \RR(\oM_\infty,\T(V))\to \RR(\oM_\infty,\lambda_\infty^{\otimes c/2})\,,
$$
where $c$ is the central charge of $V$. We thus obtain a group-like element $u_V(\oOmega_V)$ of $\RR(\oM_\infty,\lambda_\infty^{\otimes c/2})$, which, to shorten the notations, write as $u(\oOmega_V)$. Finite sets of group-like elements are known to be linearly independent \cite[Prop. 3.2.1]{SweedlerHopf}. Thus, we have the following result.

\begin{theorem}[Linear independence of partition functions]\label{T:indip}

Let $V_1,\dots , V_k$ be holomorphic vertex algebras of the same central charge $c$ whose partition functions are different. Then, for all $g$ large enough
\begin{enumerate}
\item  their vacuum sections $u(\oOmega_{V_1,g}),\dots , u(\oOmega_{V_k,g})$ are linearly independent elements of $H^0(\oM_g,\lambda_g^{c/2})$;

\item  their partition functions $\cZ_{V_1,g,i},\dots ,\cZ_{V_k,g,i}$ are linearly independent power series for every $i$.

\end{enumerate}
\end{theorem}
\begin{proof}
The stable Teichm\"{u}ller modular forms $u(\oOmega_{V_1}),\dots , u(\oOmega_{V_k})$ are linearly independent because they are group like elements. This means that for all $g$ big enough the vacuum sections $u(\oOmega_{V_1,g}),\dots , u(\oOmega_{V_k,g})$ are linearly independent in $H^0(\oM_g,\lambda_k^{\otimes k})$, where $k$ is half of the central charge. Applying the conformal block $\mu_{g,i}$ introduced in the proof of \ref{thm:altre_espansioni} to these sections, we obtain the second statement.
\end{proof}

The following is an intriguing corollary, as there are few results about the dimension of the space of Teichm\"{u}ller modular forms.

\begin{corollary}\label{cor:71}
For all $g$ large enough, we have $\dim H^0(\oM_g,\lambda_g^{\otimes 12})\geq 71$.
\end{corollary}
\begin{proof}
There are at least 71 holomorphic vertex algebras with central charge 24 and distinct partition functions, see Section \ref{sec: c=24}, in particular \ref{th: Schellekens}, so the result follows from Theorem \ref{T:indip}.
\end{proof}

Recall that, for $g$ big enough, the space of weight 12 Siegel modular forms is 24 dimensional, see Section \ref{S:Sieg}. In particular, Corollary \ref{cor:71} says that not all weight 12 Teichm\"{u}ller modular forms are pull-back of Siegel modular forms. 

To the best of our knowledge, so far there were only two known Teichm\"{u}ller modular forms which are not pull-back of Siegel modular forms: one is in genus 3 \cite[Introduction]{Ich}, and one in genus 4 (the determinant of the differential of the Schottky form, see for instance \cite[Section 3]{Mat_Volp}). In both cases, their square is the pull-back of a Siegel modular form. 
 Proposition \ref{prop:comp} gives a general recipe to produce Teichm\"{u}ller modular forms which are not pull-back of Siegel modular forms. Proposition \ref{power} shows that,  given an integer $d$, for all $g$ big enough there are Teichm\"{u}ller modular forms whose $d$-th power is not the pull-back of Siegel modular forms. We are unable to provide an example of Teichm\"{u}ller modular form such that none of its powers is the pull-back of a Siegel modular form.

\begin{proposition}[Comparison between Siegel and Teichm\"{u}ller modular forms]\label{prop:comp}
Fix a positive integer $k$ divisible by $4$. Let $m$ be the number of isomorphism classes of unimodular, positive definite, even lattices of rank $2k$. Let $V_1,\dots , V_n$ be unitary holomorphic vertex algebras of central charge $2k$ which are not lattice VOA and have distinct partition functions.  

Then, for all $g$ big enough, the dimension of the space of weight $k$ and degree $g$ Siegel modular forms is $m$, and the dimension of the space of weight $k$ and degree $g$ Teichm\"{u}ller modular forms is at least $m+n$. 

In particular, if $k\geq 12$, for all $g$ big enough the space of weight $k$ and degree $g$ Teichm\"{u}ller modular forms has dimension strictly bigger than the space of weight $k$ and degree $g$ Siegel modular forms. 
\end{proposition}
\begin{proof}
The statement about the dimension of the space of Siegel modular forms is well known, see Section \ref{S:Sieg}. 

Let $W_1,\dots , W_m$ be the holomorphic lattice VOA of central charge $2k$; their partition functions are theta series (Theorem \ref{thm:partition_of_lattice}), and they are all distinct for $g$ big enough by \cite[Corollary 1.5]{CS-B14}. The partition functions of $V_1,\dots , V_n$ are not theta series by Theorem \ref{thmLatticePartition} (to apply this result we need the unitary assumption), and are  different by assumption.

We conclude that the dimension of the space of Teichm\"{u}ller modular forms is at least $m+n$ by Theorem \ref{T:indip}.

For the last statement, it is enough to produce a unitary holomorphic VOA $V$ of central charge $2k$ which is not a lattice a VOA. Plausibly, there are plenty of them and many different ways to construct them, let us propose one construction. Let $W$ be unitary holomorphic VOA of central charge 24 which is not a lattice VOA. Let $V_{E_8}$ the holomorphic lattice VOA associated to $E_8$, it has central charge $8$. For a pair of integers $a>0$ and $b\geq 0$ such that $2k=24a+8b$, we take $V:=W^{\otimes a} \otimes V_{E_8}^{\otimes b}$. Then $V$ is a holomorphic VOA of central charge $2k$. The inductive application of Proposition \ref{PropTensorLattice} shows that $V$ is not a lattice VOA.
\end{proof}

\begin{proposition}\label{power}
Let $V$ be a holomorphic VOA which is not a lattice VOA. Fix a positive integer $d$. Then, for all $g$ large enough (depending on $V$ and $d$), the Teichm\"{u}ller modular forms $u(\oOmega_{V,g})^d$ are not a pull-back of a Siegel modular forms.
\end{proposition}
\begin{proof}
Let $c$ be the central charge of $V$. We have $u(\oOmega_{V,g})^d=u(\oOmega_{V^{\otimes d},g})$. The VOA $V^{\otimes d}$ has central charge $cd$, and it is not a lattice VOA by Proposition \ref{PropTensorLattice}. Since $V^{\otimes d}$ is not a lattice VOA, its partition function is not a theta series by Example \ref{ex: lattice}. For $g$ big enough, theta series are a basis of the space of Siegel modular forms. By \ref{T:indip}, for $g$ big enough the set of theta series - aka partition functions of lattice VOA (Theorem \ref{thm:partition_of_lattice}) - and $u(\oOmega_{V,g})^d$ are linearly independent, hence $u(\oOmega_{V,g})^d$ is not the pull-back of a Siegel modular forms.

\end{proof}

Let us now recall the notion of stable Siegel modular forms, see e.g \cite[Section 2]{PhD_Cod} and references therein. We use notations from 
Section \ref{S:Sieg}. There is a Siegel operator
$$
\Phi_g\colon H^0(\A_g,N_g^{\otimes k})\to  H^0(\A_{g-1},N_{g-1}^{\otimes k})
$$
which is surjective for $k$ even and $k>2g$. Here we take $k$ divisible by 4. Taking inductive limits we construct the ring of stable modular form $\R(\A_\infty,N_\infty^{\otimes 4})$. This also has a natural structure of a co-commutative co-algebra. The sequence of theta series $\Theta_L:=\{\Theta_{L,g}\}_{g\in \N}$, where $L$ is a positive definite even unimodular, lattice, gives examples of stable Siegel modular forms.

Each graded piece $H^0(\A_\infty,N_\infty^k)$, with $k$ divisible by 4, is finite dimensional and has a basis given by theta series $\Theta_L$, where the rank of $L$ is $2k$. The ring $\R(\A_\infty,N_\infty^{\otimes 4})$ is polynomial in the variable $\Theta_L$, where $L$ is irreducible, and we have the relation $\Theta_L \Theta_{\tilde{L}}=\Theta_{L\oplus \tilde{L}}$. Theta series are the unique group-like elements of the co-algebra of stable Siegel modular forms.

\cite[Corollary 1.5]{CS-B14} can be rephrased by saying that the restriction from $\A_g$ to $\M_g$ embeds the algebra of stable Siegel modular into the algebra of stable Teichm\"{u}ller modular forms. 

The arguments given in this section, especially the use we have made of Example \ref{ex: lattice}, says that this embedding is not surjective. More explicitly, partition functions of non-lattice holomorphic VOA give group-like elements which do not lie in its image.

We end this section with some questions motivated by the analogy with stable Siegel modular forms. First, some queries about the ring of stable Teichm\"{u}ller modular forms.
\begin{questions}\label{questions_stable}
\begin{enumerate}[$(i)$]
\item Is the map $\sigma_g$ from Equation \eqref{eq:sigma} surjective?
\item\label{question_item_finite} Given a positive integer $k$, is the dimension of $H^0(\overline{\M}_\infty,\lambda_\infty^{\otimes k})$ finite? 
\item Given a positive integer $k$, does the dimension of $H^0(\oM_g,\lambda_g^{\otimes k})$ stabilizes when $g$ grows? More explicitly, does $\dim H^0(\oM_g,\lambda_g^{\otimes k})$ depend only on $k$ (and not of $g$) for every $g$ big enough?
\item Given a positive integer $k$, does $H^0(\overline{\M}_\infty,\lambda_\infty^{\otimes k})$ have a (Hamel) basis formed by group-like elements?
\end{enumerate}
\end{questions}

The following conjecture is motivated by the Friedan-Shenker program on the geometrization of two-dimensional CFTs \cite{FS}.

\begin{conjecture}[Geometrization of holomorphic CFTs]\label{ConjectureGeoemtrization}
Let $k$ be a positive integer divisible by $4$. Every group-like element of $H^0(\overline{\M}_\infty,\lambda_\infty^{\otimes k})$ is the vacuum sections of a holomorphic vertex operator algebras.
\end{conjecture}

The set of equivalence classes of holomorphic VOAs with a fixed central charge is expected to be finite. This would be a consequence of \cite[Conjecture 3.5]{Hoh03}. It would also follow from Theorem \ref{T:indip} together with an affirmative answer to the item \ref{questions_stable} (\ref{question_item_finite}), and from the reconstruction Conjecture  \ref{Conj:reconstruction}.

\section{Relation with the Schottky problem}\label{S:schottky_problem}

We keep the notation of Section \ref{S:Sieg}. The complex dimension of $\M_g$ is $3g-3$, and the complex dimension of $\A_g$ is $\frac{g(g+1)}{2}$. This implies that for $g\geq 4$, the dimension of $\M_g$ is strictly smaller than the dimension of $\A_g$. The \emph{Schottky problem} asks for a characterization of the image of the Jacobi map inside $\A_g$. Equivalently, it asks which matrices in the Siegel upper half space $\H_g$ are the period matrix of a Riemann surfaces; i.e. a characterization of $J_g$ in $\H_g$.

A standard approach is to look at the equations of $J_g$ in $\H_g$ (or of $j(\M_g)$ in $\A_g$); more concretely, one looks at Siegel modular forms which vanish along $J_g$ (observe that $J_g$ is $Sp(2g,\Z)$-invariant, hence it is sensible to use modular forms).  One often tries to use theta series to write such equations.

The first case is the so called Schottky form $\Theta_{D_{16}^+,g}-\Theta_{E_8^{\oplus 2},g}$. It has weight 8, it is identically zero on $\H_g$ for $g \leq 3$, and it is the equation of $J_4$ in $\H_4$. The two lattices $D_{16}^+$ and $E_8^{\oplus 2}$ are related to string theory, and it has been asked by physicists if $F_g$ were zero on $J_g$ for $g\geq 5$. The problem was addressed in \cite{GruSa}, and the answer is that the Schottky form is not zero on $J_g$ for all $g\geq 5$. 

This approach was generalized in \cite{CS-B14} and \cite{PhD_Cod}: a fixed linear combination of theta series $\Phi_g:=\sum a_i \Theta_{L_i,g_i}$, where the $a_i$'s are complex numbers and the $L_i$'s are positive definite even unimodular latticed of the same rank, is not zero on $J_g$ for all $g$ large enough ($g$ depends on the  $L _i$'s and $a_i$'s, and we do not know how to compute it in an effective way). 

Similar problems were also among the motivations for the works \cite{GV09,GKV10}. In particular, in \cite[Section 3.1]{GV09}, the authors compute some coefficients of the power series from Definition \ref{def:parition_functions} when $V$ is the lattice vertex algebra associated to $D_{16}^+$ or $E_8^{\oplus 2}$, and remark that these coefficients are different when $g\geq 5$. This computation, together with Theorem \ref{thm:partition_of_lattice}, gives an alternative proof that the Schottky form is not zero on $J_g$ for $g\geq 5$. In the same spirit, one could use partition functions to understand if $\Phi_g$ is zero on $J_g$. Indeed, thanks to Theorem \ref{thm:partition_of_lattice}, one can expand the restriction of $\Phi_g$ to $J_g$ in terms of the coordinates from Equation \ref{e:coo_tw}, i.e., write $\Phi_g$ as a linear combination of partition functions from Definition \ref{def:parition_functions}. Then one could compute explicitly the coefficients and check if they vanish.

There are variants of the Schottky problem that consider other subvarieties of $\H_g$. An important one is the hyperellyptic locus, i.e. the subvariety $Hyp_g$ of $J_g$ defined as the closure period matrices of hyperelliptic curves; it has complex dimension $2g-1$. Here something different happens: there are linear combinations of theta series that vanish on $Hyp_g$ for every $g$. The first example is the Schottky form discussed above; more examples are given in \cite{Cod_hyp}, see also \cite{PhD_Cod} and references therein. The analog problem for the locus of curves with higher gonality is addressed in \cite{SB}. We do not know how to interpret this phenomenon in terms of vertex algebras.

\section{The partition function subalgebra}
\label{SecPartitionFunctionSubalgebra}

The main goal of this section is to prove Theorem \ref{T:main2_corpo}. This theorem characterizes the pairs of unitary VOAs having the same partition functions in term of certain subalgebras that we call {\bf partition function subalgebras}, see Definition \ref{def: partitionsubalgebra}. 
Although many of the results of this section have more general validity, we will restrict ourselves to simple unitary VOAs, which are those involved in 
Theorem \ref{T:main2_corpo}. This does not appear to be a too serious restriction since the focus of this paper is on holomorphic VOAs and, in many relevant cases, such as the VOAs associated to unimodular positive definite even lattices, the moonshine VOA and the Schellekens VOAs, these are known to be unitary, see \cite{CKLW18,CGGH23,DL14,Lam23}.  Actually, to the best of our knowledge, there are no known examples of non unitary holomorphic VOAs, cf. Section \ref{sec: c=24}.  For the convenience of the readers we first include some preliminaries on unitary VOAs, see e.g., \cite{CKLW18,CGH,DL14} for more details.
\subsection{Unitary vertex operator algebras} 
\label{Subsec:unitaryVOAs}

Let $V$ be a VOA and let $(\cdot | \cdot ) : V \times V \to \mathbb{C}$ be a scalar product on $V$ (i.e.\ a positive definite Hermitian form on $V$). We say that $(\cdot | \cdot )$ is {\bf normalized}  if $(\Omega | \Omega ) =1$.  Moreover, we say that $(\cdot | \cdot )$ is {\bf invariant} if there exists an anti-linear automorphism $\theta$ of $V$ (the {\bf PCT operator}) such that $(\theta \cdot | \cdot)$ is an invariant bilinear form. The PCT-operator $\theta$ is uniquely determined by the invariant scalar product $(\cdot | \cdot )$ and satisfies $\theta^2=1_V$, see \cite[Proposition 5.1]{CKLW18}.  If $V$ is a unitary with normalized invariant scalar product $(\cdot | \cdot )$ and PCT-operator $\theta$ and if $a \in V_d$ is quasi-primary, then we see from 
Equation (\ref{eqInvBilqp}) that if $a \in V_d$ is quasi-primary then 

\begin{equation}
\label{eqInvScalqp}
(b|a_nc ) = (-1)^d((\theta a)_{-n}b | c)
\end{equation}
for all $b,c \in V$ and all $n \in \mathbb{Z}$.

A {\bf unitary VOA} is a vertex operator algebra $V$ equipped with a normalized invariant scalar product. A unitary vertex operator algebra $V$ is automatically self-dual and it is simple iff it is of CFT type, see \cite[Proposition 5.3]{CKLW18}. Moreover, every unitary VOA is a direct sum of simple unitary VOAs, see \cite[Proposition 3.30]{CGH}. The central charge of a unitary VOA is always a non-negative real number. The only simple unitary VOA with $c=0$ is the trivial vertex operator algebra  $V=\mathbb{C}$.

Examples of unitary VOAs include the lattice VOAs, $V_L$ with $L$ a positive-definite even lattice, the simple affine VOAs $V_k(\mathfrak{g})$ with $\mathfrak{g}$ a simple complex Lie algebra and positive integer level $k$, and the moonshine vertex operator algebra $V^\natural$, see \cite{DLM97}. 
As we have already recalled in Section \ref{S:vert}, all these examples are also strongly rational. There are also many non-rational examples of unitary VOAs such as the rank $r$ Heisenberg VOAs $M_r(1)$ of any non-negative integer rank $r$ and the unitary Virasoro VOAs $L(c,0)$ with central charge $c \geq 1$. Further examples are discussed in Section \ref{SectionExamples}.  
\medskip

If $V$ is a unitary VOA then one can define the unitary subgroup of $\operatorname{Aut}(V)$, denoted by $\operatorname{Aut}_{(\cdot|\cdot)}(V)$, that is, the group of automorphisms $g$ of $V$ such that  $(ga|gb) = (a|b)$ for all $a,b \in V$. Then $\operatorname{Aut}_{(\cdot|\cdot)}(V)$ is a compact subgroup of $\operatorname{Aut}(V)$. If $\operatorname{Aut}(V)$ is finite, as in the case of the moonshine vertex operator algebra $V^\natural$, then  $\operatorname{Aut}_{(\cdot|\cdot)}(V) = \operatorname{Aut}(V)$. On the other hand, in many cases 
$\operatorname{Aut}(V)$ is far from being compact, so $\operatorname{Aut}_{(\cdot|\cdot)}(V)$ is a proper subgroup.
\medskip

We now discuss unitary subalgebras, cf.\ \cite[Section 5.4]{CKLW18} and \cite{CGH19}. In order to simplify the discussion, we will only consider unitary subalgebras of simple unitary VOAs. Let $V$ be a simple unitary VOA. A {\bf unitary vertex subalgebra} (or simply a { unitary subalgebra}) $U$ of $V$ is a vertex subalgebra of $V$ such that $L_1U \subset U$ and $\theta U \subset U$. Even if the subalgebra $U$ is not conformal, that is $\nu_V \notin U$, it is always a simple unitary  VOA. This is because, as a vertex algebra module for $U$, $V$ is the direct sum of $U$ and its orthogonal complement $U^\perp$ with respect to the invariant scalar product $(\cdot |\cdot)$.  Then, we have  the unique decomposition $\nu_V = \nu_U + \nu_{U^\perp}$ with $\nu_U \in U$ and $\nu_{U^\perp} \in U^\perp$ and it turns out that  $\nu_U$ is a  $\theta$-invariant conformal vector for $U$. With this conformal vectors and the restrictions of the invariant scalr product and TCP operators of $V$, $U$ becomes a unitary VOAs. Clearly, $V$ is a conformal extension of $U$ iff $\nu_U = \nu_V$. If this is the case, we say that $V$ is a {\bf unitary conformal VOA extension} of the unitary vertex operator algebra $U$.

If $V$ is a simple unitary VOA then the Virasoro vertex subalgebra $V_{\{\nu \}}$ generated by its conformal vector is unitary and hence isomorphic to the unitary Virasoro VOA $L(c,0)$ where $c$ is in the unitary range that is either $c \geq 1$ or $c= 1 -\frac{6}{(m + 2)(m+3)}$, $m \in \mathbb{Z}_{\geq 0}$. 
Sometimes we will write $L(c,0) \subset V$  instead of $V_{\{\nu \}} \subset V$.
Other examples of unitary subalgebras are given by compact orbifold subalgebras $V^G$ with $G$ a compact subgroup of the unitary automorphism group 
$\operatorname{Aut}_{(\cdot|\cdot)}(V)$. Note that by \cite[Proposition A.3]{CGGH23} we always have 
$V^{\operatorname{Aut}(V)} = V^{\operatorname{Aut}_{(\cdot|\cdot)}(V)}$ so that the smallest orbifold subalgebra $V^{\operatorname{Aut}(V)}$ is unitary. 
\medskip

In this section and in Section\ref{SectionExamples} we will need the notion of unitary VOA-module that we briefly recall here, see e.g. \cite{DL14} and \cite{Gui22}.  Let $V$ be a simple unitary VOA. A {\bf unitary $V$-module} is a $V$-module $M$ equipped with a scalar product $(\cdot |\cdot)_M$ such that $(a^M_nb|c)_M = (b|(S\theta a)^M_{-n}c)_M$ for all $a \in V$, all $b,c \in M$ and all $n \in \mathbb{Z}$ (an {\bf invariant scalar product on M}). A $V$-module is said to be {\bf unitarizable} if it can be equipped with an invariant scalar product making it into a unitary $V$-module.

\subsection{General results on the partition function subalgebra}

As we shall see, the partition function subalgebra $PV$ of a  unitary vertex operator algebra $V$ is contained in the smallest orbifold subalgebra $V^{\operatorname{Aut}(V)}$, see Proposition \ref{PropAutInvariantCasimir}, and hence it is typically strictly smaller than $V$. Moreover, $PV$ is a unitary subalgebra of $V$ containing the conformal vector of $V$, see Lemmas \ref{L:unitary} and \ref{L:conf}. On the other hand,  
as we will also discuss in this section and in Section \ref{SectionExamples} $PV$ is typically strictly larger of the Virasoro subalgebra of $V$, see e.g. 
Remark \ref{remarkLargerVirasoro} and Theorem \ref{th:design}. 

\begin{remark}
\label{remarkqp1}
Recall that a homogeneous vector $a$ in $V$ is called quasi-primary of $L_1a=0$. We denote by $V^{qp}$ the space of quasi-primary vectors of the simple unitary vertex operator algebra $V$. Since $V$ is unitary, then, recalling that the PCT operator $\theta$ commutes with $L_0$, and $L_1$ we see $\theta V_k^{qp} = V_k^{qp}$ for all $k \in \mathbb{Z}_{\geq 0}$. It follows that, for any non-negative integer $k$, we can find a basis  $\{v^{(i)}\}$ for $V_k^{qp}$ with $\theta v^{(i)} = v^{(i)}$ for 
$i=1,\dots , \dim V_k^{qp}$ which is orthonormal with respect to the invariant scalar product $(\cdot |\cdot)$ and hence with respect to the invariant bilinear form $(\cdot,\cdot)$. In particular $(\cdot,\cdot)$ restricts to a non-degenerate bilinear form on $V_k^{qp}$.
\end{remark}

We introduce the following variants of the Casimir bilocal fields.
\begin{definition}
Let $V$ be a simple unitary VOA.

\begin{description}

\item[Quasi-primary Casimir bilocal fields] For every integer $k$, the quasi-parimary $k$-th Casimir bilocal fields is
$$
\gamma_k^{qp}(w,z)=\sum_{i=1}^{\dim V_k^{qp}}Y(v^{(i)},w)Y(v^{(i)},z)\,,
$$
where $\{v^{(i)}\}$ is basis of $V_k^{qp}$ orthonormal with respect to the invariant bilinear form $(\cdot,\cdot)$ . 

The definition does not depend on the choice of the orthonormal basis.  If $V_k^{qp} =\{ 0\}$ i we set the quasi-primary $k$-th Casimir bilocal field to be zero. Note that  $\gamma_0^{qp}(w,z) = 1_V$ and that $\gamma_k^{qp}(w,z) = 0$ if $k<0$ because $V$ is unitary and hence CFT type.

\item[Casimir element]For every pair integers $k$ and $j$ we define the Casimir element
$$C_{k,j}=\sum_{i=1}^{\dim V_k} v^{(i)}_j v^{(i)} $$
where $\{v^{(i)}\}$ is an orthonormal basis of $V_k$ with respect to the invariant bilinear form $(\cdot |\cdot)$. Again, this is independent on the choice of the orthonormal basis. If $V_k$ is trivial, we set $C_{k,j}=0$ for every $j$. Note that  $C_{0,j} = \delta_{0,j}\Omega$ and that $C_{k,j}=0$ whenever 
$k < 0$ or $j >k$. 

\item[Quasi-primary Casimir element]For every pair of integers $k$ and $j$ we define the quasi-primary Casimir element
$$C_{k,j}^{qp}=\sum_{i=1}^{\dim V^{qp}_a} v^{(i)}_b v^{(i)} $$
where $\{v^{(i)}\}$ is a orthonormal basis of $V_k^{qp}$. The definition does not depend on the choice of the orthonormal basis. 
 If $V_k^{qp} = \{0\}$, we set $C_{k,j}^{qp}=0$ for every $j$. Similarly to the previous case we have $C_{0,j}^{qp} = \delta_{0,j}\Omega$ and that $C_{k,j}^{qp}=0$ whenever $k < 0$ or $j >k$. 
\end{description}
\end{definition}
Observe that $C_{k,j}$ and $C_{k,j}^{qp}$ are homogeneous of conformal weight $k-j$. In particular, 
$C_{k,j} = C_{k,j}^{qp} = 0$ if $j>k$. Moreover,   $C_{0,j} = C_{0,j}^{qp} = \delta_{i,j} \Omega$.
Note also that, in general, the quasi-primary Casimir elements $C_{k,j}^{qp}$
 are not quasi-primary vectors. Accordingly, the quasi-primary Casimir element should be intended as ``Casimir element defined using a basis or quasi-primary vectors".

We will call the corresponding vertex operators 
$$Y(C_{k,j},z)= \sum_{n \in \mathbb{Z}} (C_{k,j})_{(n)} z^{-n -1} = \sum_{n \in \mathbb{Z}} (C_{k,j})_{n} z^{-n +b-a} $$
and
$$Y(C_{k,j}^{qp},z)= \sum_{n \in \mathbb{Z}} (C_{k,j}^{qp})_{(n)} z^{-n-1}  = \sum_{n \in \mathbb{Z}} (C_{k,j})_{n} z^{-n +k-j} $$
{\bf Casimir vertex operator} and {\bf quasi-primary Casimir vertex operator} , respectively.

\begin{remark} 
\label{rem: realbasis}
If $V$ is unitary we can consider the real vertex operator subalgebra $V_\mathbb{R}$ of $\theta$-invariant elements of $V$. Then the restriction of the bilinear form $(\cdot,\cdot)$ to $V_\mathbb{R}$ coincides with the restriction of the scalar-product   $(\cdot |\cdot)$ and hence it is a real, positive-definite bilinear form. Consequently, we can choose the orthonormal basis with respect to the invariant bilinear form $(\cdot,\cdot)$  for each $V_k$, $k \in \mathbb{Z}_{\geq 0}$  belonging $V_\mathbb{R}$ so that they are also orthonormal with respect to the invariant scalar product $(\cdot |\cdot)$, cf. Remark \ref{remarkqp1}. 
\end{remark}

\begin{definition} 
\label{def: partitionsubalgebra}
Let $V$ be a simple unitary VOA. The \textbf{partition function subalgebra} $PV$ is the vertex subalgebra of $V$ generated by the Casimir elements 
$C_{k,j}$, $k \in \mathbb{Z}_{\geq 0}$, $k \in \mathbb{Z}$.
\end{definition}

We will need the following notations. We define the endomorphisms $C_{k}(m,n) \in \operatorname{End}(V)$, $k,m,n \in \mathbb{Z}$
 by  $$C_{k}(m,n) :=\sum_{i=1}^{\dim V_k} v^{(i)}_m v^{(i)}_n $$
where $\{v^{(i)}\}$ is a orthonormal basis of $V_k$. If $V_k$ is trivial, we set $C_{k}(m,n)=0$ for all pairs $(m,n) \in \mathbb{Z}^2$. We will call them 
{\bf Casimir endomorphisms}.
Accordingly, 
$$\gamma_k(w,z) = \sum_{m,n \in \mathbb{Z} } C_{k}(m,n) w^{-m-k}z^{-n-k} \,.$$
Similarly, we define the endomorphisms $C_{k}^{qp}(m,n) \in \operatorname{End}(V)$, $k,m,n \in \mathbb{Z}$
 by  $$C_{k}^{qp}(m,n) :=\sum_{i=1}^{\dim V_k^{qp}} v^{(i)}_m v^{(i)}_n $$
where $\{v^{(i)}\}$ is an orthonormal basis of $V_k^{qp}$. If $V_k^{qp}$ is trivial, we set $C_{k}^{qp}(m,n)=0$ for all pairs $(m,n)\in \mathbb{Z}^2$. We will call them 
{\bf quasi-primary Casimir endomorphisms}.
Accordingly,
$$\gamma_k^{qp}(w,z) = \sum_{m,n \in \mathbb{Z} } C_{k}^{qp}(m,n) w^{-m-k}z^{-n-k} \,.$$ 

One of the advantages of introducing quasi-primary Casimir endomorphisms is that they have a well-behaved action with respect to the invariant scalar product in $V$, that is,

\begin{equation}
\label{eqQPendScalarProduct}
(a|C_{k}^{qp}(m,n)b) = (C_{k}^{qp}(-n,-m)a | b)
\end{equation}
for all $a, b \in V$ and all $k,m,n \in \mathbb{Z}$. This can be seen by taking an orthonormal basis for $V_k^{qp}$ of $\theta$ invariant vectors, cf. Remark \ref{remarkqp1}, and by using Equation (\ref{eqInvScalqp} twice).

The following proposition will play an important role. 

\begin{proposition}
\label{PropAutInvariantCasimir} 
Let $V$ be a simple unitary VOA, let $g \in \operatorname{Aut}(V)$ and $k,m,n \in \mathbb{Z}$. Then
\begin{itemize}
\item[$(i)$] $g\gamma_k(w,z)g^{-1} = \gamma_k(w,z)$
\smallskip

\item[$(ii)$]  $g\gamma_k^{qp}(w,z)g^{-1} = \gamma_k^{qp}(w,z)$
\smallskip

\item[$(iii)$]  $gC_k(m,n)g^{-1}=C_k(m,n)$
\smallskip 

\item[$(iv)$]  $gC_k^{qp}(m,n)g^{-1}=C_k^{qp}(m,n)$
\smallskip

\item[$(v)$]  $gC_{k,m} = C_{k,m}$
\smallskip
 
\item[$(vi)$] $gC_{k,m}^{qp} = C_{k,m}^{qp}$
\smallskip

\item[$(vii)$] $PV \subset V^{\operatorname{Aut}(V)}$
  
\end{itemize} 
 
\end{proposition}

\begin{proof} The unique normalized invariant bilinear form $(\cdot,\cdot)$ on $V$ is $g$-invariant. Hence, if $\{ v^{(i)}\}$ is a basis for $V_k$ or $V_k^{qp}$ orthonormal with respect to $(\cdot,\cdot)$ so is $\{ g v^{(i)}\}$ and the conclusion follows from the independence of the choice of the basis of the involved objects. 
\end{proof}

We denote by $PV^{qp}$ the vertex subalgebra of $V$  generated by the quasi-primary Casimir elements 
$C_{k,j}^{qp}$, $k,j \in \mathbb{Z}$, we denote by $\widetilde{PV}$ the linear span of the coefficients of the series of the form
\begin{equation}
\label{eq:spanPVTilde1}
\gamma_{k_1}(w_1,z_1) \gamma_{k_2}(w_2,z_2) \dots \gamma_{k_s}(w_s,z_s) \Omega
\end{equation}
and we denote by  $\widetilde{PV^{qp}}$ the linear span of the coefficients of the series of the form
\begin{equation}
\label{eq:spanPVqpTilde1}
\gamma_{k_1}^{qp}(w_1,z_1) \gamma_{k_2}^{qp}(w_2,z_2) \dots \gamma_{k_s}^{qp}(w_s,z_s)\Omega
\end{equation}
so that $\widetilde{PV}$
is spanned by vectors of the form

\begin{equation}
\label{eq:spanPVTilde2}
C_{k_1}(m_1,n_1) C_{k_2}(m_2,n_2) \dots C_{k_s}(m_s,n_s) \Omega
\end{equation}
and $\widetilde{PV^{qp}}$
is spanned by vectors of the form
\begin{equation}
\label{eq:spanPVqpTilde2}
C_{k_1}^{qp}(m_1,n_1) C_{k_2}^{qp}(m_2,n_2) \dots C_{k_s}^{qp}(m_s,n_s) \Omega \, .
\end{equation}

Note also that $PV$
is spanned by vectors of the form
\begin{equation}
\label{eq:spanPV}
(C_{k_1,m_1})_{n_1}  (C_{k_2,m_2})_{n_2} \dots (C_{k_s,m_s})_{n_s}\Omega
\end{equation}
and $PV^{qp}$
is spanned by vectors of the form
\begin{equation}
\label{eq:spanPVqp}
(C_{k_1,m_1}^{qp})_{n_1}  (C_{k_2,m_2}^{qp})_{n_2} \dots (C_{k_s,m_s}^{qp})_{n_s}\Omega \,.
\end{equation}

\begin{proposition}
\label{prop: partition} 
Let $V$ be a simple unitary VOA. Then
 $PV = PV^{qp} = \widetilde{PV} = \widetilde{PV^{qp}} $.

\end{proposition}

\begin{proof}
We first focus on the quasi-primar Casimir bilocal fields. The first observation is that  unitarity  implies that $V$ is an orthogonal direct sum of irreducible positive-energy representations of the complex Lie algebra generated by $L_1,L_0$ and $L_{-1}$; this Lie algebra is isomorphic to $\mathfrak{sl}_2$. Besides the trivial subrepresentation $V_0 = \mathbb{C}1$ each other irreducible subrepresentation 
has an orthonormal basis of the form $\frac{1}{\| L_{-1}^j v\|}L_{-1}^j v$  $j \in\mathbb{Z}_{\geq 0}$ with $v \in V^{qp}_k$, $\|v\| = 1$, $k >0$. Moreover,  
$$\| L_{-1}^j v\| = \sqrt{\frac{j! (2k - 1 +j)}{(2j-1)!}} \,, $$
for all $v \in V^{qp}_k$, see \cite[Appendix B]{CKLW18}.

Recall the formula
$$
Y(L_{-1}v,z)=\frac{d}{dz}Y(v,z)
$$
we then conclude that 
\begin{equation}\label{E:rel_campi}
\gamma_k(w,z) =\sum_{j=1}^k   \frac{(2j-1)!}{(k-j)! (k + j -1)} \left(\frac{\partial^2}{\partial w \partial z} \right)^{k-j}\gamma_j^{qp}(w,z) \,, 
\end{equation}
for all $k \in \mathbb{Z}_{>0}$. Note also that, $\gamma_0 (w,z) = {\gamma_0}^{qp}(w,z) = 1_V$ and that  $\gamma_1 (w,z) = {\gamma_1}^{qp}(w,z)$ because  $V_1 = V_1^{qp}$. Hence,
 from Eq. (\ref{E:rel_campi}) it follows that, for any $k \in \mathbb{Z}_{ > 0}$, $\gamma_k^{qp} (w,z)$ is a linear combination of the bilocal fields 
 $\left(\frac{\partial^2}{\partial w \partial z}\right)^{k-j}\gamma_j(w,z)$, $j=1, \dots , k$, with coefficients independent of the vertex algebra $V$. As a consequence  $\widetilde{PV} = \widetilde{PV^{qp}}$.
 Now, it follows directly by the Borcherds identity that $\widetilde{PV}$ is invariant  for the action of the endomorphisms $(C_{k,m})_{(n)}$ and hence 
 $PV \subset \widetilde{PV}$. The same argument shows that  $PV^{qp} \subset \widetilde{PV^{qp}}$.  Now, let $a \in V$ be any vector orthogonal 
 to $PV$  with respect to the invariant scalar product $(\cdot |\cdot)$ and let $b$ be any vector in $PV$. Since $Y(C_{k,m},z)$ preserves $PV$, we have 
 $$
 \sum_{m \in \mathbb{Z}}(a|Y(C_{k,m},z)b) (w-z)^{-m-k}=0
 $$

 Then
 
 \begin{eqnarray*}
 0 &=&\sum_{m \in \mathbb{Z}}(a|Y(C_{k,m},z)b) (w-z)^{-m-k}  \\
 &=& (a | \sum_{i=1}^{\dim V_k}(a|Y(Y(v^{(i)},w-z)v^{(i)},z)b)
 \end{eqnarray*} 
 for all $k \in \mathbb{Z}_{\geq 0}$. Hence, by the associativity property of VOAs \cite[Prop. 3.3.2]{Axiomatic} 
 we have  
 \begin{equation*}
 0 =    \sum_{i=1}^{\dim V_k}(a|Y(v^{(i)},w)Y(v^{(i)},z)b) = (a|\gamma_k(w,z)b)  
\end{equation*}
for all $k \in \mathbb{Z}_{\geq 0}$. As a consequence, $PV$ is invariant under the action of all endomorphisms $C_{k}(m,n)$ and therefore
$\widetilde{PV} \subset PV$ because, by definition, $\widetilde{PV}$ is the smallest subspace of $V$ containing $\Omega$ and invariant under the action of all endomorphisms $C_{k}(m,n)$. A similar argument shows that  $\widetilde{PV^{qp}} \subset PV^{qp}$ and completes the proof.  
\end{proof}

\begin{lemma}\label{L:unitary} Let $V$ be a simple unitary VOA. 
Then the partition function algebra $PV$ is a unitary subalgebra of $V$. 
In particular, it is preserved by the duality operator $S$.
\end{lemma}
\begin{proof}
By the $\mathfrak{sl}_2$-covariance commutation relations for quasi-primary vectors we have 
\begin{eqnarray*}
[L_{1},C_{k}^{qp}(m,n)] &=& \sum_{i=1}^{\dim V_k} [L_1, v^{(i)}_m v^{(i)}_n] \\
&=& - \sum_{i=1}^{\dim V_k} \left( (m-k+1)v^{(i)}_{m+1} v^{(i)}_n + (n -k+1)v^{(i)}_m v^{(i)}_{n+1}  \right) \\
&=& - (m-k+1) C_{k}^{qp}(m+1,n) -(n-k+1) C_{k}^{qp}(m,n+1)
\end{eqnarray*}
for all  $k \in \mathbb{Z}_{\geq 0}$ and all $m,n \in \mathbb{Z}$. It follows that $\widetilde{PV^{qp}} = PV$ is invariant for the action of $L_1$. 
By Remark \ref{rem: realbasis} we can choose $\theta$-invariant orthonormal basis in the definition of the Casimir elements and then it is 
clear that $\theta C_{a,b} = C_{a,b}$ for all $a \in \mathbb{Z}_{\geq 0}$ and all $b \in \mathbb{Z}$. Since $PV$ is generated by the Casimir elements 
we also have $\theta PV = PV$ and the unitarity follows from \cite[Proposition 5.23]{CKLW18}.
The duality operator $S$ is expressed in terms of $L_0$ and $L_1$; as the unitary subalgebras are both $L_0$ and $L_1$ invariant, they are stable by the duality operator.
\end{proof}

We now give constraints on the elements orthogonal to the partition function subalgebra with respect to the invariant scalar product $(\cdot |\cdot)$.

\begin{remark}
\label{remarkorthogonal}
If $V$ is a unitary VOA and $U \subset V$ is a unitary subagebra then, recalling that $\theta U = U$,  for any $b \in V$ we have $(a|b) = 0$ for all 
$a \in U$ iff  $(\theta a|b) = 0$ for all $a \in U$ iff  $(a,b) = 0$ for all $a \in U$. Consequently, $b$ is orthogonal to $U$ with respect to the invariant scalar product $(\cdot |\cdot)$ iff it is orthogonal to $U$ with respect to the invariant bilinear form $(\cdot,\cdot)$. Note that the same argument works if $U$ is only a $\theta$-invariant subspace of $V$. Consequently, in these cases, we can simply say that $b$ is orthogonal to $U$ without specifying whether orthogonality is meant with respect to $(\cdot |\cdot)$ or $(\cdot ,\cdot)$, cf. Remark \ref{remarkqp1} and Remark \ref{rem: realbasis}.
\end{remark}

\begin{proposition}\label{P:ort}
Let $a \in V$ be a  vector orthogonal to $PV$ with respect to the invariant scalar product $(\cdot | \cdot)$, then
$$ \Tr_{V_k} a_0=0$$
for every non-negative integer $k$.
\end{proposition}
\begin{proof}
Since the orthogonal complement of $PV$ is invariant for $L_0$ we can assume, without loss of generality, that $a \in V_d$ for some 
integer $d>0$.  As $a$ is orthogonal to $PV$,
all coefficients of $Y(a,t)\Omega = e^{tL_{-1}}a$ are orthogonal to $PV$. Using Equation (\ref{eqQPendScalarProduct})
 we have 
\begin{eqnarray*}
(\Omega | \gamma_k^{qp}(w,z) Y(a,t) \Omega)  &=& \sum_{m,n \in \mathbb{Z} } (\Omega | C_{k}^{qp}(m,n) Y(a,t)\Omega) w^{-m-k}z^{-n -k} \\ 
&=& \sum_{m,n \in \mathbb{Z} } ( C_{k}^{qp}(-m,-n)\Omega | Y(a,t)\Omega) w^{-m-k}z^{-n -k}  \\
&=& 0 
\end{eqnarray*}
for every integer $k$.  It follows that 
\begin{eqnarray*}
\sum_{i=1}^{\dim V_k^{qp}}(Y(v^{(i)},w) \Omega,Y(v^{(i)},z)  Y(a,t) \Omega) &=& 
(-w^2)^{-k}  (\Omega , \gamma_k^{qp}(w^{-1},z) Y(a,t) \Omega)\\
&=& (-w^2)^{-k}  (\Omega | \gamma_k^{qp}(w^{-1},z) Y(a,t) \Omega) \\
&=& 0
\end{eqnarray*}
 Hence, by the commutativity property of VOAs \cite[Pro. 3.5.1]{Axiomatic}, 
$$ \sum_{i=1}^{\dim V_k^{qp}}(Y(v^{(i)},w) \Omega, Y(a,t) Y(v^{(i)},z) \Omega) = 0 \,.$$

Taking derivatives of the latter equality with respect to $w$ and $z$ we find

$$\sum_{i=1}^{\dim V_k^{qp}}(Y(L_{-1}^k v^{(i)},w) \Omega, Y(a,t) Y(L_{-1}^kv^{(i)},z) \Omega) = 0 $$
for all $a, k \in \mathbb{Z}_{\geq 0}$.  
It follows that
$$\sum_{i=1}^{\dim V_k}(Y(v^{(i)},w) \Omega, Y(a,t) Y(v^{(i)},z) \Omega) = 0,  $$
cf. Eq. (\ref{E:rel_campi}). Here, the symbol $v^{(i)}$ denotes an element of a basis for $V_k$ orthonormal with respect to the invariant bilinear form $(\cdot, \cdot)$. 
Now,  taking $w=z=0$ and recalling that we are assuming that $a\in V_d$ we find

$$\Tr_{V_k}a_0=   t^d  \Tr_{V_k}Y(a,t) = t^d \sum_{i=1}^{\dim V_k}(v^{(i)}, Y(a,t) v^{(i)}) = 0 \,.$$
\end{proof}

\begin{corollary}
\label{cor: projection partition}
Let $V$ be a simple unitary VOA and let $e_{PV}: V \to PV$ be the orthogonal projection onto $PV$ with respect to the invariant scalar product 
$(\cdot|\cdot)$. If $a \in V$ is such that  
$\Tr_{V} a_0q^{L_0} \neq 0$ then $e_{PV} a \neq 0$. In particular, if $a$ is a primary vector in $V_d$ and $\Tr_{V} a_0 q^{ L_0} \neq 0$ then 
$e_{PV} a$ is a nonzero primary vector in $PV_d$. 

\end{corollary}

We can now show that the partition function subalgebra contains the conformal vector

\begin{lemma}\label{L:conf}
The partition function subalgebra $PV$ contains the conformal vector $\nu$ of $V$.
\end{lemma}
\begin{proof}
Following \cite[Example 5.27]{CKLW18}, consider the coset subalgebra of $PV$ defined as
$$PV^c=\{a \textrm{ in } V \textrm{ such that } [Y(a,w),Y(b,z)]=0 \textrm{ for all } b \textrm{ in } PV\}\,.$$
As $PV$ is unitary, the same is true for $PV^c$. The orthogonal projections $\nu^{PV}$ and $\nu^{PV^c}$ of $\nu$ on $PV$ and $PV^c$ are conformal vectors for $PV$ and $PV^c$ and $\nu=\nu^{PV}+\nu^{PV^c}$. Moreover, $\nu_0^{PV}$ and 
$\nu_0^{PV^c}$ are simultaneously diagonalizable with non-negative eigenvalues (see \cite[Propositions 5.29 and 5.31]{CKLW18}). As $\nu^{PV^c}$ is orthogonal to $PV$ and it is diagonalizable with non-negative eigenvalues, we can apply Proposition \ref{P:ort} to conclude that its zero mode is the zero operator, hence $\nu^{PV^c}=0$ and $\nu=\nu^{PV}$. (As byproduct, this also shows that $PV^c=V_0$, see \cite[Remark 5.30]{CKLW18}.)
\end{proof}

\begin{remark}
In the case of the moonshine vertex operator algebra $V^{\natural}$, there is a more direct way to show that $PV^{\natural}$ contains the conformal vector. This argument also shows that for any vertex algebra with $\dim V_1=0$, the weight two part of $PV$ is non-trivial. Let us explain it. Assume that $V_1=0$. 
It follows by skew-symmetry property of vertex algebras (see e.g. \cite[Sec. 4.2]{Kac01}), that $V_2$ is a commutative non-associative algebra with the product $a_0b$, $a,b \in V_2$, (the so-called Griess algebra). 
In particular, we have $a_0 \nu =\nu_0a=2a$ for all $a \in V_2$. 
It follows that $(\nu | C_{2,0})= 2\dim V_2$; in particular, $C_{2,0}\neq 0$ and $PV_2$ is non-trivial. In the moonshine vertex algebra, the unique 
$\Aut(V^\natural)$-invariant vector of weight two is the conformal vector, so $C_{2,0}$ is a non-zero multiple of the conformal vector.
\end{remark} 

\begin{remark}
\label{remarkLargerVirasoro} 
By Lemma \ref{L:conf} $PV$ contains the Virasoro subalgebra $V_{\{\nu\}}$ of $V$. On the other hand, if there is a primary vector $a \in V$ with 
weight $d>0$ and such that  $\Tr_{V} a_0 q^{ L_0} \neq 0$ then, by Corollary \ref{cor: projection partition} $PV$ is strictly larger than $V_{\{\nu\}}$. \end{remark}

We will need some general results on the unitary representation theory of unitary VOAs. 

\medskip

\begin{lemma}
\label{lem: irrsubmodule}
Let $V$ be a simple unitary VOA. Then every unitary $V$-module $M$ contains an irreducible submodule. 
\end{lemma}
\begin{proof}
If $N$ is a submodule fo $M$ then its orthogonal complement $N^\perp$ in $M$ with respect to the  scalar product $(\cdot | \cdot)_M$ is also a submodule of $M$. 
$e^{2\pi i  L^M_0}$  is a unitary operator on $M$ commuting with every vertex operator $Y^M(a,z)$, $a \in V$. Hence, every  eigenspace of   
$e^{2\pi i  L^M_0}$
is a $V$-submodule and $M$ is the orthogonal direct sum of these submodules. Accordingly, we can assume without loss of generality that 
$e^{2\pi i  L_0} = e^{2\pi i  h}$ with $h \geq 0$ the smallest eigenvalue of $L_0$ on $M$ (the {\it lowest energy}) and hence that 
$$M =\bigoplus_{n \in \mathbb{Z}_{\geq 0}}M_{h+n} \, .$$
Since $M_h$ is finite dimensional, there is a minimal submodule $N$ such that $N_h \neq \{ 0\}$. Then $N$ is irreducible because if $L$ is a non-zero proper submodule of $N$ then,  either $L_h \neq \{ 0\}$ or $(L^\perp \cap N)_h\neq \{ 0\}$  in contradiction with the minimality of $N$. 
\end{proof}

\begin{proposition}
\label{prop: unitarymodule}
 Let $V$ be a simple unitary VOA and let $M$ be a unitary $V$-module. Then $M$ is an orthogonal direct sum of irreducible unitary submodules. Moreover, 
 if $e^{2\pi i L^M_0}$ acts as a scalar on $M$ then the direct sum is at most countable. 
\end{proposition}
\begin{proof}
As in the proof of Lemma \ref{lem: irrsubmodule} we can assume that $e^{2\pi i  L^M_0} = e^{2\pi i  h_M}$ with $h_M \geq 0$ the smallest eigenvalue of $L^M_0$ on $M$. 
We set $h_1 = h_M$ and let $M^1$ a minimal submodule of $M$ with $M^1_{h_1} \neq \{ 0\}$. Now, if $(M^1)^\perp =  \{0\}$ 
then $M=M^1$. On the other hand,  if $(M^1)^\perp \neq  \{0\}$ we denote by  $h_2 \geq h_1$ the smallest eigenvalue of $L^M_0$ on 
$(M^1)^\perp$ and let $M^2$ be a minimal submodule of $(M^1)^\perp$ with $M^2_{h_2} \neq \{0\}$.  
Again, if $(M^1 \oplus M^2)^\perp =  \{0\}$ we find $M= M^1 \oplus M^2$.  On the other hand,  if $(M^1 \oplus M^2)^\perp \neq  \{0\}$ we denote by  $h_3 \geq h_2$ the smallest eigenvalue of $L_0$ on $(M^1 \oplus M^2)^\perp$ and let $M^3$ be a minimal submodule of $(M^1)^\perp$ with $M^3_{h_3} \neq \{0\}$. Proceeding in this way either there is a positive integer $k$ and irreducible $V$-submodules $M^1,\dots M^k$ such that 
$M = M^1\ \oplus \dots \oplus M^k$ or we find a sequence $M^j$, $j  \in \mathbb{Z}_{> 0}$ of irreducible $V$-submodules of $M$ and a corresponding increasing unbounded sequence of non negative real numbers $h_1 \leq h_2 \leq \dots $.  In the first case we have nothing to prove. In the second case
$$\left(M^1\oplus M^2 \oplus \dots \oplus M^j \right)^\perp_r = \{0\} $$
for $r < h_j$ and hence 
$$\left( \bigoplus_{j \in \mathbb{Z}_{\geq 0}} M^j  \right)^\perp = \{ 0\} $$
so that 
$$\bigoplus_{k \in \mathbb{Z}_{\geq 0}} M^k  =  M \,.$$
\end{proof}

It is well known that the characters of a strongly rational VOA associated to inequivalent modules are linear independent. Here we give a similar result for any not necessarily rational simple unitary VOA.

\begin{proposition}
\label{prop: characters independence}
Let $V$ be a simple unitary VOA and $M^1, \dots,M^k$ be pairwise inequivalent irreducible unitary $V$-modules. Then, the characters 
$a \mapsto \Tr_{M^i}(a_0^{M^i}q^{ L^{M^i}_0})$, $i=1,\dots,k$ are linearly independent.   
\end{proposition}

\begin{proof}
Let $\alpha_1,\dots ,\alpha_k$ be complex numbers such that 
 $$\alpha_1 \Tr_{M^1}(a_0^{M^1}q^{L^{M^1}_0}) + \alpha_2 \Tr_{M^2}(a_0^{M^2}q^{L^{M^2}_0}) + \dots + 
 \alpha_k \Tr_{M^k}(a_0^{M^k}q^{ L^{M^k}_0}) = 0 $$
  for all $a \in V$. Let $h_i$ be the minimal energy of $M^i$. 
  We can assume that $h_1 \leq h_2 \leq \dots \leq h_k$.  
  Let $M$ be the orthogonal direct sum of the $M^i$, $i=1,\dots,k$. Then $M$ is naturally a unitary $V$-module and we have a unital representation 
 $\pi$ of the Zhu algebra $A(V)$ on the top space
 $$M_{(0)}:= \bigoplus_{i=1}^k M^i_{h_i} \, ,$$ 
 a finite-dimensional complex Hilbert space, see \cite[Sec. 2]{Zhu}. Moreover, $\pi$ is the direct sum of pairwise inequivalent irreducible subrepresentation $\pi_i$ on  $M^i_{h_i}$,  see \cite[Thm. 2.2.2]{Zhu}. Accordingly $\pi(A(V))$ is a finite dimensional algebra over $\mathbb{C}$ 
 with zero Jacobson radical. Hence, $\pi(A(V))$ is semisimple
so that  $$\pi(A(V)) = \bigoplus_{i=1}^k \operatorname{End}(M^i_{h_i} ).$$ 
  Thus we can find elements $b^1,b^2,\dots ,b^k$ of $V$ such that the corresponding equivalence classes
 $[b^1],[b^2],\dots , [b^k] \in A(V) = V/O(V)$  satisfy $\pi_j([b^i]) = \delta_{i,j} 1_{M_{h_j}^j}$, that is $(b^i)^{M_{h_j}^j}_0 = \delta_{i,j} 1_{M_{h_j}^j}$,.
 It follows that 
   $$\Tr_{M^j}((b^i)_0^{M^j}q^{L^{M^j}_0}) = \delta_{i,j} \, \mathrm{dim}(M^i_{h_i}) q^{h_i} + o(q^{h_j}) \; \mathrm{as} \; 
   q  \to 0  \,.$$  
   Accordingly, for any $i \in \{1,\dots,k\}$ we have 
   \begin{eqnarray*}
   0 &=&    \alpha_1 \Tr_{M^1}((b^i)^{M^1}_0 q^{L^{M^1}_0}) + \alpha_2 \Tr_{M^2}((b^i)_0^{M^2}q^{L^{M^2}_0})+ \dots + \alpha_k \Tr_{M^k}  ((b^i)_0^{M^k}q^{L^{M^k}_0})   \\
   &=& \alpha_i \mathrm{dim}(M^i_{h_i}) q^{h_i} + \sum_{j=1}^k o( q^{h_j})  \; \mathrm{as} \; q \to 0\,.
    \end{eqnarray*}
 and recalling that $h_1 \leq h_2 \leq \dots \leq h_k$ we see that  $0= \alpha_1 \mathrm{dim}(M^1_{h_1}) 
 q^{h_1} + o(q^{h_1})$ as $ q \to 0$ and thus 
 $\alpha_1=0$. Similarly, $\alpha_2 = 0$ and so on.
   \end{proof}

\begin{proposition}
\label{propr: charactermodules}
Let $V$ be a simple unitary VOA  and let $M, N$ two unitary $V$-modules. If 
$$\Tr_{M}(a_0^{M}q^{ L^M_0}) = \Tr_{N}(a_0^{N}q^{ L^N_0})$$
for all $a \in V$ then $M$ and $N$ are unitarily equivalent $V$-modules. 
\end{proposition}

\begin{proof}
Assume that $\Tr_{M}(a_0^{M}q^{L^M_0}) = \Tr_{N}(a_0^{N}q^{L^N_0})$ for all $a\in V$. For a given $h \geq 0$ let $M^h = \{b \in M: e^{2\pi i  L^M_0}b  = e^{2\pi i  h}b \}$ and $N^h = \{b \in N: e^{2\pi i  L^N_0}b  = e^{2\pi i  h}b \}$. 
Then $M^h$ and $N^h$ are $V$-submodules of $M$ and $N$ respectively. Moreover, 
$$\Tr_{M^h}(a_0^{M^h}q^{L^M_0}) = \Tr_{N^h}(a_0^{N^h}q^{L^N_0})$$
for all $a\in V$ and all $h \geq 0$ and $M$ and $N$ are unitarily equivalent if and only if 
$M^h$ is unitarily equivalent to $N^h$ for any $h \geq 0$.
Accordingly, we can assume that $M= M^h$ and $N=N^h$ for some $h \geq 0$. 
Then, by Proposition \ref{prop: unitarymodule} $M$ there are pairwise inequivalent irreducible $V$-modules $M^1, M^2, \dots$ with lowest energies $h_1 \leq h_2 \leq \dots$ such that $M$ is unitary equivalent to $\bigoplus_i m_i M^i$ and $N$ is unitarily equivalent to $\bigoplus_i n_i M^i$ where $m_i$ and $n_i$ denote the multiplicities of $M^i$ as a $V$-submodule of $M$ and $N$ respectively. If the number of irreducible submodules $M^i$ with $m_i n_i \neq 0$ is finite, then the conclusion follows directly from Proposition \ref{prop: characters independence} because for every $a \in V$, 
$$0 = \Tr_{M}(a_0^{M}q^{L^M_0}) - \Tr_{N}(a_0^{N}q^{L^N_0}) = \sum_i (m_i-n_i) \Tr_{M^i}(a_0^{M^i}q^{L^{M^i}_0}) \, .$$
Hence, we can assume that $M^i$, $i \in \mathbb{Z}_{>0}$ is an infinite sequence. In this case $h_i \to + \infty$ as $i \to +\infty$. 
Now, fix a positive integer $n> h_1$ and let $k$ be a positive integer such that $h_k \leq n$ and $h_{k+1}>n$. 
Arguing as in the proof of Proposition \ref{prop: characters independence} we can find, for every positive integer $i \leq k$ a vector $b^i \in V$ such that 
$$\Tr_{M^j}((b_0^i)^{M^j}q^{L^{M^j}_0} ) = \delta_{i,j} \mathrm{dim}(M_{h_j}^j) q^{h_i} + o(q^{h_j}) \; \text{as} \; q \to 0$$
for $j = 1,\dots ,k$. It follows that for $i \leq k$,   
$$0 = (m_i - n_i)\mathrm{dim}(M_{h_i}^i) q^{h_i} + o(q^{h_1})\; \text{as}\; q\to 0 \, ,$$ 
$i = 1, \dots ,k$. Thus, arguing again as in the proof of Proposition \ref{prop: characters independence}, we find that 
for $h_i \leq n$, we have $n_i = m_i$, and since $n$ was arbitrary, we have $n_i = m_i$ for every positive integer $i$ so that 
$M$ and $N$ are unitarily equivalent.
 \end{proof}

 \subsection{Relation between the partition function subalgebra and the partition function}
 We now come back to the discussion of the partition functions and the partition function algebras of unitary VOAs. Let $V$ and $U$ be simple unitary VOAs. Our aim is to give some characterizations of the equality of the partition functions ${\cZ}_V$ and ${\cZ}_U$. Recall that the partition function is as in Definition \ref{def:parition_functions}, and by $\cZ_V$ we denote the collection of power series $\{\cZ_{V,g,0}\}_{g \in\mathbb{Z}_{\geq 0}}$.

 We denote by  $\gamma_k(w,z)$, $\gamma_k^{qp}(w,z)$, $C_{k}(m,n)$ and  $C^{qp}_{k}(m,n)$ the Casimir bilocal fields, the quasi-primary Casimir bilocal fields, the Casimir endormorphisms, and  the quasi-primary Casimir endormorphisms of $V$ respectively.  Moreover, 
 we denote by  $\eta_k(w,z)$, $\eta_k^{qp}(w,z)$, $D_{k}(m,n)$ and  $D^{qp}_{k}(m,n)$ the Casimir bilocal fields, the quasi-primary Casimir bilocal fields, the Casimir endormorphisms and  the quasi-primary Casimir endormorphisms of $U$ respectively.

 \begin{proposition} 
 \label{prop: equality of partition functions} 
 Let $V$ and $U$ be two simple unitary VOAs. Then the following are equivalent: 
 \begin{itemize}
 \item[$(i)$] $\cZ_V = \cZ_U$
 \smallskip
 
 \item[$(ii)$] $(\Omega^V|\gamma_{k_1}(w_1,z_1) \dots \gamma_{k_s}(w_s,z_s)\Omega^V) = 
 (\Omega^U|\eta_{k_1}(w_1,z_1) \dots \eta_{k_s}(w_s,z_s)\Omega^U)$ 
 \smallskip
 
 \noindent for all $s \in \mathbb{Z}_{\geq 0}$ and all  $(k_1, \dots , k_s) \in \mathbb{Z}_{\geq 0}^s$ 
 \smallskip
 
  \item[$(iii)$] $ \Tr_V(\gamma_{k_1}(w_1,z_1) \dots \gamma_{k_s}(w_s,z_s)q^{L_0}) = \Tr_U(\eta_{k_1}(w_1,z_1) \dots \eta_{k_s}(w_s,z_s)q^{L_0})$ 
 \smallskip
 
 \noindent for all $s \in \mathbb{Z}_{\geq 0}$ and all  $(k_1, \dots , k_s) \in \mathbb{Z}_{\geq 0}^s$ 
 \smallskip 
 
 \item[$(iv)$] $(\Omega^V | C_{k_1}(m_1,n_1) \dots C_{k_s}(m_s,n_s)\Omega^V) = 
 (\Omega ^U| D_{k_1}(m_1,n_1) \dots D_{k_s}(m_s,n_s) \Omega^U)  $ 
 \smallskip
 
 \noindent for all $s \in \mathbb{Z}_{\geq 0}$, all  $(k_1, \dots , k_s) \in \mathbb{Z}_{\geq 0}^s$ and all $(m_1,n_1, \dots m_s, n_s ) \in
 \mathbb{Z}^{2s}$  
 \smallskip

  \item[$(v)$] $\Tr_V(C_{k_1}(m_1,n_1) \dots C_{k_s}(m_s,n_s)q^{L_0}) = \Tr_U(D_{k_1}(m_1,n_1) \dots D_{k_s}(m_s,n_s)q^{L_0})$ 
 \smallskip
 
 \noindent for all $s \in \mathbb{Z}_{\geq 0}$, all  $(k_1, \dots , k_s) \in \mathbb{Z}_{\geq 0}^s$ and all $(m_1,n_1, \dots m_s, n_s ) \in
 \mathbb{Z}^{2s}$   
  \smallskip 
 
 \item[$(vi)$] $(C_{k_1}(m_1,n_1) \dots C_{k_r}(m_r,n_r)\Omega^V | C_{k_{r+1}}(m_{r+1},n_{r+1}) \dots C_{k_s}(m_s,n_s)\Omega^V) = \\
(D_{k_1}(m_1,n_1) \dots D_{k_r}(m_r,n_r)\Omega^U | D_{k_{r+1}}(m_{r+1},n_{r+1}) \dots D_{k_s}(m_s,n_s)\Omega^U)$ 
 \smallskip
 
 \noindent for all $r \leq s \in \mathbb{Z}_{\geq 0}$, all  $(k_1, \dots , k_s) \in \mathbb{Z}_{\geq 0}^s$ and all $(m_1,n_1, \dots m_s, n_s ) \in
 \mathbb{Z}^{2s}$   
 
  \end{itemize}
  
 \end{proposition} 
 
 \begin{proof} The equivalence$(i) \Leftrightarrow (ii)$ follows in a straightforward  way from the definition of the partition function. 
 The implication $(iii) \Rightarrow (ii)$ follows from the fact that taking $q=0$ in $(iii)$ we get $(ii)$. To prove that  
 $(ii) \Rightarrow (iii)$  we first observe that by Eq. (\ref{E:rel_campi}) and the subsequent discussion we see that $(ii)$ implies that 
 \begin{eqnarray*}
 (\Omega^V|\gamma_k^{qp}(z,w) \gamma_{k_1}(w_1,z_1) \dots \gamma_{k_s}(w_s,z_s)\Omega^V) = \\
 (\Omega^U|\eta_k^{qp}(z,w)\eta_{k_1}(w_1,z_1) \dots \eta_{k_s}(w_s,z_s)\Omega^U)
 \end{eqnarray*}
 for all $k \in \mathbb{Z}_{\geq 0}$, all $s \in \mathbb{Z}_{\geq 0}$ and all  $(k_1, \dots , k_s) \in \mathbb{Z}_{\geq 0}^s$.  Then, arguing as in the proof of Proposition \ref{P:ort}, we get $(iii)$. 
 Now, $(ii) \Leftrightarrow (iv)$ and $(iii) \Leftrightarrow (v)$ follow directly by the definition of the 
 Casimir endomorphisms. Since $( iv)$ follows from $(vi)$ by taking $r=0$ it only remains to show that $(vi)$ follows from the other five conditions. Going to prove $(iv) \Rightarrow (vi)$ by induction on the non-negative integer $r$. As already noted, for $r=0$ $(vi)$ coincides with $(iv)$. Assume that $(vi)$ holds for a given $r$ 
 and all $s\geq r$. Then, for all $s\geq r +1$ we have that
 $$
 (C_{k_2}(m_2,n_2) \dots C_{k_{r+1}}(m_{r+1},n_{r+1})\Omega^V | C_{k_1}(m_1,n_1) C_{k_{r+2}}(m_{r+2},n_{r+2}) \dots C_{k_s}(m_s,n_s)\Omega^V)
 $$
 is equal to
$$(D_{k_2}(m_2,n_2) \dots D_{k_{r+1}}(m_{r+1},n_{r+1})\Omega^U |D_{k_1}(m_1,n_1)  D_{k_{r+2}}(m_{r+2}, n_{r+1}) \dots D_{k_s}(m_s, n_s)\Omega^U)
$$  
all  $(k_1, \dots , k_s) \in \mathbb{Z}_{\geq 0}^s$ and all $(m_1,n_1, \dots m_s, n_s ) \in
 \mathbb{Z}^{2s}$.  
 Using again  Eq.(\ref{E:rel_campi}) we have
 $$(C_{k_2}(m_2,n_2) \dots C_{k_{r+1}}(m_{r+1},n_{r+1}) \Omega^V | C_{k_1}^{qp}(m_1,n_1) C_{k_{r+2}}(m_{r+2}, n_{r+2}) \dots C_{k_s}(m_s, n_s)\Omega^V)$$
 is equal to 
$$(D_{k_2}(m_2, n_2) \dots D_{k_{r+1}}(m_{r+1}, n_{r+1}) \Omega^U  |D_{k_1}^{qp}(m_1, n_1)  D_{k_{r+2}}(m_{r+2}, n_{r+1}) \dots D_{k_s}(m_s, n_s) \Omega^U)$$
all  $(k_1, \dots , k_s) \in \mathbb{Z}_{\geq 0}^s$ and all $(m_1,n_1, \dots m_s, n_s ) \in 
 \mathbb{Z}^{2s}$. Now, it follows from Equation (\ref{eqQPendScalarProduct}) that for any $k \in \mathbb{Z}_{\geq 0}$ and any $m,n \in \mathbb{Z}$ we have 
 $$(a|C_{k}^{qp}(m,n)b) = (C_{k}^{qp}(-n,-m)a|b)$$
 for all $a, b \in V$ and 
  $$(a|D_{k}^{qp}(m,n)b) = (D_{k}^{qp}(-n,-m)a|b)$$    
   for all $a, b \in U$. Hence,    
    \begin{eqnarray*}
 (C^{qp}_{k_1}(m_1,n_1) \dots C_{k_{r+1}}(m_{r+1},n_{r+1})\Omega^V |C_{k_{r+2}}(m_{r+2},n_{r+2}) \dots C_{n_s}(m_s, n_s) \Omega^V) = \\
(D^{qp}_{k_1}(m_1, n_1) \dots D_{k_{r+1}}(m_{r+1}, n_{r+1}) \Omega^U |D_{n_{r+2}}(m_{r+2}, n_{r+1}) \dots D_{k_s}(m_s,n_s) \Omega^U)
\end{eqnarray*} 
all  $(k_1, \dots , k_s) \in \mathbb{Z}_{\geq 0}^s$ and all $(m_1,n_1, \dots m_s, n_s ) \in 
 \mathbb{Z}^{2s}$ so that, again by  Eq.(\ref{E:rel_campi}),   we can conclude that    
    \begin{eqnarray*}
 (C_{k_1}(m_1,n_1) \dots C_{k_{r+1}}(m_{r+1},n_{r+1})\Omega^V |C_{k_{r+2}}(m_{r+2},n_{r+2}) \dots C_{k_s}(m_s,n_s)\Omega^V) = \\
(D_{k_1}(m_1,n_1) \dots D_{k_{r+1}}(m_{r+1},n_{r+1})\Omega^U |D_{k_{r+2}}(m_{r+2},n_{r+1}) \dots D_{k_s}(m_s,n_s)\Omega^V)
\end{eqnarray*} 
all  $(k_1, \dots , k_s) \in \mathbb{Z}_{\geq 0}^s$ and all $(m_1,n_1, \dots m_s, n_s ) \in
 \mathbb{Z}^{2s}$ that proves $(vi)$.
 \end{proof}

We are now in a position to prove the central result of this section.

\begin{theorem}\label{T:main2_corpo}
Let $V$ and $U$ be two simple unitary VOAs. Then the equality of partition functions $\cZ_V = \cZ_U$ holds 
if and only if there exists 
a unitary operator 
$$\Phi: V \to U$$ 
which restricts to a unitary VOA isomorphism
$$\varphi \colon PV \to PU $$
and satisfying
$$\Phi Y(a,z) = Y(\varphi a,z) \Phi$$ 
for all $a \in PV$. If this is the case, then $\Phi$ also intertwines the action of the Casimir bilocal fields of $V$ and $U$, namely,
$$\Phi \gamma_k(w,z)= \eta_k(w,z)\Phi $$ 
for every non-negative integer $k$.  In particular, if $V$ and $U$ have the same partition function, then they have the same central charge. 
\end{theorem}
 
\begin{proof} 
Assume first that there is such a unitary operator $\Phi$ and, moreover,  $\Phi \gamma_k(w,z) = \eta_k(w,z)\Phi$ for every $k \in \mathbb{Z}_{\geq 0}$. 
Note that
$\Phi \Omega^V = \varphi \Omega^V = \Omega^U$ because $\varphi: PV \to PU$ is a VOA isomorphism. Hence, $\Phi^{-1}\Omega^U = \Omega^V$. 
Then,  for every non negative integer $s$ and every $(k_1, \dots k_s) \in \mathbb{Z}_{\geq 0}^s$, by the unitarity of $\Phi$ we have 
\begin{eqnarray*}
 (\Omega^V|\gamma_{k_1}(w_1,z_1) \dots \gamma_{k_s}(w_s , z_s)\Omega^V) &=&  (\Phi \Omega^V|\Phi \gamma_{k_1}(w_1,z_1) \dots \gamma_{k_s}(w_s, z_s)\Omega^V) \\
 &=& (\Omega^U|\Phi \gamma_{k_1}(w_1,z_1) \dots \gamma_{k_s}(w_s ,z_s) \Phi^{-1}\Omega^U)  \\
 &=& (\Omega^U|\eta_{k_1}(w_1,z_1) \dots \eta_{k_s}(w_s, z_s)\Omega^U)
\end{eqnarray*}
and hence $\cZ_V = \cZ_U$ by $(ii) \Rightarrow (i)$ in Proposition \ref{prop: equality of partition functions}. 

\medskip

We now assume that $\cZ_V = \cZ_U$. 
If $a \in PV$ is given by 
$$a := \alpha_1 C_{k_1^1}(m_1^1,n_1^1) \dots C_{k_{s^1}^1}(m_{s^1}^1,n_{s^1}^1) \Omega^V + \dots 
+  \alpha_j C_{k_1^j}(m_1^j,n_1^j) \dots C_{k_{s^j}^j}(m_{s^j}^j,n_{s^j}^j) \Omega^V$$   
and $b \in PU$ is given by 
$$b := \alpha_1 D_{k_1^1}(m_1^1,n_1^1) \dots D_{k_{s^1}^1}(m_{s^1}^1,n_{s^1}^1) \Omega^U + \dots 
+  \alpha_j D_{k_1^j}(m_1^j,n_1^j) \dots D_{k_{s^j}^j}(m_{s^j}^j,n_{s^j}^j) \Omega^U$$   
with $k \in \mathbb{Z}_{\geq 0}$, then by $(i) \Rightarrow (ii)$ in  Proposition \ref{prop: equality of partition functions}
$(a|a) = (b|b)$. 
Hence, letting $\varphi a =b$, we obtain a well-defined unitary operator $\varphi: PV \to PU$ such that 

\begin{eqnarray*}
\varphi C_{k_1}(m_1,n_1) C_{k_2}(m_2,n_2) \dots C_{k_s}(m_s,n_s) \Omega^V  = \\
D_{k_1}(m_1,n_1) D_{k_2}(m_2,n_2) \dots D_{k_s}(m_s,n_s) \Omega^U
\end{eqnarray*}
for all $s \in \mathbb{Z}_{\geq 0}$, all $(k_1,\dots,k_s) \in \mathbb{Z}_{\geq 0}^s$ and all $(m_1,n_1, \dots, m_s,n_s) \in \mathbb{Z}^{2s}$.

Now, it follows from the Borcherds identity that, if $C_{a,b}$ and $D_{a,b}$ denote the Casimir elements of $V$ and $U$ respectively,
we have 
\begin{eqnarray*}
\varphi (C_{k,m})_n  C_{k_1}(m_1,n_1) \dots C_{k_s}(m_s,n_s) \Omega^V   = \\ 
(D_{k,m})_n  D_{k_1}(m_1,n_1) \dots D_{k_s}(m_s,n_s) \Omega^U
\end{eqnarray*}

for all $s \in \mathbb{Z}_{\geq 0}$, all $(k,k_1,\dots,k_s) \in \mathbb{Z}_{\geq 0}^{s+1}$, all and all $(m,n,m_1,n_1, \dots, m_s,n_s) \in \mathbb{Z}^{2s+2}$.
Hence, $\varphi C_{k,m} = D_{k,m}$ and $\varphi (C_{k,m})_n a  = (D_{k,m})_n \varphi a =  (\varphi C_{k,m})_n \varphi a$  
for all $k \in \mathbb{Z}_{\geq 0}$, all $m,n \in \mathbb{Z}$ and all $a \in PV$.  Hence, recalling that the Casimir elements $C_{k,j}$, $k,j \in \mathbb{Z}$  generate $PV$, the Borcherds identity implies that $\varphi$ is a vertex algebra isomorphism. We now want to show that $\varphi$ is also a VOA isomorphism, i.e. that $\varphi$ preserves the conformal vector. To this end first note that  
$\varphi$ is grading preserving, namely $\varphi L^V_0 a= L^U_0\varphi a$ for all $a \in PV$. Then the result follows from \cite[Prop. 4.8]{CKLW18}. 

Using the Borcherds identity, it is quite straightforward to see that $(i) \Rightarrow (v)$ in Proposition \ref{prop: equality of partition functions} implies that 
$\Tr_V(a_0q^{L_0}) = \Tr_U(\varphi a)_0q^{L_0})$. Using the unitary  isomorphism $\varphi :PV \to PU$ we can consider the vector space $U$ as a 
$PV$-module $M$ with vertex operators $Y^M(a,z) := Y(\varphi a,z)$. In this way, we find 
$$\Tr_M(a_0^M q^{L_0}) = \Tr_U( (\varphi a)_0q^{L_0}) =\Tr_V(a_0q^{L_0}) \, .$$
Hence, by Proposition \ref{propr: charactermodules}  $M$ and $V$ are unitarily equivalent $V$ modules and hence there is a unitary operator 
$\Phi: V \to U$ such that 
$$\Phi Y(a,z) = Y(\varphi a,z) \Phi$$
for all $a \in PV$. Taking $a$ equal to the conformal vector $\nu^V$, we find 
$L^U_0 \Phi \Omega^V = \Phi L^V_0 \Omega^V = 0$. Hence, being $U$ of CFT type,  $\Phi \Omega^V \in \mathbb{C}\Omega^U$. Accordingly, we can assume that 
$\Phi \Omega^V = \Omega^U$.  Hence, for any $a \in PV$ 
$$\Phi a = \Phi Y(a,z)\Omega^V |_{z=0} = Y(\varphi a,z)\Omega^U|_{z=0} = \varphi a \,.$$

It remains to prove that if $\Phi: V \to U$ is a unitary which restricts to a VOA isomorphism $\varphi: PV \to PU$ and such that 
$$\Phi Y(a,z) = Y(\varphi a,z) \Phi$$
for all $a \in PV$ then  we also have 
$$\Phi \gamma_k(w,z) = \eta_k(w, z) \Phi$$
for every non-negative integer $k$. 
If $\nu^V \in PV$ is the conformal vector of $V$ then $\Phi\nu^V =\varphi \nu^V = \nu^U$ is the conformal vector of $U$. It follows that 
$$\Phi L^V_n = L^U_n \Phi $$ 
for all $n\in \mathbb{Z}$ and hence 
$$\Phi (L^V_n)^jY(a,x)Y(b,y)(L^V_m)^k \Phi^{-1}  = (L^U_n)^jY(\varphi a,x)Y(\varphi b,y))(L^U_m)^k $$ 
for all  $j,k  \in \mathbb{Z}_{\geq 0}$, all $m,n \in \mathbb{Z}$ and all $a,b \in PV$. It follows that, for any non-negative integer $k$, we 
have $\Phi V_k^{qp} = U_k^{qp}$ so that 

$$\Tr_{V_k^{qp}} ((L^V_n)^jY(a,x)Y(b,y)(L^V_m)^k)  = \Tr_{U_k^{qp}} ((L^U_n)^jY(\varphi a,x)Y(\varphi v,y)(L^U_m)^k)$$
for all  $j,k  \in \mathbb{Z}_{\geq 0}$, all $m,n \in \mathbb{Z}$ and all $a,b \in PV$.

Thus, taking a basis $\{v^{(i)} \}$ for $V_k^{qp}$ orthonormal with respect to the invariant bilinear form $(\cdot, \cdot)$ in $V$, we find the following. 

\begin{eqnarray*}
\sum_{i =1}^{\mathrm{dim}(V_k^{qp})}(v^{(i)},(L^V_n)^jY(a,x)Y(b,y) (L^V_m)^kv^{(i)})  = \\
 \sum_{i=1}^{\mathrm{dim}(U_k^{qp})} (u^{(i)},(L^U_n)^j Y(\varphi a,x)Y(\varphi b,y) (L^U_m)^k u^{(i)}) \,,
\end{eqnarray*}

where $\{u^{(i)} \}:= \{ \Phi v^{(i)} \}$ is a basis for $U_k^{qp}$ orthonormal with respect to the invariant bilinear form $(\cdot, \cdot)$ on $U$.
It follows that 

\begin{eqnarray*}
\sum_{i =1}^{\mathrm{dim}(V_k^{qp})}(e^{wL^V_{-1}} v^{(i)},Y(a,x)Y(b,y) e^{zL^V_{-1}}v^{(i)})  = \\
\sum_{i=1}^{\mathrm{dim}(U_k^{qp})} (e^{wL^U_{-1}}u^{(i)},Y( \varphi a,x)Y(\varphi b,y) e^{zL^U_{-1}} u^{(i)})
\end{eqnarray*}
i.e. (see e.g. \cite[Proposition 4.1]{Kac01})
\begin{eqnarray*}
\sum_{i =1}^{\mathrm{dim}(V_k^{qp})}(Y(v^{(i)},w)\Omega^V,Y(a,x) Y(b,y) Y(v^{(i)},z)\Omega^V)  = \\
\sum_{i=1}^{\mathrm{dim}(U_k^{qp})} (Y(u^{(i)},w)\Omega^U,Y(\varphi a,x) Y(\varphi b,y) Y(u^{(i)},z)1)
\end{eqnarray*}
and thus 
\begin{eqnarray*}
\sum_{i =1}^{\mathrm{dim}(V_k^{qp})}(\Omega^V, Y(v^{(i)},w)Y(a,x) Y(b,y) Y(v^{(i)},z)\Omega^V)  =  \\
\sum_{i=1}^{\mathrm{dim}(U_k^{qp})} (\Omega^U, Y(u^{(i)},w)Y(\varphi a ,x) Y(\varphi b,y) Y(u^{(i)},z)1) \,.
\end{eqnarray*}

It follows from the commutativity of VOAs \cite[Prop. 3.5.1]{Axiomatic} that
$$(\Omega^V|Y(a,x)\gamma_k^{qp}(w,z)Y(b,y)\Omega^V) = (\Omega^U|Y(\varphi a,x)\eta_k^{qp}(w,z)Y(\varphi b,y) \Omega^U)$$
for every non-negative integer $k$ and every $a, b \in PV$. Hence, it follows from Eq. (\ref{E:rel_campi}) that
\begin{eqnarray*}
(\Omega^V|Y(a,x)\gamma_k(w,z)Y(b,y)\Omega^V) &=& (\Omega^U |Y(\varphi a,x)\eta_k(w,z)Y(\varphi b,y)\Omega^U) \\
&=& (\Phi \Omega^V |Y(\varphi a,x)\eta_k(w,z)Y(\varphi b,y)\Phi \Omega^V) \\
&=& (\Omega^V |Y(a,x)\Phi^{-1}\eta_k(w,z)\Phi Y(b,y) \Omega^V)
\end{eqnarray*}
for every non-negative integer $k$ and every $a, b \in PV$. Hence, 
$$\Phi \gamma_k (w,z) \restriction_{PV} = \eta_k (w, z) \Phi \restriction_{PV}.$$
It follows that $\Phi C_{k, n} = D_{k,n}$ and hence 
$$\Phi Y(C_{k,n},z)= Y(D_{k,n},z)\Phi $$ 
if $k$ and $n$ are integers. By the associativity property of VOAs \cite[Prop. 3.3.2]{Axiomatic} 
we can conclude that
$$\Phi \gamma_a(w,z)  = \eta_a(w, z) \Phi \, , $$
c.f the proof of Proposition \ref{prop: partition}. 

Let us prove the statement about central charges. Since $\varphi: PV \to PU$ is a VOA isomorphism, it maps the conformal vector of $PV$ into the one of $PU$. Hence $PV$ and $PU$ have the same central charge. The conclusion then follows from the lemma \ref{L:conf}.
\end{proof}

\begin{corollary}
\label{cor: unique structure}
Let $V$ be a simple unitary VOA and assume that the $PV$-module $V$  has a unique, up to isomorphisms, unitary VOA structure. Then  every simple unitary vertex operator algebra $U$ with $\cZ_U =\cZ_V$ is isomorphic to $V$.
\end{corollary}
\begin{proof}
Assume that the $PV$-module $V$  has a unique unitary VOA structure, up to isomorphisms, and let $U$ be a simple unitary VOA with
$\cZ_U =\cZ_V$. By Theorem \ref{T:main2_corpo} there is a unitary operator $\Phi: V \to U$ 
restricting to a unitary VOA isomorphism
$\varphi \colon PV \to PU $
 satisfying
$$\Phi Y(a,z) = Y(\varphi a,z) \Phi$$ 
for all $a \in PV$. For any $a \in V$ we define a new vertex operator  $\widetilde{Y}(a,z)$ by 
$$ \widetilde{Y}(a,z) = \Phi^{-1}Y(\Phi a,z)\Phi . $$
Then, the vertex operators $\widetilde{Y}(a,z)$ give a new VOA structure on $V$ and satisfy
$$ \widetilde{Y}(a,z) = Y(a,z)\ $$
for all $a \in PV$ so that they define a unitary VOA structure on the $PV$-module $V$. By the uniqueness assumption, there is a vector space isomorphism 
$\Psi :V \to V$ such that $\Psi \Omega^V= \Omega^V$, $\Psi \nu^V = \nu^V$ and 
$$ \Psi \widetilde{Y}(a,z) = Y(\Psi a,z)\Psi \,, $$
for all $a \in V$. It follows that $\Psi \Phi^{-1} :U \to V$ is a unitary VOA isomorphism. 

\end{proof}

\section{Explicit examples of partition function  subalgebras}
\label{SectionExamples}
In this section, we give various examples of unitary VOA., not necessarily rational of $C_2$-cofinite,  which are determined by their partition function and of unitary VOAs whose partition function subalgebra can be explicitly identified with a known VOA. We also discuss some problems and conjectures that arise from the analysis of these models.

\begin{example} 
\label{ex:Virasoro} {\it Virasoro unitary VOAs.} 
Let $L(c,0)$ be a simple unitary VOA with central charge $c$ and generated by its conformal vector. Then, either 
 $c = c_m := 1-\frac{6}{m(m+1)}, \, m=2,3,4 \ldots$ ({\it unitary discrete series}) or $c \geq 1$. It follows directly from Lemma \ref{L:conf} that the partition function subalgebra $PL(c,0)$ coincides with $L(c,0)$. Accordingly, if $V$ is a unitary VOA with the same partition function of $L(c,0)$ then  $V$ is isomorphic to $L(c,0)$. Note also that $\operatorname{Aut}(L(c,0))$ is the trivial group so that 
 $PL(c,0) = L(c,0)^{\operatorname{Aut}(L(c,0))}$. 
\end{example}

\begin{example} 
\label{ex:c < 1} {\it Unitary VOAs with $c<1$.}
The simple unitary VOAs with central charge $c < 1$ are all simple CFT type extensions of the discrete series unitary Virasoro VOAs 
$L(c_m,0)$. The latter have been classified by Dong and Lin in \cite{DL14} in close connection with the classification of conformal nets with $c<1$ obtained 
by Kawahigashi and Longo in \cite{KL04}. It turns out that alle these extensions are unitary and are in one-to-one correspondence with the conformal nets with $c<1$, see \cite{Gui22} and \cite[Sec. 6]{CGGH23}. In fact the classification of these VOA extensions can be directly obtained by the classification of the corresponding conformal nets and the results in \cite{CGGH23}, cf. \cite[Thm. 6.1]{CGGH23}. These extensions can be divided in three families. The first is given 
by the Virasoro VOAs $L(c_m,0)$ (trivial extensions), $m=2,3,4 \dots$ that have been considered in Example \ref{ex:Virasoro}. The second family is given by $\mathbb{Z}_2$ simple current extensions of the $L(c_m,0)$ for $m=4n + 1$, $n\geq 1$ or $m=4n+2$, $n\geq 1$. If $V$ is a unitary VOA with central charge $c_m$ in this second family, then $\operatorname{Aut}(V)$ is isomorphic to the order two cyclic group $\mathbb{Z}_2$ and the orbifold subVOA 
$V^{\operatorname{Aut}(V)}$ coincides with $L(c_m,0)$. Accordingly, we again have the equality $PV= V^{\operatorname{Aut}(V)}$. The third family consists of four exceptional extensions corresponding to $m= 11,12,29,30$. For any such exceptional extension $V$ one can use the classification in  \cite{KL04} together with the fact that the Jones index for the corresponding subfactor is not an integer, see \cite[Section 4]{KL04}, to show that 
$\operatorname{Aut}(V)$ is trivial so that $V^{\operatorname{Aut}(V)}=V$ and we conjecture that we have again $PV= V^{\operatorname{Aut}(V)}$ for these four exceptional VOAs. In any case if $V$ is any unitary VOA in the above three families having central charge $c_m$ then, as a consequence of the classification results in \cite{DL14,KL04,CGGH23} the isomorphism class of $V$ is completely determined by its $L(c_m,0)$-module structure and hence by its central charge and its genus one partition function. Accordingly, if $U$ is a unitary VOA and $\cZ_U = \cZ_V$ then 
$U$ is isomorphic to $V$.

\end{example}

\begin{example}
\label{ex:W3} {\it Unitary $\mathcal{W}_3$ algebras.} Let $W^c(3)$ be the universal VOA associated to the $\mathcal{W}_3$-algebra with central charge $c$ 
and let $W_c(3)$ be its simple quotient. Then $W_c(3)$ is a unitary VOA if and only if  either  $c = \tilde{c}_m:= 2(1-\frac{12}{m(m+1)}),\, m=3,4,5 \ldots$ ({\it unitary discrete series}) or $c \geq 2$, c.f. \cite{CTW23}. Now, let $c$ be such that $W_c(3)$ is unitary. Then  $W_c(3)$ is generated by the conformal vector $\nu$  and a primary vector  $\xi \in W_c(3)_3$ such that $\theta \xi = - \xi$.  In fact $W_c(3)$ is spanned by vectors of the form 
$$\nu_{-m_1} \dots \nu_{-m_r}\xi_{-n_1} \dots \xi_{-n_s}\Omega$$
with $m_1 >m_2 > \dots >m_r > 1$ and $n_1 >n_2 > \dots >n_s > 2$, see e.g.  \cite[Appendix A]{CTW23}. The automorphism group  
$\operatorname{Aut}(W_c(3))$ is generated by a single order-two element $g$ such that $g\xi=-\xi$. In particular $\operatorname{Aut}(W_c(3))$ is isomorphic to $\mathbb{Z}_2$.  We can fix an orthonormal basis  $v^{(i)}$, $i=1,2$ for $W_c(3)_3$ with 
$v^{(1)} = \frac{1}{\| \nu_{-3}1\|} \nu_{-3}\Omega$ and  $v^{(2)}= \frac{i}{\|\xi\|}\xi$.    
It follows that for any $a\in P W_c(3)$ the coefficients of the series 
$$Y(\xi,w)Y(\xi,z)a = \|\xi\|^2\gamma_3(w,z)a -  \left(\frac{\|\xi\|}{\|\nu_{-3} \Omega\|}\right)^2Y(\nu_{-3}\Omega,w) Y(\nu_{-3}\Omega,z)a$$
belongs to the partition subalgebra $P W_c(3)$. Accordingly, being $W_c(3)^{\mathbb{Z}_2}$ spanned by the coefficients of the series 
$$Y(\nu,t_1) \dots Y(\nu, t_m) Y(\xi,z_1)Y(\xi,w_1) \dots Y(\xi,z_n)Y(\xi,w_m)\Omega \, ,$$ 
we can conclude that $W_c(3)^{\mathbb{Z}_2} \subset P W_c(3)$ and hence that $W_c(3)^{\mathbb{Z}_2} = P W_c(3)$.  Now, let $V$ be a simple unitary VOA with the same partition function of $W_c(3)$ and let us denote by $\tilde{\gamma}_k(w,z)$, $k \in \mathbb{Z}_{\geq 0}$ the Casimir bilocal fields of $V$. By Theorem \ref{T:main2_corpo} there is a unitary map $\Phi: W_c(3) \to V$ which restricts to a unitary VOA  isomorphism $\varphi: P W_c(3) \to PV$. Then, 
$\tilde{\xi}:= \Phi \xi$ is a primary vector in $V_3$. Moreover, 
\begin{eqnarray*}
Y(\tilde{\xi},w)Y({\tilde \xi},z) &=& \|\xi \|^2\tilde {\gamma}_3(z,w) -  \left(\frac{\|\xi\|}{\|\nu_{-3}\Omega\|}\right)^2Y(\nu^V_{-3}\Omega^V,w) 
Y(\nu^V_{-3}\Omega^V,z) \\
&=& \Phi \left(  \|\xi\|^2 \gamma_3(w,z) -  \left( \frac{\|\xi\|}{\|\nu_{-3}\Omega\|}\right )^2 Y(\nu_{-3}\Omega,w) Y(\nu_{-3}\Omega,z) \right) \Phi^{-1} \\
&=& \Phi Y(\xi,w)Y(\xi,z) \Phi^{-1} \,.
\end{eqnarray*}
Hence, for the commutator we have 
$$ \left[Y(\tilde{\xi},w),Y({\tilde \xi},z) \right] =  \Phi \left[Y(\tilde{\xi},w),Y({\tilde \xi},z) \right] \Phi^{-1}\,.$$ 
so that the fields $Y(\tilde{\xi},z)$ and $Y(\nu^V,z)$ generate a unitary subalgebra $U$ of $V$ isomorphic to $W_c(3)$. 
On the other hand $U = \Phi  W_c(3) = V$ so that $V$ and  $W_c(3)$ are isomorphic. 
\end{example}

Before giving further examples, we describe some general results concerning VOAs with non-zero weight one subspace. For a given simple complex Lie algebra 
$\mathfrak{g}$ and complex number $k$ we denote by $V^k(\mathfrak{g})$ the universal vertex algebra associated with $\mathfrak{g}$ at level $k$ and by  
$V_k(\mathfrak{g})$ its simple quotient. If $h^\vee$ denotes the dual Coxeter number of $\mathfrak{g}$ and $k\neq - h^\vee$ then $V_k(\mathfrak{g})$
is a simple VOA. Moreover, $V_k(\mathfrak{g})$ is unitary if and only if it is strongly rational if and only if $k$ is a non-negative integer. The case $k=0$ corresponds to the trivial one-dimensional VOA so that $k>0$  for non trivial VOAs. We also denote by $M_{r}(1)$ the rank $r$ Heisenberg VOA. $M_0(1)$ denotes the trivial VOA. Note that $M_{r}(1)= M(1)^{\otimes r}$, where $M(1)=M_1(1)$ id the rank one Heisenberg VOA. $M_r(1)$ is a simple unitary VOA for any rank $r$.

First, recall that if $V$ is a VOA then the one-way weight subspace $V_1$ is a complex Lie algebra with brackets $[a,b]= a_0b$, $a,b \in V$. 
We denote by $V_{V_1}$ the VOA subalgebra of $V$ generated by $V_1$. It is known that if $V$ is strongly rational, then $\mathfrak{g}:=V_1$ is reductive (a direct sum of a semisimple Lie algebra and an abelian Lie algebra); see \cite[Thm. 1.1]{DM04b}. The same is true for every unitary VOA $V$. This follows from the fact that the real subspace 
$$\mathfrak{g}_{\mathbb{R}}:= \{a \in \mathfrak{g}: \theta a = a \}$$ is real form of $\mathfrak{g}$ and the restriction of the normalized invariant bilinear form   on $V$ to $\mathfrak{g}_{\mathbb{R}}$ is positive-definite. Now, let $V$ be unitary and $\mathfrak{g}=V_1$. Then 
$$\mathfrak{g} = \mathfrak{g}_1 \oplus \dots \oplus \mathfrak{g}_j \oplus \mathfrak{h}$$
with $ \mathfrak{g}_1, \dots \mathfrak{g}_j$ semisimple and $\mathfrak{h}$ abelian. Moreover, $V_{V_1}$ is a unitary subalgebra of $V$ and is isomorphic to the unitary affine VOA  
\begin{equation}
\label{Eq: Affine}
V_{k_1}(\mathfrak{g}_1)\otimes \dots \otimes V_{k_j}(\mathfrak{g}_j) \otimes M_r(1)
\end{equation}
for suitable positive integer level $k_1, \dots, k_j$ and $r=\mathrm{dim}(\mathfrak{h})$.

The next result follows rather directly by the work of H\"{o}hn on conformal design \cite{Hoh08}, also c.f. \cite{Hur02,Hur06}, together with our Proposition 
\ref{P:ort}.
\begin{theorem}
\label{th:design}
Let $V$ be a non-trivial simple unitary with $V_1 \neq \{ 0 \}$ and central charge $c$. If $PV = L(c,0)$ then $V$ is isomorphic to $V_1(\mathfrak{sl}_2)$.  
\end{theorem}
\begin{proof}
Assume that $PV = L(c,0)$. Then, from Proposition \ref{P:ort} it follows that the weight one subspace $V_{1}$ is a conformal $t$-design in the sense of \cite[Sec. 2]{Hoh08} for all $t >0$. Consequently, by \cite[Thm. 4.1 (b)]{Hoh08}, $V$ is isomorphic either to $V_1(\mathfrak{sl}_2)$ or to 
$V_1(\mathfrak{e}_8)$. On the other hand  it follows from \cite[Cor. 5.2.1]{DMN01} and by Corollary \ref{cor: projection partition} that 
$P V_1(\mathfrak{e}_8)$ contains non-zero primary fields of positive weight and hence cannot be isomorphic to  $L(c,0)$. 
\end{proof}

The following theorem has been proven in \cite[Sec. 5]{GV09}. Although the authors of \cite{GV09} appear to assume that the VOA $V$ is holomorphic, a careful inspection of their proof shows that this assumption is not essential. What appears to be essential is the fact that $V_1$ is a complex reductive Lie algebra that is always true for a unitary $V$.

\begin{theorem}
\label{th:V_{V_1}}
Let $V$ and $U$ be two simple unitary VOAs with $\cZ_V = \cZ_U$. Then, the unitary affine VOAs $V_{V_1}$ and $U_{U_1}$ are isomorphic. 
\end{theorem}

We will also need the following.

\begin{theorem}
\label{M(1)extension}
Let $V$ be a simple unitary conformal (i.e. with the same conformal vector) extension of the rank $r$ Heisenberg VOA $M_r(1)$. Then there is a rank $s$ even positive-definite lattice $L$, with $s \leq r$ such that $V$  is isomorphic to $V_L \otimes M_{r-s}(1)$.  
\end{theorem}

\begin{proof}
We can choose an orthonormal basis $\{h^{(1)},\dots, h^{(r)}\}$ for  $M_r (1)_1$, with respect to the given invariant scalar product $(\cdot |\cdot)$ on $M_r(1)$, such that $\theta  h^{(j)} = h^{(j)}$, $j=1,\dots r$. The operators 
$\{h^{(1)}_0,\dots, h^{(r)}_0\}$ restrict to skew-adjoint operators on each $V_n$ and hence are jointly diagonalizable with immaginary eigenvalues. 
For any $\lambda = (\lambda_1,\dots, \lambda_r) \in \mathbb{R}^r$ let   
$$V^\lambda := \{a \in V: h^{(j)}_0 a = i \lambda_j  a , \, j=1,\dots,r \} \,.$$
Then each $V^\lambda$ is a (possibly 0)  $M_r(1)$-module. 
$$V = \bigoplus_{\lambda \in L} V^\lambda $$
where 
$$L :=\{\lambda \in  \mathbb{R}^r : V^\lambda \neq \{ 0 \}\} \,. $$  
Note that $\Omega \in V^0$ so that $0\in L$ . 
Moreover, if $\lambda \in L$ then $V^\lambda$ is a direct sum of equivalent irreducible  $M_r(1)$-modules. 
If $a$ is a $M_r(1)$ highest weight vector in $V^\lambda$ then $L_0a = \frac{(\lambda,\lambda)}{2}a$ where
$(\lambda,\lambda) = \sum_{j=1}^r \lambda_j^2$. It follows that $V^0 = M_r(1)$. Moreover,  $(\lambda,\lambda) \in 2\mathbb{Z}_{\geq 0}$ and 
so that $(\lambda,\lambda)\geq 1$ for all $\lambda \in L \setminus \{0\}$.  

Each $h^{(j)}_0$ is a derivation of $V$ and, in fact, $e^{h^{(j)}_0}$ is a unitary automorphism of $V$. 
If $a \in   V^\alpha$ and $b \in V^\beta$ are non zero then it can be shown that  $a_{(m)}b_{(k)}1$ is non-zero for some $m, k \in \mathbb{Z}$. 
Furthermore, $h^{(j)}_0 a_{(m)}b_{(k)}1 = (\alpha_j + \beta_j) a_{(m)}b_{(k)}1$ and $h^{(j)}_0 \theta a = \theta(i \alpha_j)a = -i\alpha_j\theta a$, 
$j=1,\dots,r$. It follows that $L$ is a discrete subgroup of $\mathbb{R}^r$, c.f. \cite[Sec. 3.1]{CKLR19}. 
Let $s$ be the dimension of $\mathrm{span} L $. Then $L$ is a rank $s$ even positive definite lattice.  Without loss of generality, we can assume that  
$\lambda_j =0$ for all $\lambda \in L$ and all  $j > s$. Accordingly, we write $M_r(1)= M_s(1) \otimes M_{r-s}(1)$
The map  $$\mathbb{R}^s \ni (t_1,\dots, t_s) \mapsto e^{2\pi (t_1h^{(1) }_0 + \dots + t_sh^{(s)}_0)}$$
defines an action of the torus $\mathbb{R}^s / L$ into the unitary automorphism group of $V$ and the fiexed point subalgebra 
$V^{\mathbb{R}^s / L}$ coincides with $M_r(1)$. It follows from \cite[Thm. 1]{DM99} that for any $\lambda \in L$, 
$V^\lambda$ is an irreducible $M_r(1)$-module. Moreover, the irreducible  $M_s(1) \otimes M_{r-s}(1)$-module $V_\lambda$ is equivalent to 
$U^\lambda \otimes  M_{r-s}(1)$, where $U^\lambda$ is the irreducible $M_s(1)$-module in which each $h^{(j)}_0$ acts as the scalar
$\lambda_j$, $j=1, \dots s$.

It follows that, as a  $M_s(1) \otimes M_{r-s}(1)$-module $V$ is equivalent to 
$U \otimes M_{r-s}(1)$ with 
$$U:= \bigoplus_{\lambda \in L}U^\lambda $$
and we denote by $\varphi: V \to  U \otimes M_{r-s}(1)$ a corresponding equivalence. It is clear that $\varphi$ maps the coset subalgebra 
$\left(1 \otimes M_{r-s}(1)\right)^c$ onto $U$ so that $U$ admits a VOA structure such that $V$ is isomorphic to $U \otimes M_{r-s}(1)$. 
Now, by  \cite[Corollary 5.4]{DM04b}, $U$ is isomorphic to the lattice VOA $V_L$ and the conclusion follows. 
\end{proof}

\begin{example}
\label{ex: Heisenberg} 
{\it Heisenberg VOAs}. Let $M_r(1)$ be the rannk $r$ Heisenberg VOA. It is a simple unitary VOA, cf. \cite{CKLW18,DL14}.Then we can identify 
$\operatorname{Aut}(M_r(1))$ with the orthogonal group $\mathrm{O}(r)$, se e.g. \cite{Lin12}. By \cite[Lemma 4.2]{Lin12} the compact orbifold subalgebra $M_r(1)^{\mathrm{O}(r)}$ is generated by the Casimir 
element $C_{1,3}$.  It follows that   $M_r(1)^{\mathrm{O}(r)} \subset PM(1)$ and hence $M(1)^{\mathrm{O}(r)} = PM(1)$. Now, let $V$ be a simple unitary
VOA with $\cZ_V = \cZ_{M_r(1)}$. Then, by  Theorem \ref{th:V_{V_1}} $V_{V_1}$ is isomorphic to $M_r(1)_{M_r(1)_1}=M_r(1)$. Moreover, 
$V$ and $M_r(1)$ have the same genus one partition function and hence the same vacuum character. Hence $V_{V_1}=V$ so that $V$ is isomorphic to $M_r(1)$.
\end{example}

\begin{example} 
\label{ex: lattice}
{\it Lattice VOAs.} Let $L$ be an even positive-definite Lattice or rank $r>0$. Then the corresponding VOA $V_L$ is unitary, see e.g. 
\cite{CKLW18,DL14}.  Let $V$ be a simple unitary VOA with $\cZ_V = \cZ_{V_L}$. Then, by Theorem \ref{T:main2_corpo} the central charge of $V$ is equal to that of $V_L$ which is equal to $r$.

By  Theorem \ref{th:V_{V_1}}, the unitary affine algebras generated by $V_1$ in $V$ and by $(V_L)_1$ in $V_L$ are isomorphic, so $V$ is, up to isomorphisms, a local extension of the rank $r$  Heisenberg VOA 
$M_r(1)$. By Theorem \ref{M(1)extension} there is a non-negative integer $s\leq r$ such that  $V$ is isomorphic to $V_{\tilde{L}} \otimes M_{r-s}(1)$ with 
$\tilde{L}$ a positive-definite even lattice of rank $s$. By the equality of the genus one partition functions, we have 
$$\operatorname{Tr}_V e^{2\pi i  \tau L^V_0} = \operatorname{Tr}_{V_L} e^{2\pi i  \tau L^{V_L}_0}$$ or, equivalently 
$$ \Theta_{\tilde{L},1}(e^{2\pi i  \tau})  \operatorname{Tr}_{M_r(1)}e^{2\pi i  \tau L^{M_r(1)}_0} = 
 \Theta_{L,1}(e^{2\pi i  \tau}) \operatorname{Tr}_{M_r(1)}e^{2\pi i  \tau L^{M_r(1)}_0} \,.  $$

 Hence the genus one theta series $\Theta_{L,1}(e^{2\pi i  \tau})$ and 
$\Theta_{\tilde{L},1}(e^{2\pi i  \tau})$ coincide, i.e. $L$ an $\tilde{L}$ are isospectral. In particular $s=r$ (this follows e.g. from Weyl asymptotyc formula for the Laplacian on a torus) so that $V$ is isomorphic to $V_{\tilde{L}}$. If $L$ is unimodular, equivalently $V_L$ is holomorphic,  then $r$ is a positive integer multiple of $8$ and $\Theta_{\tilde{L},1}(e^{2\pi i  \tau}) = \Theta_{L,1}(e^{2\pi i  \tau})$ is a modular form of weight $\frac{r}{2}$. It follows that also $\tilde{L}$ must be unimodular. To see this note that by 
\cite[Proposition 16, Section 6.2]{Serre} we have 
$$ 1 =  \lim_{t\to +\infty} \Theta_{\tilde{L},1}(e^{-2\pi t}) = \frac{1}{ ( \operatorname{det} G_{\tilde{L}} )^{\frac12} }$$
 where $G_{\tilde{L}}$ is the Gram matrix  of $\tilde{L}$. Hence, $\operatorname{det} G_{\tilde{L}}=1$ so that  $\tilde{L}$ is unimodular. This shows that $V$ is a holomorphic lattice VOA.

 Let us also show that $L$ is isomorphic to $\tilde{L}$. The equality $\cZ_V=\cZ_{V_L}$ together with Theorem \ref{thm:partition_of_lattice} shows that, for every $g$, the pull-back via the Jacobi map of the degree $g$ theta series associated to $L$ is equal to the pull-back of the degree $g$ theta series associated to $\tilde{L}$. Now  \cite[Corollary 1.5]{CS-B14} implies that $L$ and $\tilde{L}$ are isomorphic. 
 
 (Note also that two isospectral positive-definite even lattices, not necessarily unimodular, of rank $r \leq 3$ must be isomorphic by \cite{Schi97}. Hence, if $r\leq 3$ we again find that $V$ is isomorphic to $V_L$ without assuming unimodularity, cf. Example \ref{ex: rank-one lattice} below. )
\end{example}

We resume the discussion in Example \ref{ex: lattice} in the following theorem. 

\begin{theorem} 
\label{thmLatticePartition} 
Let $L$ be a unimodular positive definite even lattice and let $V_L$ be the corresponding lattice VOA. Moreover, let  $V$ be a simple unitary VOA with 
$\mathcal{Z}_V = \mathcal{Z}_{V_L}$. Then $V$ is isomorphic to $V_L$. 
\end{theorem}

\begin{example}
\label{ex: rank-one lattice}
{\it Rank-one lattice VOAs}. If $L$ is a rank-one positive definite even lattice then $L$ is isomorphic to $L_{2n}:=\sqrt{2n}\mathbb{Z}$ for some 
$n \in \mathbb{Z}_{>0}$. We see from  Example \ref{ex: lattice} that in the rank-one case the partition function detemines the VOA up to isomorphism. 
We now want to find an explicit description of the partition function subalgebra in these cases. Note that the unitary subalgebras of the Lattice VOAs 
$V_{L_{2n}}$,  $n \in \mathbb{Z}_{>0}$, have been completely classified in \cite{CGH19}. For any non-negative integer n, let $M(1)$ be the rank-one unitary Heisenberg subalgebra   of $V_{L_{2n}}$ and let $M(1)^{+}$ be the $\mathbb{Z}_2$-orbifold $M(1)^{\mathrm{O}(1)}$. Then,  
$ V_{L_{2n}}^{\operatorname{Aut}(V_{L_{2n}})} \subset M(1)^+$ by \cite[Prop. 2.9]{CGH19}. If $n=1$ we have, up to group isomorphisms,  
$\operatorname{Aut}(V_{L_2}) = \mathrm{PSL}(2,\mathbb{C})$ and $\operatorname{Aut}_{(\cdot I \cdot)}(V_{L_2}) = \mathrm{SO}(3)$. Moreover,
$$V_{L_2}^{\mathrm{PSL}(2,\mathbb{C})} = V_{L_2}^{\mathrm{SO}(3)} = L(1,0)\, ,$$ 
see \cite[Theorem 3.2]{CGH19} and \cite[Prop. A.3]{CGGH23}. Hence, 
$$PV_{L_{2}}= V_{L_{2}}^{\operatorname{Aut}(V_{L_{2}})}= L(1,0) \,.$$
Note that  $V_{L_{2}}$ is isomorphic  to $V_1(\mathfrak{sl}_2)$ and compare with Theorem \ref{th:design}.

For $n>1$ the lattice $L_{2n}$ has no roots and therefore $\left(V_{L_{2n}}\right)_1 = M(1)_1 = \mathbb{C}h$ for some $h\in M(1)_1$ 
with $\theta h=h$ and $(h|h)=1$.   Hence, the Casimir element $C_{1,3}$ of $V_{L_{2n}}$ is equal to $h_{-3}h$ and by 
\cite[Lemma 4.2]{Lin12} it generates the compact orbifold subalgebra $M(1)^{\mathrm{O}(1)}= M(1)^{+}$. It follows that 
$$M(1)^+ \subset P V_{L_{2n}} \subset  V_{L_{2n}}^{\operatorname{Aut}(V_{L_{2n}})} \subset M(1)^{+}$$ and hence 
$$ P V_{L_{2n}} =  V_{L_{2n}}^{\operatorname{Aut}(V_{L_{2n}})} = M(1)^{+} \,.$$ 
In conclusion we have $P V_{L} =  V_{L}^{\operatorname{Aut}(V_{L})}$ for every positive-definite rank-one even lattice $L$. 
\end{example}

\begin{example}
\label{ex: unitary affine}
{\it Unitary affine VOAs}. We say that the simple unitary VOA $V$ is affine if it is generated by $V_1$, i.e. $V= V_{V_1}$. Consequently, a simple non-trivial unitary VOA $V$ is affine if and only if it is isomorphic to a VOA of the form in Eq. (\ref{Eq: Affine}). In particular, if $\mathfrak{g} := V_1$ is a non-zero simple complex Lie algebra  then $V$ is isomorphic to $V_k(\mathfrak{g})$ for some positive integer level $k$. It follows directly from Theorem \ref{th:V_{V_1}} that if $V$ is an affine simple unitary VOA and if $U$ is a simple unitary  VOA with $\cZ_U = \cZ_V$ then $V$ is isomorphic to $U_{U_1}$. Moreover, arguing as in Example  \ref{ex: Heisenberg} we see that $U = U_{U_1}$ so that $U$ and $V$ are isomorphic. In other words, the partition function determines the VOA completely. We have already seen in Example \ref{ex: Heisenberg} that, if $V_1$ is abelian, then $PV = V^{\operatorname{Aut}(V)}$. The same holds for 
$V = V_1(\mathfrak{sl}_2)$ as discussed in Example \ref{ex: rank-one lattice}. 
We are not able to give an explicit description of  the partition function algebra of an affine simple unitary affine algebra in general. However, using the classical Weyl invariant theory,  we can give a complete description of $PV$ when $V= V_k(\mathfrak{sl}_2)$ for an arbitrary positive integer $k$. 
As in the special case $k=1$, see Example \ref{ex: rank-one lattice}, we have 
$\operatorname{Aut}(V_k(\mathfrak{sl}_2))=\mathrm{PSL}(2,\mathbb{C})$ and 
$\operatorname{Aut}_{(\cdot | \cdot)}(V_k(\mathfrak{sl}_2)) = \mathrm{SO}(3)$. Moreover, it follows by  \cite[Theorem 3.2]{CGH19} that 
$$ V_k(\mathfrak{sl}_2)^{\mathrm{PSL}(2,\mathbb{C})} = V_k(\mathfrak{sl}_2)^{\mathrm{SO}(3)} \,.$$
Let $\{v^{(1)}, v^{(2)}, v^{(3)} \}$ be an orthonormal basis of $V_k(\mathfrak{sl}_2)_1$ such that $\theta v^{(i)}= v^{(i)}$, $i=1,2,3$.
From classical invariant theory it follows that $V_k(\mathfrak{sl}_2)^{\mathrm{SO}(3)}$ is spanned by the coefficients of the series of the form 
$$\gamma_1(z_1,w_1) \dots \gamma_1(z_j,w_j)\Gamma(x^1_1,x^2_1,x^3_1,x^4_1 ) \dots \Gamma(x^1_m,x^2_m,x^3_m,x^4_m )1$$
where 
$$
\Gamma(x^1,x^2,x^3_1,x^4):= \sum_{i,j= 1}^3 [Y(v^{(i)},x_1), Y(v^{(j)},x_2)] Y(v^{(i)},x_3)Y(v^{(j)},x_4)
$$
takes into account of the determinant invariant, cf. \cite[Sec. 8]{Lin13} and  \cite{dBFH94}. Now, as an operator valued rational function 
 $$\sum_{i,j= 1}^3 Y(v^{(i)},x_1)Y(v^{(j)},x_2) Y(v^{(i)},x_3)Y(v^{(j)},x_4)$$ 
 ``agree" with  
$$\gamma_1(x_1,x_3)\gamma_1(x_2,x_4)$$
and  
$$\sum_{i,j= 1}^3 Y(v^{(j)},x_2)Y(v^{(i)},x_2) Y(v^{(i)},x_3)Y(v^{(j)},x_4)$$ 
 ``agree" with  
$$\gamma_1(x_2,x_4)\gamma_1(x_x,x_3) \, .$$

It follows that, for any $a\in PV_k(\mathfrak{sl}_2)$, the coefficients of $\Gamma(x^1,x^2,x^3_1,x^4)a$ belong to $PV_k(\mathfrak{sl}_2)$. 
As a consequence $V_k(\mathfrak{sl}_2)^{\mathrm{SO}(3)} \subset PV_k(\mathfrak{sl}_2)$ and hence 
$PV_k(\mathfrak{sl}_2) = V_k(\mathfrak{sl}_2)^{\operatorname{Aut}(V_k(\mathfrak{sl}_2))}$.

\end{example}

The examples of holomorphic VOAs with $c=24$ will be discussed in Section \ref{sec: c=24}. We conclude this section with the following problem motivated by the previous examples and the results in the next section.

\begin{problem} 
\label{problem: Aut(V) orbifold}
Find a simple unitary VOA $V$ with $PV \neq V^{\operatorname{Aut}(V)}$. 
\end{problem} 

Note that, by \cite[Prop. A.3]{CGGH23}, the orbifold subalgebra $V^{\operatorname{Aut}(V)}$ coincides with the compact orbifold 
$V^{\operatorname{Aut}_{(\cdot | \cdot)}(V)}$. Problem \ref{problem: Aut(V) orbifold} appears to be deeply related to the question of whether or not a simple unitary 
VOA is determined, up to isomorphism, by its partition function $\cZ_V$. In fact we expect that $PV = V^{\operatorname{Aut}(V)}$ then $V$ is determined by 
$\cZ_V$ up to isomorphisms. This is because of Corollary \ref{cor: unique structure} and of the following conjecture. 

\begin{conjecture}
\label{conjecture: compact orbifold}
Let $V$ be a simple unitary VOA and let $G$ be a closed subgroup of  $\operatorname{Aut}_{(\cdot | \cdot)}(V)$. Then the unitary $V^G$-module $V$ has a unique unitary VOA structure up to isomorrphisms. 
\end{conjecture}

\begin{remark} If $\operatorname{Aut}(V)$ is a Lie group, which is the case if e.g. $V$ is finitely generated, and  if $G$ is a compact subgroup of 
of  $\operatorname{Aut}(V)$, then, modifying if necessary the invariant scalar product, one can always assume that 
$G \subset \operatorname{Aut}_{(\cdot | \cdot)}(V)$, see \cite[Prop. A.5]{CGGH23}
\end{remark}

One motivation for Conjecture \ref{conjecture: compact orbifold} comes from the correspondence between unitary VOAs and conformal nets whose study began in \cite{CKLW18}. For the theory of conformal nets and their representations, we refer the reader to \cite{Car04,CCHW13,CKLW18,KL04} and the references therein. In what follows, we will always assume that the conformal nets and their representations act on separable Hilbert spaces. Moreover, we will 
always assume that the conformal nets are irreducible.

Let $\mathcal{A}$ be an irreducible conformal net. We say that a representation $\pi$ of $\mathcal{A}$ on a Hilbert space 
$\mathcal{H}$ admits a conformal net structure if there is a local extension $\mathcal{B} \supset \mathcal{A}$ on $\mathcal{H}$ such that $\pi$ is the restriction to $\mathcal{A}$ of the of the vacuum representation of $\mathcal{B}$. This definition should be intended as the conformal net analogous to the existence of a VOA structure for a VOA module.  The following theorem is a conformal net version of Conjecture \ref{conjecture: compact orbifold}. It is based on the Dopliche-Roberts duality theory for compact groups \cite{DR89a,DR89b}, see also \cite{DR90}, and can bee seen as a conformal net version of the uniqueness 
of the Doplicher-Roberts field algebra in \cite{DR90}, see also \cite[Sec. 3]{Car04}.

\begin{theorem}
\label{thm: conformal net orbifold}
Let $\mathcal{A}$ be a conformal net, and let $G$ be a compact group of automorphisms of $\mathcal{A}$. 
Let $\mathcal{A}^G$ be the correponding fixed-pont subnet and let $\pi$ be the restriction to  $\mathcal{A}^G$ of the vacuum representation of $\mathcal{A}$. Then, up to isomorphisms, the only conformal net structure of $\pi$ is the one given by $\mathcal{A}$. 
\end{theorem}
\begin{proof}
We closely follow the proof of \cite[Theorem 3.5]{Car04}. 
Let $\pi_0$ be the vacuum representation of  $\mathcal{A}^G$ and fix an interval $I$ of the circle $S^1$. Then, there is a DHR endomorphism $\rho$ of $\mathcal{A}^G$, localized in $I$ such that $\pi$ is unitarily equivalent to $\pi_0 \circ \rho$. Let $\Delta$ be the semigroup of DHR endomorphisms of 
$\mathcal{A}^G$ localized in $I$, consisting of the DHR endomorphisms localized in $I$ and equivalent to a finite direct sum of irreducible endomorphisms contained in $\rho$. Then, up to isomorphisms,  $\mathcal{A}$ can be recovered from $\mathcal{A}^G$ as a crossed product $\mathcal{A}^G \rtimes \Delta$.  Let $\mathcal{B}$ be a conformal net giving a conformal net structure for $\pi$.  As in the proof of \cite[Theorem 3.5]{Car04} the $\alpha$-induction gives a semigroup $\alpha_\Delta$ of DHR endomorphisms of $\mathcal{B}$ localized in $I$. Moreover, 
$$\mathcal{A} = \mathcal{A}^G \rtimes \Delta \subset \mathcal{B} \rtimes \alpha_\Delta .$$ 
Note that each irreducible endomorphism $\sigma$ contained in $\rho$ appears with multiplicity $n_\rho$ equal to the dimension $d(\sigma)$ of $\sigma$. Hence, it follows from the results in \cite[Sec. 3]{ILP98}, see in particular the paragraph before \cite[Thm. 3.3]{ILP98}, that every endomorphism $\sigma \in \Delta$ can be realized  by a Hilbert space of isometries  
 $\mathscr{H}_\sigma \subset \mathcal{B}(I)$ with support of support $s(\mathscr{H}_\sigma)=1$ and hence $\alpha_\sigma$ is inner in $\mathcal{B}$, i.e. implemented by $\mathscr{H}_\sigma$. It follows that $\mathcal{B} \rtimes \alpha_\Delta = 
\mathcal{B}$ and hence $\mathcal{B}$ is isomorphic to $\mathcal{A}$. 
\end{proof}

As a consequence, we obtain the following variant of Conjecture \ref{conjecture: compact orbifold}. 

\begin{theorem}
\label{thm: strongly local compact orbifold}
Let $V$ be a simple unitary VOA which is also strongly local in the sense of \cite{CKLW18} and let $G$ be a closed subgroup of 
$\operatorname{Aut}_{(\cdot | \cdot)}(V)$. Then the unitary $V^G$-module $V$ has, up to isomorphisms, a unique structure of simple strongly local VOA. 
\end{theorem}
\begin{proof} Let $U$ be a strongly local VOA containing $V^G$ as a unitary subalgebra and such that the unitary $V^G$ module $U$ coincides with $V$. 
$U$ and $V$ coincide as vector spaces and have the same scalar product because they coincide as unitary  $V^G$-modules. Hence, they have the same Hilbert space completion 
$\mathcal{H}_U = \mathcal{H}_V$. Let $\mathcal{A}_V$ $\mathcal{A}_U$ the conformal nets corresponding to $V$ and $U$ respectively. By 
\cite[Theorem 6.9]{CKLW18} we can identify $G$ with a compact subgoup of $\operatorname{Aut}(\mathcal{A}_V)$. By \cite[Thm. 7.1]{CKLW18} the compact orbifold $V^G$ is also strongly local, and by \cite[Prop. 7.6]{CKLW18} $\mathcal{A}_V^G = \mathcal{A}_{V^G}$. Let $\pi$ denote the restriction to 
$\mathcal{A}_{V^G}$ of the vacuum representation of $\mathcal{A}_V$. Then, since $V^G$ is also a unitary subalgebra of $U$, it follows by \cite[Thm. 7.5]{CKLW18} that $\mathcal{A}_U$ gives a conformal net structure to $\pi$ and hence $\mathcal{A}_U$ and $\mathcal{A}_V$ are isomorphic conformal nets and it follows from \cite[Theorem 9.2]{CKLW18} that $U$ and $V$ are isomorphic VOAs.
\end{proof}

\begin{remark} It is conjectured that all simple unitary VOA are strongly local, see \cite[Conjecture 8.18]{CKLW18}. If this is the case, Conjecture \ref{conjecture: compact orbifold} would follow from Theorem \ref{thm: strongly local compact orbifold}.
\end{remark}

With suitable assumptions on the representation theory of the compact orbifold $V^G$ and if the compact group $G$ is either finite or abelian then
Conjecture \ref{conjecture: compact orbifold} can be proved directly in the VOA setting thanks to the results in \cite{DM04b,HKL15,KO02,McR20}.  Here, we assume unitarity to simplify our statement, although one could assume weaker conditions. We refer the reader to \cite{CKLR19,CKM24,HKL15,McR20} for further details and definitions.  

\begin{theorem}
\label{thm VOA compact orbifold}
Let $V$ be a simple unitary VOA and let $G$  be a closed subgroup of $\operatorname{Aut}_{(\cdot | \cdot)}(V)$. Assume that all irreducible $V^G$-submodules of $V$ belong to some vertex tensor category $\mathcal{C}$ of $V^G$-modules. Then, if $G$ is abelian or finite, the $V^G$-module $V$ has a unique VOA structure of up isomorphisms. In particular if $G$ is finite and $V^G$ is strongly rational then the $V^G$-module $V$ has a unique structure of simple unitary VOA up isomorphisms.
\end{theorem}
\begin{proof} Let $U$ be a simple unitary VOA that coincides with $V$ as a $V^G$-module. Then $U$ is of CFT type. 
If $G$ is abelian it follows from \cite[Corollary 4.8]{McR20} that $U$ is a simple current extension by the Ponryagin dual $\hat{G}$ of $G$, cf. also  \cite[Example 4.11]{McR20} and \cite[Thm. 3.1]{CKLR19}. Hence, by \cite[Prop. 5.3]{DM04b}, $U$ is isomorphic to $V$. 
Now, let $G$ be finite. By  \cite[Corollary 4.8]{McR20} the linear semisimple full subcategory $\mathcal{C}_V$ of $\mathcal{C}$ generated by the irreducible  $V^G$-submodules of $V$ is a braided tensor category braided tensor equivalent to the symmetric tensor category $\mathrm{Rep}(G)$ of finite-dimensional 
representations of $G$. Moreover, the image of $V$ under this tensor equivalence is equivalent to the left-regular representation of $G$.   By 
\cite[Thm. 3.2 and Remark 3.7]{HKL15} $V$ and $U$ gives two rigid commutative haploid algebras in $\mathcal{C}_V$, see also 
\cite[Thm. 3.42]{CKM24}. It follows from  \cite[Thm. ]{KO02} that the regular representation of $G$ has, up to isomorphisms, a unique structure of rigid commutative haploid algebra in $\mathrm{Rep}(G)$, the one given by the multiplication of functions on $G$. It follows that $U$ and $V$ are equivalent haploid algebras in $\mathcal{C}_V$ and hence $U$ and $V$ are isomorphic VOAs, see e.g. the proof of \cite[Thm. 4.7]{CGGH23}.
\end{proof}

\section{The case of holomorphic VOAs with central charge 24}
\label{sec: c=24}
The classification program of holomorphic VOAs $V$ with $c=24$ and weight one subspace $V_1\neq \{ 0\}$ started with the influential work of Schellekens \cite{Sche93} and has recently been completed thanks to the work of various authors, see e.g.\ \cite{EMS20,ELMS21,LS19,MS23}. There are exactly 70 such VOAs  corresponding to entries 1--70 of the Schellekens list \cite{Sche93}. Each of these $70$ Schellekens VOAs is completely determined by the Lie algebra structure of its weight one subspace.  
The moonshine VOA is a holomorphic VOA with $V_1= \{0 \}$ and corresponds to entry 0 in the list. In \cite{FLM88} it has been conjectured that the moonshine VOA is the \emph{only} holomorphic VOA with $c=24$ and $V_1= \{0 \}$. A proof of this conjecture would complete the classification by showing that there are exactly 71 holomorphic VOAs with $c=24$ and that they are in one-to-one correspondence with entries 0--70 in \cite{Sche93}.

All the $70$ Schellekens VOAs have been shown to be unitary in \cite{CGGH23} and, independently, in \cite{Lam23}, cf. also \cite{Gau25} for the vertex operator superalgebra analogue of these results. Also the moonshine VOA is unitary, see \cite{CKLW18,DL14}. Accordingly, all the known holomorphic VOAs with $c=24$ are unitary. 
The holomorphic VOAs with $c<24$ are classified in \cite{DM04a}. They are the trivial vertex operator algebra $V=\C$ at $c=0$ and the lattice vertex operator algebra 
$V_{E_8}$ at $c=8$ and the lattice vertex operator algebra $V_{E_8 \times E_8}$ and $V_{\Gamma_{16}}$ at $c=16$ so they are all unitary by \cite{DL14}, see also \cite[Example 5.9]{CKLW18}. To sum up, all the holomorphic vertex operator algebras with $c\leq 24$ are unitary. 
Known examples of holomorphic VOAs with $c> 24$ are given by the lattice VOAs associated to positive definite even unimodular lattices of 
rank $c$, that are known to be unitary again by \cite{DL14},  and by taking tensor product of holomorphic VOAs with smaller central charge. Further examples are given, by possibly non-abelian orbifold constructions starting from holomorphic lattice VOAs \cite{GK19,GK21,MS23} and we expect all these VOAs to be unitary, cf. \cite{CGGH23,Gui24b,Lam23}. This motivates the following empirical conjecture. 

\begin{conjecture}
\label{ConjectureHolomorphicUnitary}
Every holomorphic VOA is unitary.
\end{conjecture}

The next theorem follows from the results by  Gaberdiel and Volpato \cite{GV09}, cf. Theorem \ref{th:V_{V_1}}, together with the classification of holomorphic VOAs with $c=24$ and $V_1 \neq \{0\}$.

\begin{theorem}
\label{th: Schellekens}
Let $V$ be a holomorphic VOA with $c=24$ and  $V_1 \neq \{0\}$. If $U$ is a simple unitary VOA with $\cZ_U = \cZ_V$ then 
$U$ is isomorphic to $V$.     
\end{theorem}
\begin{proof}
Since $\cZ_U = \cZ_V$ also $U$ has central charge $24$. Moreover, by Theorem \ref{th:V_{V_1}} $U_{U_1}$ is isomorphic to $V_{V_1}$. 

If we assume that $U$ is holomorphic, we can already conclude that $V\cong U$ because a holomorphic VOA of central charge 24 and $U_1\neq \{0\}$ is uniquely determined by the Lie algebra $U_1$. The rest of the proof is devoted to show that $U$ is holomorphic.
\medskip

If $V_1$ is an abelian Lie algebra then $V$ is isomorphic to the lattice VOA $V_\Lambda$ where lambda is the Leech lattice.  It follows that $U$ is isomorphic to a lattice VOA $V_L$ with $L$ a positive-definite even lattice unimodular of rank 24 isospectral to $\Lambda$, see Example \ref{ex: lattice}. Since $L$ is isospectral to $\Lambda$ then $L$ has no roots and hence it must be isomorphic to the Leech lattice which is, up to isomophisms, the unique positive-definite even lattice unimodular of rank 24
with no roots. It follows that $U$ is isomorphic to $V$.  If $V_1$ is semisimple then $U_{U_1}$ is strongly rational. It follows that also $U$ is strongly rational, see e.g. \cite[Sec. 3.5]{Gui22} and \cite[Corollary 4.11]{CGGH23}.  Actually, by   \cite[Corollary 4.11]{CGGH23} $U$ is also completely unitary in the sense of \cite[Definition 1.8]{Gui22}. It follows that every $U$-module is unitarizable and that the modular fusion category $\operatorname{Rep}(V)$ admits a natural unitary structure, see e.g. \cite{Gui22} and the references therein. Let $n$ be the finite number of equivalence classes of inequivalent irreducible $U$-modules and let $M^0,\dots M^{n-1}$ be representatives for each equivalence class with $M^0=U$ the adjoint module. If $h_i$ denotes the lowest energy of $M^i$ then, by \cite[Proposition 4.5]{CGGH23} $h_0 =0$ and $h_i > 0$ if $i>0$. 

For any $i=0, \dots, n-1$ we denote by $\chi_{M^i}(\tau)$ the character of $M^i$ that is 
$$ \chi_{M^i}(\tau) := e^{-2\pi i \tau}\Tr_{M^i} e^{2\pi i \tau (L^{M^i}_0 - 1_{M^i})} $$
which is convergent if $\Im \tau > 0$. Since the genus one partition function of $U$ coincide with the one of $V$ then $\chi_{M^0}(\tau)$ is a modular function of weight $0$ so that  
$$\chi_{M^0}\left(\frac{i}{t}\right) = \chi_{M^0}\left(it\right)$$ 
for all $t>0$. On the other hand, by \cite{Hua08} we have 
$$\chi_{M^k}\left(\frac{i}{t}\right) = \sum_{k=0}^{n-1}S_{k i}\chi_{M^i}\left(it \right) \,,$$
where $S$ is the unitary Verlinde matrix of the modular fusion category $\operatorname{Rep}(V)$ which is unitary and symmetric.
Hence,  
$$\chi_{M^0}\left(it\right) = \sum_{i=0}^{n-1}S_{0 i}\chi_{M^i}\left(it \right) \,.$$
Recalling that $h_0=0$ and $h_i>0$ if $i>0$ and taking $k=0$ we find $S_{0 0} =1$. 
Moreover, a standard argument, see e.g. \cite[Section 6.2.3]{Gan06}, shows that $S_{i0}>0$ for $i=0,\dots,n-1$

We have
$$
1 = S_{00} \leq \sum_{i=0}^{n-1} (S_{i0})^2 = (S^2)_{00}  =1\,.
$$
It follows that $n=1$ so that $U$ is holomorphic.  Since $U_1$ and $V_1$ are isomorphic Lie algebras the conclusion follows from the classification of holomorphic VOAs with $c=24$ and non-zero weight one subspace. 

\end{proof}

We now come to the moonshine VOA $V^\natural$. The latter is a holomorphic VOA with $c=24$ and zero weight one subspace and was constructed by 
Frenkel, Lepowsky and Meurman in \cite{FLM88} who also conjectured that $V^\natural$. is the unique holomorphic VOA with the above properties. 
It corresponds to entry $0$ in the Schellekens list \cite{Sche93} and, differently from the VOAs corresponding to the other 70 entries, its automorphism group is finite and equal to the Monster group $\mathbb{M}$, the largest among the 26 sporadic simple groups. The moonshine VOA is unitary and the Monster group acts by unitary automorphisms, \cite{CKLW18,DL14}. The automorphism group of Schellekens VOAs has recently been determined in \cite{BLS23}.

The monster orbifold ${V^\natural}^{\mathbb{M}}$ was studied by Harada and Lang in \cite{HL98}. They showed, in particular, that the dimension of the linear span of the (Virasoro) primary fields in ${V^\natural}^{\mathbb{M}}_k$ is equal to 0 for  $0<k < 12$ and for $k = 13, 14, 15, 17, 19, 21, 23$ and is equal to 1
for $k=12, 16, 18, 20,  22$. In particular ${V^\natural}^{\mathbb{M}}_k = L(24,0)_k$ for $k \leq 11$ and ${V^\natural}^{\mathbb{M}}_k \neq  L(24,0)_k$ for 
$k \geq 12$. Here $L(24,0)$ denotes the unitary subalgebra of $V^\natural$ generated by the conformal vector, cf. Example \ref{ex:Virasoro}. By the results \cite{DM00} and \cite{Hur03}, for any $k =  12, 16, 18, 20, 22$ there is a primary vector  $a \in   {V^\natural}_k$ 
with $\Tr a_0q^{L_0} \neq 0$. Hence, as a consequence of Corollary \ref{cor: projection partition} we have the following proposition.

\begin{proposition} 
\label{prop: monster orbifold} $PV^\natural_k =  {V^\natural}^{\mathbb{M}}_k$ for every integer $k < 24$. In particular  $PV^\natural \neq L(24,0)$. 
\end{proposition}
\begin{proof} By Proposition \ref{prop: unitarymodule} ${V^\natural}^{\mathbb{M}}$ is a direct sum of submodules $M^n$, $n\in \Z_{\geq 0}$
where each $M^n$ is a multiple $m(n)L(24,n)$ of the $L(24,0)$-module $L(24,n)$ with lowest energy $n$. The multiplicity $m(n)$ is equal to the dimension of the subspace of primary fields in ${V^\natural}^{\mathbb{M}}_n$.  
Similarly ${PV^\natural}$ is a direct sum of submodules $W^n$, $n\in \Z_{\geq 0}$ where each $W^n$ is a multiple $k(n)L(24,n)$ of the $L(24,n)$ and the multiplicity $k(n)$ is equal to the dimension of the subspace of primary fields in $PV^\natural_n$.  Clearly $W^n \subset M^n$ and $k(n) \leq m(n)$ for all 
$n\in \Z_{\geq 0}$. As we discussed above, we must have $k(n) = m(n)$ and hence $W^n=M^n$ for $n\leq 23$ Hence, if $k\leq 23$ we have 
\begin{eqnarray*}
PV^\natural_k &=&  V^\natural_k \cap  \bigoplus_{n\leq k}  W^n = V^\natural_k \cap  \bigoplus_{n\leq k}  M^n \\
&=&  {V^\natural}^{\mathbb{M}}_k  \,.
\end{eqnarray*}

\end{proof}

\begin{conjecture}
\label{conj: monster orbifold}
 $PV^\natural=  {V^\natural}^{\mathbb{M}}$. 
\end{conjecture} 

\begin{theorem}
\label{thm: monster orbifold}
Assume Conjectures \ref{conj:slope} and \ref{conj: monster orbifold}. Then, up isomorphisms,  $V^\natural$ is the unique strongly local holomorphic VOA with central charge $24$ and zero weight one subspace. Moreover, if the monster orbifold ${V^\natural}^{\mathbb{M}}$ is strongly rational then, 
up isomorphisms,  $V^\natural$ is the unique unitary holomorphic VOA with central charge $24$ and zero weight one subspace. 
\end{theorem}

\begin{proof}
The first part of the statement follows directly from Theorem \ref{thm: strongly local compact orbifold} while the second follows from Theorem 
\ref{thm VOA compact orbifold}, 
\end{proof}

\end{document}